%% file: main.tex
\documentclass{amsart}
\usepackage{graphicx} 
\usepackage{amsthm, amsmath, amssymb, tikz, bm}
\usepackage{mathtools}
\usepackage{xifthen}
\usepackage{comment}
\usepackage[margin=4.5cm]{geometry}
\usepackage[shortlabels]{enumitem}
\usepackage[colorlinks=true,
linkcolor=blue,citecolor=blue,
urlcolor=blue]{hyperref}

\usetikzlibrary{calc,shapes, backgrounds}
\usetikzlibrary[patterns]

\newtheorem{theorem}{Theorem}[section]
\newtheorem{lemma}[theorem]{Lemma}
\newtheorem{sublemma}{}[theorem]
\newtheorem{corollary}[theorem]{Corollary}

\newtheorem{proposition}[theorem]{Proposition}

\newcommand{\romannum}[1]{\romannumeral#1\relax}

\makeatletter
\newdimen\p@renwd \setbox0=\hbox{\phantom B} \p@renwd=\wd0
\def\bordersquare#1{\begingroup \m@th
  \setbox0=\vbox{\def\cr{\crcr\noalign{\kern2pt\global\let\cr=\endline}}
      \ialign{$##$\hfil\kern2pt\kern\p@renwd&\thinspace\hfil$##$\hfil
        &&\quad\hfil$##$\hfil\crcr
        \omit\strut\hfil\crcr\noalign{\kern-\baselineskip}
        #1\crcr\omit\strut\cr}}
  \setbox2=\vbox{\unvcopy0 \global\setbox1=\lastbox}
  \setbox2=\hbox{\unhbox1 \unskip \global\setbox1=\lastbox}
  \setbox2=\hbox{$\kern\wd1\kern-\p@renwd \left[ \kern-\wd1
    \global\setbox1=\vbox{\box1\kern2pt}
    \vcenter{\kern-\ht1 \unvbox0 \kern-\baselineskip} \,\right]$}
  \null\;\vbox{\kern\ht1\box2}\endgroup}
\makeatother

\newcommand{\cC}{\mathcal{C}}

\newcommand{\cI}{\mathcal{I}}

\newcommand{\cP}{\mathcal{P}}
\newcommand{\cQ}{\mathcal{Q}}
\newcommand{\cR}{\mathcal{R}}

\newcommand{\cT}{\mathcal{T}}
\newcommand{\cS}{\mathcal{S}}

\newcommand{\si}{\operatorname{si}}
\newcommand{\cl}{\operatorname{cl}}
\newcommand{\GF}{\operatorname{GF}}
\newcommand{\AG}{\operatorname{AG}}
\newcommand{\PG}{\operatorname{PG}}
\newcommand{\dist}{\operatorname{dist}}

\makeatletter
\newcommand*{\rom}[1]{\expandafter\@slowromancap\romannumeral #1@}

\makeatother

\title{Unavoidable flats in connected regular matroids}
\author{Wayne Ge}
\address{Mathematics Department \\
Louisiana State University \\ Baton Rouge, LA}
\email{yge4@lsu.edu}

\date{\today}
\subjclass[2020]{Primary 05B35; Secondary 05C40}
\keywords{Unavoidable flats, regular matroids, connected matroids}

\begin{document}

\begin{abstract}
This paper proves that every $2$-connected regular matroid of sufficiently large rank has a $2$-connected graphic flat of large rank. Furthermore, we explicitly determine a list of such unavoidable flats.
\end{abstract}
\maketitle

\section{Introduction}

A classical theme in combinatorics is that sufficiently large objects contain large highly structured subobjects. Ramsey's Theorem~\cite{Ramsey1930} implies that every sufficiently large graph contains either a large clique or a large stable set. Moreover, every sufficiently large connected graph contains, as an induced subgraph, a large clique, a large star, or a long path~\cite[Proposition~9.4.1]{Diestel}.

Beyond these Ramsey-type results, induced subgraphs are central objects in structural graph theory. For matroids, flats play a role analogous to that of induced subgraphs in graphs. Recently, Geelen and Kroeker~\cite[Theorem~1.2]{unavoidable_flats} proved the following matroidal analogue of Ramsey's theorem for matroids representable over a fixed prime field.

\begin{theorem}
\label{thm:Geelen-Kroeker}
    For every prime $p$ and positive integer $k$, there is an integer $N_p(k)$ such that every simple $\GF(p)$-representable matroid $M$ with $r(M)\geq N_p(k)$ has a rank-$k$ flat $F$ such that $M|F$ is independent, isomorphic to $\AG(k-1,p)$, or isomorphic to $\PG(k-1,p)$.
\end{theorem}

It is natural to try to determine unavoidable highly structured $2$-connected flats in large $2$-connected matroids in particular classes. In this paper, we prove such a result for the class of regular matroids.

For regular matroids, an earlier counterpart of Theorem~\ref{thm:Geelen-Kroeker} follows from a result of Bonnice and Edelstein~\cite{Bonnice-Edelstein}.

\begin{theorem}
    For every integer $k\geq 2$, every simple rank-$(2k-1)$ real-representable matroid contains an independent rank-$k$ flat.
\end{theorem}

Our main result, which is stated next, shows that, under the additional hypothesis of $2$-connected\-ness, one can force more structure than an independent flat. When we say that a flat $F$ of $M$ is graphic or $2$-connected, we mean that the restriction $M|F$ has the corresponding property.

\begin{theorem}\label{thm:large_graphic_flat}
    For every positive integer $k$, there is an integer $f_{\ref{thm:large_graphic_flat}}(k)$ such that every $2$-connected regular matroid $M$ of rank at least $f_{\ref{thm:large_graphic_flat}}(k)$ has a $2$-connected graphic flat of rank at least $k$.
\end{theorem}

This theorem implies that every $2$-connected regular matroid of sufficiently large rank has a flat whose restriction is the cycle matroid of a large $2$-connected graph. Thus the problem of finding unavoidable $2$-connected flats in regular matroids is reduced to the problem of finding unavoidable $2$-connected induced subgraphs in large $2$-connected graphs. The latter problem was solved by Allred, Ding, and Oporowski~\cite{SDO2020}. By combining their theorem with Theorem~\ref{thm:large_graphic_flat}, we obtain the following explicit characterization. This theorem involves clean ladders, which are graphs that admit a path decomposition in which each vertex is labeled by a copy of $K_4$, $K_3$, or the planar dual of $K_3$. The formal definition of clean ladders will be given in Section~\ref{sec: preliminaries} while, in Section~\ref{sec: distance in matroids}, we explain why, in $2$-connected graphs and regular matroids, clean ladders play a role analogous to that of induced paths in connected graphs.

\begin{theorem}\label{thm: explicit characterization}
    For every positive integer $k$, there is an integer
    $f_{\ref{thm: explicit characterization}}(k)$ such that every simple $2$-connected regular matroid $M$ of rank at least
    $f_{\ref{thm: explicit characterization}}(k)$ has a $2$-connected flat $F$ such that $M|F=M(H)$ for some graph $H$. Moreover, $H$ is one of the following:
    \begin{enumerate}
        \item[(\romannum{1})] $K_{k+1}$;
        \item[(\romannum{2})] a subdivision of $K_{2,k+1}$;
        \item[(\romannum{3})] a graph obtained from a subdivision of $K_{2,k+1}$ by adding an edge joining the two vertices of degree $k+1$; or
        \item[(\romannum{4})] a clean ladder with at least $k+1$ vertices.
    \end{enumerate}
\end{theorem}

To prove Theorem~\ref{thm:large_graphic_flat}, we introduce a notion of distance in matroids. Let $e$ and $f$ be distinct elements of a $2$-connected matroid $M$. The {\it distance} between $e$ and $f$ in $M$ is the size of a smallest circuit containing both $e$ and $f$. In Section~\ref{sec: distance in matroids}, we show that, in a regular matroid, a pair of elements at large distance forces the existence of a large-rank flat whose restriction is the cycle matroid of a clean ladder.

We combine these distance arguments with several standard decomposition theorems, including the canonical tree decomposition of Cunningham and Edmonds~\cite{CE1980} and Seymour's decomposition theorem for regular matroids~\cite{Seymour1980}. The canonical tree decomposition is well suited to our distance arguments. A long path in this decomposition forces large distance, while the absence of such a path forces either a large-rank $3$-connected vertex label or a high-degree vertex from which a large-rank graphic flat can be constructed. By contrast, controlling distance through Seymour's decomposition is much more delicate, since the triangles involved in different $3$-sums may interact in a less controlled way. We explain this issue in Section~\ref{sec: 3-con cases} and develop some new tools to overcome this significant obstacle.

Although Seymour's decomposition theorem reduces regular matroids to graphic matroids, cographic matroids, and copies of $R_{10}$ via $1$-, $2$-, and $3$-sums, Theorem~\ref{thm:large_graphic_flat} is not a routine consequence of the graphic and cographic cases. The obstruction is that flats, unlike minors, do not pass cleanly through $3$-sums. A flat of a $3$-sum may meet both sides of the decomposition, which is especially delicate because we seek flats that are both graphic and $2$-connected. We address this by developing tools to control flats across $3$-sums. In Section~\ref{sec: downgrade 3-sums}, we prove an enlarged-wheel replacement theorem, which permits one side of a $3$-sum to be replaced by a controlled graphic flat while retaining the common triangle. This replacement is a key ingredient in constructing large $2$-connected graphic flats across Seymour decompositions.
\section{Preliminaries}\label{sec: preliminaries}

We assume familiarity with the basic notions of matroid theory as in, for example,~\cite{Oxl11}. All matroids considered in this paper are finite and may have loops or parallel elements. For graph-theoretic terminology not explicitly defined in this paper, we refer to Bondy and Murty~\cite{Bondy-Murty}.

Our main tools are decomposition theorems that allow us to decompose large regular matroids into smaller pieces and to simplify the structure within each piece. Before introducing these decomposition theorems, we recall the definitions of $t$-sums for $t\in\{1,2,3\}$, as they are essential for understanding these results.

\subsection{$1$-, $2$-, $3$-sum and quasi-$3$-sum}
This subsection follows~\cite{Seymour1980} and~\cite[Section~9.3]{Oxl11}. To define $t$-sums for $t\in\{1,2,3\}$, we restrict to binary matroids. For sets $A$ and $B$, the {\it symmetric difference} of $A$ and $B$, denoted $A\triangle B$, is $(A-B)\cup(B-A)$. Let $M_1$ and $M_2$ be binary matroids.

\begin{enumerate}
\item[(\romannum{1})] If $E(M_1)\cap E(M_2)=\emptyset$, then the {\it $1$-sum} of $M_1$ and $M_2$ is the direct sum $M_1\oplus M_2$.

\item[(\romannum{2})] If $E(M_1)\cap E(M_2)=\{p\}$, $\min\{|E(M_1)|,|E(M_2)|\}\geq 2$, and neither $M_1$ nor $M_2$ has $p$ as a loop or a coloop, then the {\it $2$-sum} of $M_1$ and $M_2$, denoted $M_1\oplus_2 M_2$, is the matroid with ground set $E(M_1)\triangle E(M_2)$ whose set of circuits is
\[
\cC(M_1\backslash p)\cup\cC(M_2\backslash p)\cup\{C\triangle D:\; p\in C\in\cC(M_1)\text{ and }p\in D\in\cC(M_2)\}.
\]

\item[(\romannum{3})] If $E(M_1)\cap E(M_2)=T=\{a,b,c\}$, $\min\{|E(M_1)|,|E(M_2)|\}\geq 7$, and $T$ is a coindependent triangle in both $M_1$ and $M_2$, then the {\it $3$-sum} of $M_1$ and $M_2$, denoted $M_1\oplus_3 M_2$, is the matroid with ground set $E(M_1)\triangle E(M_2)$ whose circuits are the members of $\cC(M_1\backslash T)\cup \cC(M_2\backslash T)$ together with the minimal sets of the form $C_1\triangle C_2$, where $C_i\in\cC(M_i)$, $C_1\cap T=C_2\cap T$, and this common intersection has size one (see~\cite[(2.7)]{Seymour1980} and~\cite[Lemma~9.3.3]{Oxl11}).
\end{enumerate}

We note that, in Seymour's definition~\cite{Seymour1980}, $2$-sum is defined only for matroids with at least three elements. For technical reasons, we allow $U_{1,2}$ to occur as one part of a $2$-sum, consistent with~\cite[p.~260]{Oxl11}. In the rest of the paper, when we refer to the $1$-, $2$-, or $3$-sum of two binary matroids $M_1$ and $M_2$, we implicitly assume that the corresponding conditions in (\romannum{1})--(\romannum{3}) are satisfied. We omit the straightforward proof of the following result, which gives the rank of a $3$-sum; see, for example,~\cite[Proposition 9.3.4]{Oxl11}.

\begin{proposition}\label{prop: rank of 3-sum}
Let $M_1$ and $M_2$ be binary matroids. Then
\[
r(M_1\oplus_3 M_2)=r(M_1)+r(M_2)-2.
\]
\end{proposition}

Since we seek unavoidable $2$-connected flats in this paper, such flats often meet both sides of a $2$- or $3$-sum. For $2$-sums, the following elementary proposition is well suited to our purposes.

\begin{proposition}\label{prop: 2-sum of two flats}
    Let $M$ be the $2$-sum $M_1\oplus_2 M_2$ with basepoint $p$. For each $i\in\{1,2\}$, let $F_i$ be a $2$-connected flat of $M_i$ such that $p\in F_i$ and $|F_i|\geq 2$. Then $(F_1\cup F_2)-\{p\}$ is a $2$-connected flat $F$ of $M$. Moreover, $M|F=(M_1|F_1)\oplus_2 (M_2|F_2)$.
\end{proposition}

For $3$-sums, the situation is more complicated. We therefore introduce quasi-$3$-sums. Let $M_1$ and $M_2$ be binary matroids such that $E(M_1)\cap E(M_2)=T$, where $T$ is a triangle of both $M_1$ and $M_2$. The {\it generalized parallel connection} $P_T(M_1,M_2)$ of $M_1$ and $M_2$ across $T$ is the matroid on $E(M_1)\cup E(M_2)$ whose flats are those subsets $X$ of $E(M_1)\cup E(M_2)$ such that $X\cap E(M_i)$ is a flat of $M_i$ for each $i\in\{1,2\}$. The {\it quasi-$3$-sum} of $M_1$ and $M_2$, denoted by $M_1\oplus_3' M_2$, is the matroid $P_T(M_1,M_2)\backslash T$. Every $3$-sum is a quasi-$3$-sum, but the converse does not hold, since a $3$-sum imposes additional conditions on the common triangle $T$ and on the sizes of $M_1$ and $M_2$.

Let $G$ and $H$ be graphs such that $V(G)\cap V(H)=\{x,y,z\}$, $E(G)\cap E(H)=\{e,f,g\}$, and $\{e,f,g\}$ is a triangle with vertex set $\{x,y,z\}$. The {\it quasi-$3$-sum} of $G$ and $H$, denoted by $G\oplus_3' H$, is the graph $(G\cup H)\backslash \{e,f,g\}$. The following proposition is immediate. It tells us that the quasi-$3$-sum of two graphic matroids is graphic.

\begin{proposition}\label{prop: graph of quasi-3-sum}
Let $G$ and $H$ be graphs such that $V(G)\cap V(H)=\{x,y,z\}$, $E(G)\cap E(H)=\{e,f,g\}$, and $\{e,f,g\}$ is a triangle with vertex set $\{x,y,z\}$. Then
\[
M(G)\oplus_3' M(H)=M(G\oplus_3' H).
\]
\end{proposition}

The motivation for introducing quasi-$3$-sums is as follows. Let $M$ be the $3$-sum of $M_1$ and $M_2$ across a common triangle $T$. Suppose that we want to find a flat of $M$ that meets both sides of the decomposition. A natural approach is to choose flats $F_1$ of $M_1$ and $F_2$ of $M_2$, and then consider their union with the elements of $T$ deleted. However, if both $F_1$ and $F_2$ contain the triangle $T$ and we wish to describe the structure of the resulting flat of $M$, the ordinary $3$-sum of $M_1|F_1$ and $M_2|F_2$ need not be well defined. The following result follows immediately from the definition of generalized parallel connection.

\begin{lemma}\label{lem: flats in 3-sum}
    Let $M$ be the quasi-$3$-sum of $M_1$ and $M_2$ across a common triangle $T$. Let
    $F_1$ and $F_2$ be flats of $M_1$ and $M_2$, respectively. Then
    \[
        F=(F_1\cup F_2)-T
    \]
    is a flat of $M$. Moreover,
    \begin{enumerate}
        \item[(\romannum{1})] if $F_1\cap T=F_2\cap T=\emptyset$, then
        \[
            M|F=M_1|F_1\oplus M_2|F_2;
        \]
        \item[(\romannum{2})] if $F_1\cap T=F_2\cap T=\{p\}$ and both
        $F_1-F_2$ and $F_2-F_1$ are nonempty, then
        \[
            M|F=M_1|F_1\oplus_2 M_2|F_2;
        \]
        \item[(\romannum{3})] if $F_1\cap F_2=T$, then
        \[
            M|F=M_1|F_1\oplus_3' M_2|F_2.
        \]
    \end{enumerate}
\end{lemma}

\subsection{Decomposition theorems}
Throughout the paper, we use several decomposition theorems to analyze regular matroids. In this subsection, we state these theorems and fix the terminology used later.

A \emph{matroid-labeled tree} is a tree $\cT$ with vertex set $\{M_1,M_2,\dots,M_k\}$ for some positive integer $k$, satisfying the following conditions for all distinct $i,j\in\{1,2,\dots,k\}$:
\begin{enumerate}
    \item[(\romannum{1})] each $M_i$ is a matroid with at least three elements;
    \item[(\romannum{2})] if $M_i$ and $M_j$ are joined by an edge of $\cT$, then that edge corresponds to a $2$-sum or a $3$-sum of $M_i$ and $M_j$; and
    \item[(\romannum{3})] if $M_i$ and $M_j$ are not adjacent in $\cT$, then $E(M_i)\cap E(M_j)=\emptyset$.
\end{enumerate}
We call the matroids $M_1,M_2,\dots,M_k$ the {\it vertex labels} of $\cT$. This definition differs slightly from the one in~\cite[Section~8.3]{Oxl11}, since we allow edges of a decomposition tree to represent $3$-sums. This convention makes the terminology compatible with Seymour's decomposition theorem for regular matroids and its variants.

We say that a matroid-labeled tree $\cT$ is a {\it tree decomposition} of a matroid $M$ if $M$ can be constructed from the vertex labels of $\cT$ by performing the $2$- and $3$-sums corresponding to the edges of $\cT$. A {\it path decomposition} of $M$ is a tree decomposition whose underlying tree is a path.

When $G$ is a graph, a tree decomposition of $G$ is a graph-labeled tree defined analogously. We do not repeat the definition here, and refer the reader to~\cite[p.~3]{cc-minors} for the graph-theoretic version.

Cunningham and Edmonds~\cite{CE1980} proved that every $2$-connected matroid $M$ has a tree decomposition satisfying the conditions in the following theorem. This tree decomposition is called the {\it canonical tree decomposition} of $M$.

\begin{theorem}
    Let $M$ be a $2$-connected matroid with at least three elements. Then $M$ has a tree decomposition $\cT$ such that
    \begin{enumerate}
        \item[(\romannum{1})] every vertex label is $3$-connected, a circuit, or a cocircuit;
        \item[(\romannum{2})] every edge corresponds to a $2$-sum of its adjacent vertex labels; and
        \item[(\romannum{3})] no two adjacent vertices are both labeled by circuits or are both labeled by cocircuits.
    \end{enumerate}
    Moreover, $\cT$ is unique up to relabeling of the basepoints used in the $2$-sums.
\end{theorem}

We also use Seymour's decomposition theorem for regular matroids~\cite{Seymour1980}. The matroid $R_{10}$ is a ten-element rank-$5$ regular matroid that is neither graphic nor cographic. It is represented over $\GF(2)$ by the ten column vectors of length five that have exactly three ones per column.

\begin{theorem}\label{thm: Seymour decomp}
    A matroid is regular if and only if it can be constructed from graphic matroids, cographic matroids, and copies of $R_{10}$ by a sequence of $1$-, $2$-, and $3$-sums.
\end{theorem}

When a regular matroid is $3$-connected, we use the following refinement due to Aprile and Fiorini~\cite[Theorem~6]{Aprile-Fiorini}.

\begin{theorem}\label{thm: Aprile and Fiorini}
Let $M$ be a $3$-connected regular matroid other than $R_{10}$. Then there is a matroid-labeled tree $\cT$ such that
\begin{enumerate}
    \item[(\romannum{1})] each vertex of $\cT$ is labeled by a graphic or cographic matroid;
    \item[(\romannum{2})] each edge of $\cT$ corresponds to a $3$-sum of the labels of its endvertices; and
    \item[(\romannum{3})] $M$ is obtained by performing the $3$-sums corresponding to the edges of $\cT$ in any order.
\end{enumerate}
\end{theorem}

A tree $\cT$ satisfying the conditions of Theorem~\ref{thm: Aprile and Fiorini} is called a {\it $\triangle$-tree decomposition} of $M$, or simply a {\it $\triangle$-tree}. Unlike the canonical tree decomposition, a $\triangle$-tree decomposition need not be unique. Aprile and Fiorini proved a stronger form of their $\triangle$-tree decomposition theorem. In particular, they showed that one may assume that, whenever a triangle of a nonplanar cographic matroid $M^*(G)$ is used in a $3$-sum, that triangle corresponds to a vertex bond of $G$; see~\cite[Proposition 27]{Aprile-Fiorini} and also~\cite[Theorem 2.5]{BMS2026}. We do not need this additional assumption.

A matroid $N$ is {\it almost $3$-connected} if its simplification $\si(N)$ is $3$-connected. The following result follows easily from a result of Seymour~\cite[(4.3)]{Seymour1980}; see also~\cite[Lemma~30]{Aprile-Fiorini}.

\begin{proposition}\label{prop: almost 3-con}
    If $M$ is a $3$-connected matroid and $\cT$ is a tree decomposition of $M$, then every edge of $\cT$ corresponds to a $3$-sum and every vertex label of $\cT$ is almost $3$-connected.
\end{proposition}

\begin{proposition}\label{prop: vertex label rank at least three}
    If $M$ is a $3$-connected regular matroid of rank at least three and $M$ is not isomorphic to $R_{10}$, then $M$ has a $\triangle$-tree in which every vertex label has rank at least three.
\end{proposition}

\begin{proof}
    By Theorem~\ref{thm: Aprile and Fiorini}, $M$ has a $\triangle$-tree $\cT$. If $\cT$ has only one vertex, then its unique vertex label is $M$, which has rank at least three. We may therefore assume that $\cT$ has at least two vertices. Since every edge of $\cT$ corresponds to a $3$-sum, every vertex label contains a triangle used in a $3$-sum and hence has rank at least two.

    Suppose that $\cT$ has a vertex labeled by a rank-$2$ matroid $M_v$. By Proposition~\ref{prop: almost 3-con}, the matroid $M_v$ is almost $3$-connected, so $\si(M_v)$ is a triangle. Since $r(M)\geq 3$, the vertex $M_v$ has a neighbor $M_u$. The $3$-sum of a graphic or cographic matroid with a matroid whose simplification is a triangle is again graphic or cographic. Thus performing the $3$-sum corresponding to the edge joining $M_v$ and $M_u$ produces a new $\triangle$-tree of $M$ in which the number of vertices that are labeled by a rank-$2$ matroid is reduced by one.

    Repeating this operation gives a $\triangle$-tree of $M$ with no rank-$2$ vertex label. Hence every vertex label in the resulting $\triangle$-tree has rank at least three.
\end{proof}

\subsection{Graphic matroids and cographic matroids}

As Seymour's decomposition theorem shows, two fundamental classes of building blocks for regular matroids are graphic and cographic matroids. The following well-known theorem of Tutte~\cite{Tutte1959} gives an excluded-minor characterization of graphic matroids.

\begin{theorem}\label{thm: graphic regular matroids}
    A regular matroid $M$ is graphic if and only if $M$ has no minor isomorphic to $M^*(K_5)$ or $M^*(K_{3,3})$.
\end{theorem}

It is straightforward to check that the classes of graphic and cographic matroids are both closed under direct sums and $2$-sums. Moreover, if $M_1$ and $M_2$ are graphic matroids, then $M_1\oplus_3 M_2$ is also graphic. In contrast, the $3$-sum of two cographic matroids need not be cographic. For example, the non-cographic matroid $R_{12}$ is isomorphic to a $3$-sum of $M^*(K_{3,3})$ and $M(K_5\backslash e)$; see~\cite[p.~490]{Oxl11}.

\subsubsection{Unavoidable flats of graphic matroids}

For graphs, induced subgraphs play an analogous role to flats in matroids. For graphic matroids, the following result is straightforward (see, for example,~\cite[Proposition~4.1.7]{Oxl11}).

\begin{proposition}
    Let $G$ be a $2$-connected loopless graph, and let $M=M(G)$. If $H$ is a $2$-connected induced subgraph of $G$ with at least three vertices, then $M$ has a $2$-connected flat $F$ such that $M|F=M(H)$.
\end{proposition}

This means that finding large-rank $2$-connected flats in a $2$-connected graphic matroid $M(G)$ is equivalent to finding large $2$-connected induced subgraphs in $G$. The following theorem of Allred, Ding, and Oporowski~\cite{SDO2020} is well suited to our purpose (see also~\cite[Theorem~1.4]{cc-minors}). A {\it bond graph} $B_n$ consists of two vertices joined by $n$ parallel edges.

\begin{theorem}\label{thm:unavoidable induced subgraph}
Let $k\geq 3$ be an integer. Then there is an integer $f(k)$ such that every simple $2$-connected graph with at least $f(k)$ vertices has an induced subgraph isomorphic to one of the following:
\begin{enumerate}
    \item[(\romannum{1})] $K_k$;
    \item[(\romannum{2})] a subdivision of $K_{2,k}$;
    \item[(\romannum{3})] a graph obtained from a subdivision of $K_{2,k}$ by adding an edge joining the two vertices of degree $k$; or
    \item[(\romannum{4})] a graph on $k$ vertices admitting a path decomposition $\cP$ such that every edge of $\cP$ corresponds to a $2$-sum and each vertex of $\cP$ is labeled by
    \begin{itemize}
        \item[(a)] a copy of $K_4$, with the two basepoints forming a matching whenever the vertex is internal; or
        \item[(b)] a copy of $K_3$ or $B_3$,
    \end{itemize}
    where neither end of $\cP$ is labeled by $B_3$, and no two consecutive vertices of $\cP$ are labeled by $B_3$.
\end{enumerate}
\end{theorem}

\begin{figure}[htb]
    \centering
    \resizebox{11cm}{!}{\input{figures/fig_clean_laddar}}%
    \caption{An example of a clean ladder.}
    \label{fig: clean ladder}
\end{figure}
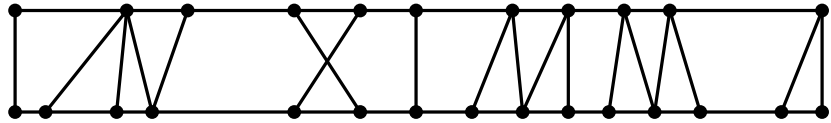

We call the graphs described in (\romannum{4}) of Theorem~\ref{thm:unavoidable induced subgraph} {\it clean ladders}. Figure~\ref{fig: clean ladder} shows an example of a clean ladder. We shall use the following immediate consequence of Theorem~\ref{thm:unavoidable induced subgraph}.

\begin{corollary}\label{cor: graphic case}
    For every positive integer $k$, there is an integer $f_{\ref{cor: graphic case}}(k)$ such that every $2$-connected graphic matroid $M$ of rank at least $f_{\ref{cor: graphic case}}(k)$ has a $2$-connected graphic flat $F$ such that $M|F=M(H)$ for some graph $H$. Moreover, the simplification of $H$ is one of the following:
    \begin{enumerate}
        \item[(\romannum{1})] $K_{k+1}$;
        \item[(\romannum{2})] a subdivision of $K_{2,k+1}$;
        \item[(\romannum{3})] a graph obtained from a subdivision of $K_{2,k+1}$ by adding an edge joining the two vertices of degree $k+1$; or
        \item[(\romannum{4})] a $(k+1)$-vertex clean ladder.
    \end{enumerate}
\end{corollary}

\subsubsection{Unavoidable flats of cographic matroids}

For cographic matroids, cc-minors were introduced in~\cite{cc-minors} as a tool for analyzing unavoidable flats. A {\it cycle-contraction minor}, or {\it cc-minor}, of a graph $G$ is a minor obtained from $G$ by repeatedly contracting cycles. Similarly, a {\it circuit-contraction minor}, or {\it cc-minor}, of a matroid $M$ is a minor obtained from $M$ by repeatedly contracting circuits. Taking cc-minors is dual to restricting to flats.

\begin{proposition}{\cite[Proposition~1.6]{cc-minors}}\label{prop: cc-minor_flat}
    Let $M$ be a matroid and let $F\subseteq E(M)$. Then $F$ is a flat of $M$ if and only if $(M|F)^*$ is a cc-minor of $M^*$.
\end{proposition}

By this proposition, finding unavoidable $2$-connected flats in a large cographic matroid $M^*(G)$ is equivalent to finding large $2$-connected cc-minors of $G$. An unavoidable-flats characterization for $2$-connected cographic matroids was proved in~\cite[Theorem~1.5]{cc-minors}. In this paper, we use the following theorem from~\cite{cc-minors}. A {\it fan graph} $F_n$ is obtained from an $n$-vertex path $v_1v_2\dots v_n$ by adding a new vertex $u$ adjacent to each $v_i$. For positive integers $t_1,t_2,\dots,t_n$, a {\it fan-type graph} $F_{t_1,t_2,\dots,t_n}$ is obtained from $F_n$ by replacing each edge $uv_i$ by $t_i$ parallel edges. Figure~\ref{fig: fans.} shows three examples of fan-type graphs. The bond graph $B_n$ is a fan-type graph, corresponding to a $1$-vertex path.

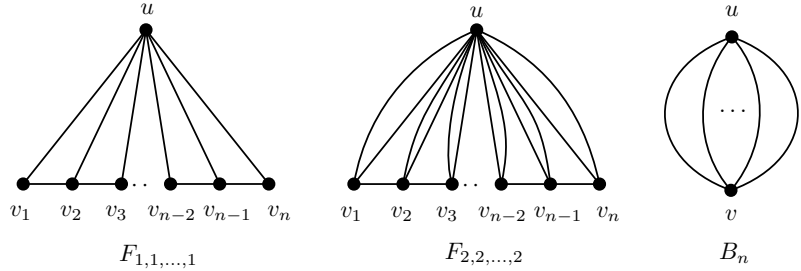
\begin{figure}[htb]
    \centering
    \resizebox{11cm}{!}{\input{figures/fig_fan}}%
    \caption{Fan-type graphs.}
    \label{fig: fans.}
\end{figure}

\begin{theorem}{\cite[Theorem~7.1]{cc-minors}}\label{thm: large_3_con(cc-minors)}
    For every integer $t\geq 3$, there is an integer $f_{\ref{thm: large_3_con(cc-minors)}}(t)$ such that if a simple $3$-connected graph $G$ has more than $f_{\ref{thm: large_3_con(cc-minors)}}(t)$ edges, then $G$ has a fan-type graph $F_{t_1,t_2,\dots,t_n}$ as a cc-minor such that $\sum_{i=1}^n t_i\geq t$.
\end{theorem}

For the purposes of this paper, we shall use the following consequence.

\begin{corollary}\label{cor: large cc-minor}
    For every integer $k\geq 2$, there is an integer $f_{\ref{cor: large cc-minor}}(k)$ such that if $M$ is a $3$-connected cographic matroid with $r(M)\geq f_{\ref{cor: large cc-minor}}(k)$, then $M$ has a flat $F$ such that $M|F= M(L)$, where $L$ is a clean ladder with at least $k+1$ vertices.
\end{corollary}

\begin{proof}
    We show that one may take
    \[
        f_{\ref{cor: large cc-minor}}(k)=\max\{3,f_{\ref{thm: large_3_con(cc-minors)}}(k+1)\}.
    \]
    Suppose that $M$ is a $3$-connected cographic matroid with
    $r(M)\geq f_{\ref{cor: large cc-minor}}(k)$. Since $r(M)\geq 3$, there is a $3$-connected graph $G$ such that $M\cong M^*(G)$. Moreover,
    \[
        |E(G)|=|E(M)|\geq r(M)\geq f_{\ref{thm: large_3_con(cc-minors)}}(k+1).
    \]
    By Theorem~\ref{thm: large_3_con(cc-minors)}, the graph $G$ has a cc-minor $H$ such that
    $H\cong F_{t_1,t_2,\dots,t_n}$, where $\sum_{i=1}^n t_i\geq k+1$. Thus $M(H)$ is a cc-minor of $M(G)$. By Proposition~\ref{prop: cc-minor_flat}, the matroid $M$ has a flat $F$ such that $M|F= M^*(H)$.
    
    It is straightforward to verify that the graph $H$ is planar, and its dual is a clean ladder with at least $k+1$ vertices. Hence $M^*(H)= M(H^*)= M(L)$ for some clean ladder $L$ with at least $k+1$ vertices. Therefore $M|F= M(L)$, as required.
\end{proof}

We conclude this section by combining Corollaries~\ref{cor: graphic case} and~\ref{cor: large cc-minor}.

\begin{corollary}\label{cor: graphic and cographic in one}
    For every positive integer $k$, there is an integer $f_{\ref{cor: graphic and cographic in one}}(k)$ such that every $3$-connected graphic or cographic matroid $M$ of rank at least $f_{\ref{cor: graphic and cographic in one}}(k)$ has a $2$-connected graphic flat $F$ such that $M|F\cong M(H)$ for some graph $H$. Moreover, $H$ is one of the following:
    \begin{enumerate}
        \item[(\romannum{1})] $K_{k+1}$;
        \item[(\romannum{2})] a subdivision of $K_{2,k+1}$;
        \item[(\romannum{3})] a graph obtained from a subdivision of $K_{2,k+1}$ by adding an edge joining the two vertices of degree $k+1$; or
        \item[(\romannum{4})] a clean ladder with at least $k+1$ vertices.
    \end{enumerate}
\end{corollary}
\section{Distance in matroids}\label{sec: distance in matroids}
A useful concept in the proof of Ramsey's theorem for connected graphs is the distance between two vertices; see, for example,~\cite[p.~301]{Diestel}. If two vertices of a simple connected graph $G$ are far apart, then a shortest path between them is a long induced path, one of the unavoidable induced subgraphs. In this section, we introduce an analogous notion of distance for matroids and show that two elements at a large distance in a regular matroid force the existence of a large-rank flat whose restriction is the cycle matroid of a clean ladder. In this sense, clean ladders play the role of $2$-connected analogues of induced paths.

In the introduction, we briefly described the distance between two distinct elements in a $2$-connected matroid. We now give the formal definition. Let $M$ be a matroid, and let $e$ and $f$ be distinct elements of $M$. The \emph{distance between $e$ and $f$ in $M$}, denoted by $\dist_M(e,f)$, is the size of a smallest circuit of $M$ containing both $e$ and $f$. If no circuit of $M$ contains both $e$ and $f$, then we set $\dist_M(e,f)=\infty$. We omit the subscript when the host matroid is clear from the context. The \emph{diameter} of $M$ is the maximum of $\dist(e,f)$ over all pairs of distinct elements $e$ and $f$ of $M$. If $M$ is $2$-connected and $|E(M)|\geq 2$, then every pair of distinct elements of $M$ is contained in a circuit, and hence the diameter of $M$ is finite.

We next consider $2$-connected flats containing two prescribed elements $e$ and $f$. Among such flats, the minimum possible rank is that of the closure of a smallest circuit containing both $e$ and $f$. Indeed, if $C$ is such a circuit, then $\cl_M(C)$ is a $2$-connected flat whose rank is $|C|-1$, that is, $\dist_M(e,f)-1$. Conversely, every $2$-connected flat containing both $e$ and $f$ contains a circuit through $e$ and $f$, and hence has rank at least $\dist_M(e,f)-1$. We will use the following result~\cite[Corollary~1.4]{Ge_Jagdeep_Oxley} that characterizes when the distance between two elements in a $2$-connected regular matroid is a maximum.

\begin{lemma}\label{lem: M is a ladder}
Let $M$ be a simple $2$-connected regular matroid, and let $e$ and $f$ be distinct elements of $M$. If every circuit of $M$ containing both $e$ and $f$ is spanning, then $M\cong M(L)$ for some clean ladder $L$.
\end{lemma}

The following result states that if a $2$-connected regular matroid has large diameter, then it has a large-rank $2$-connected graphic flat.

\begin{lemma}\label{lem: diameter implies graphic flat}
    Suppose that $M$ is a $2$-connected regular matroid with diameter $d$. Then $M$ has a $2$-connected graphic flat of rank $d-1$.
\end{lemma}

\begin{proof}
    Let $e$ and $f$ be distinct elements of $M$ such that $\dist(e,f)=d$. Then there is a circuit $C$ of $M$ such that $\{e,f\}\subseteq C$ and $|C|=d$. Let $F=\cl_M(C)$. Clearly, $M|F$ is $2$-connected. Moreover, since every circuit of $M$ containing $\{e,f\}$ has at least $d$ elements, every circuit of $M|F$ containing $\{e,f\}$ is a spanning circuit. Thus, by Lemma~\ref{lem: M is a ladder}, the simplification of $M|F$ is the cycle matroid of a clean ladder. As $F$ is a flat of $M$ and $r(M|F)=|C|-1=d-1$, the result follows.
\end{proof}

In fact, if the canonical tree decomposition $\cT$ of a $2$-connected matroid $M$ has a long path, then $M$ must have large diameter. Recall that the {\it length} of a path is the number of edges in that path.

\begin{lemma}\label{lem: long paths implies large diameter}
    Let $M$ be a $2$-connected matroid, and let $\cT$ be its canonical tree decomposition. If a longest path of $\cT$ has length $l$, then the diameter of $M$ is at least $\lfloor \frac{l+5}{2}\rfloor$.
\end{lemma}

\begin{proof}
    Let $M_0,M_1,\dots,M_l$ be the vertices of a longest path $\cP$ of $\cT$. For each $i\in\{0,1,\dots,l-1\}$, let $x_i$ be the basepoint involved in the $2$-sum of $M_i$ and $M_{i+1}$. Performing all the $2$-sums corresponding to edges in $E(\cT)-E(\cP)$, we obtain a path decomposition $\cP'$ of $M$ with vertices $N_0,N_1,\dots,N_l$ where $x_i$ is the basepoint involved in the $2$-sum of $N_i$ and $N_{i+1}$ for each $i\in\{0,1,\dots,l-1\}$.

    Choose elements $e\in E(N_0)-\{x_0\}$ and $f\in E(N_l)-\{x_{l-1}\}$. Then
    \begin{align*}
        \dist_M(e,f)=\dist_{N_0}(e,x_0)+\dist_{N_1}(x_0,x_1)+\dots+\dist_{N_l}(x_{l-1},f)-2l.
    \end{align*}
    Let $(d_0,d_1,\dots,d_l)=(\dist_{N_0}(e,x_0),\dist_{N_1}(x_0,x_1),\dots,\dist_{N_l}(x_{l-1},f))$. It is straightforward to see that $d_i\geq 2$ and that equality is attained if and only if $M_i$ is labeled by a cocircuit. Since $\cT$ is the canonical tree decomposition of $M$ and there are no adjacent vertices labeled by cocircuits, it follows that there are at least $\lfloor \frac{l+1}{2}\rfloor$ vertices in $\{M_0,M_1,\dots,M_l\}$ that are not labeled by cocircuits. Hence
    \begin{align*}
        \dist_M(e,f)\geq 2(l+1)+\lfloor \frac{l+1}{2}\rfloor-2l=\lfloor \frac{l+5}{2}\rfloor.
    \end{align*}
    Therefore, the diameter of $M$ is at least $\lfloor \frac{l+5}{2}\rfloor$.
\end{proof}

We note that the bound in Lemma~\ref{lem: long paths implies large diameter} is best possible. Indeed, one can take copies of $U_{1,3}$ and $U_{2,3}$ and form their $2$-sums alternately. If the construction starts and ends with copies of $U_{1,3}$, then the bound in Lemma~\ref{lem: long paths implies large diameter} is attained, as illustrated in Figure~\ref{fig:extremal diameter}.

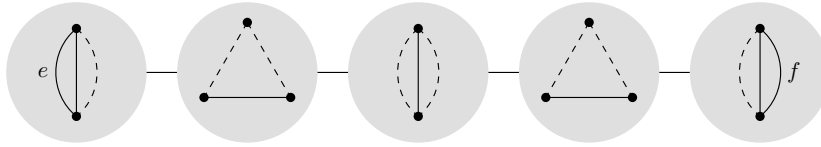
\begin{figure}[htb]
    \centering
    \resizebox{11cm}{!}{\input{figures/fig_extremal_diameter}}%
    \caption{A graph $G$ such that the diameter of $M(G)$ is exactly $\lfloor\frac{l+5}{2}\rfloor$.}
    \label{fig:extremal diameter}
\end{figure}

In the following subsection, we use Lemmas~\ref{lem: diameter implies graphic flat} and~\ref{lem: long paths implies large diameter} to reduce the proof of Theorem~\ref{thm:large_graphic_flat} to the $3$-connected case.

\subsection{Reduction to the $3$-connected case}\label{subsec: reduction to 3-connected case}

The following proposition is well known. A proof may be found in~\cite[Lemma~3.1(\romannum{1})]{oxsin}.

\begin{proposition}\label{prop: spanning circuit containing e}
Let $M$ be a loopless $2$-connected binary matroid, and let $e\in E(M)$. Then every circuit of $M$ containing $e$ is spanning if and only if $M$ is obtained from a circuit $C$ with $|C|\geq 2$ and $e\in C$ by replacing some elements of $C-\{e\}$ by nonempty parallel classes. In particular, $M$ is $2$-connected and graphic.
\end{proposition}

We now use this proposition to prove a result that helps simplify the structure of a matroid when seeking $2$-connected flats.

\begin{lemma}\label{lem: cleaning 2-connected leaves}
    Let $\cT$ be the canonical tree decomposition of a $2$-connected matroid $M$, and let $N$ be a vertex of $\cT$. If $N$ has a $2$-connected flat $F$, then $M$ has a $2$-connected flat $F'$ such that $r(M|F')\geq r(N|F)$ and $F'$ contains every element of $E(M)\cap F$. Moreover, if $F$ is graphic, then $F'$ is graphic.
\end{lemma}

\begin{proof}
    Let $A_1,A_2,\dots,A_k$ be the $2$-connected matroids corresponding to the components of $\cT-N$. It is clear that $M$ can be obtained by taking $2$-sums of $N$ with $A_1,A_2,\dots,A_k$, in any order. Moreover, $E(A_1),E(A_2),\dots,E(A_k)$ are pairwise disjoint, and each meets $E(N)$ in exactly one element.

    For each $i\in\{1,2,\dots,k\}$, let $a_i$ be the unique element of $E(A_i)\cap E(N)$, and let $C_i$ be a smallest circuit of $A_i$ containing $a_i$. Put $F_i=\cl_{A_i}(C_i)$. By Proposition~\ref{prop: spanning circuit containing e}, the set $F_i$ is a $2$-connected graphic flat of $A_i$ containing $a_i$. Let
    \[
        F'=\left(\bigcup_{i:\ a_i\in F}F_i\right)\cup \left(F-\{a_1,a_2,\dots,a_k\}\right).
    \]
    Then $F'$ contains every element of $E(M)\cap F$. By Proposition~\ref{prop: 2-sum of two flats}, the set $F'$ is a flat of $M$, and $M|F'$ can be obtained by taking $2$-sums, in any order, of $N|F$ and matroids $A_i|F_i$, for all $i$ such that $a_i\in F$. Therefore $r(M|F')\geq r(N|F)$. Since $N|F$ and each $A_i|F_i$ with $a_i\in F$ are $2$-connected, so is $M|F'$. Moreover, since each $A_i|F_i$ is graphic, it follows that $M|F'$ is graphic whenever $N|F$ is graphic.
\end{proof}

The main result of this subsection is the following theorem. It shows that, for a large-rank $2$-connected regular matroid $M$ with canonical tree decomposition $\cT$, either $M$ has a large-rank $2$-connected graphic flat, or some vertex of $\cT$ is labeled by a large-rank $3$-connected matroid.

\begin{theorem}\label{thm: high-rank 3-connected vertex or high-rank graphic flat}
    For every pair of positive integers $k$ and $\rho$ with $3\leq k\leq \rho$, there is an integer
    $f_{\ref{thm: high-rank 3-connected vertex or high-rank graphic flat}}(k,\rho)$ such that every
    $2$-connected regular matroid $M$ with
    \[
        r(M)\geq f_{\ref{thm: high-rank 3-connected vertex or high-rank graphic flat}}(k,\rho)
    \]
    satisfies at least one of the following:
    \begin{enumerate}
        \item[(\romannum{1})] $M$ has a $2$-connected graphic flat $F$ with $r(F)\geq k$, or
        \item[(\romannum{2})] the canonical tree decomposition $\cT$ of $M$ has a vertex labeled by a
        $3$-connected matroid $N$ with $r(N)\geq \rho$.
    \end{enumerate}
\end{theorem}

Before giving the proof, we recall the following well-known lemmas.

\begin{lemma}\label{lem: long path or high degree vertex}
    For all positive integers $d$ and $l$, there is an integer $f_{\ref{lem: long path or high degree vertex}}(d,l)$ such that every tree with at least $f_{\ref{lem: long path or high degree vertex}}(d,l)$ vertices has a vertex of degree at least $d$, or a path of length at least $l$.
\end{lemma}

\begin{lemma}\label{lem: binary projective geometry}
    For every positive integer $r$, a rank-$r$ simple binary matroid has at most $2^r-1$ elements.
\end{lemma}

\begin{proof}[Proof of Theorem~\ref{thm: high-rank 3-connected vertex or high-rank graphic flat}] Let
\[
    v=f_{\ref{lem: long path or high degree vertex}}(2^{\rho-1},2k-3).
\]
We shall show that if
\[
    r(M)\geq \rho v-2v+1,
\]
then at least one of (\romannum{1}) and (\romannum{2}) holds. Suppose that both (\romannum{1}) and (\romannum{2}) fail. We first prove the following.

\begin{sublemma}\label{sublem: every vertex rank less than rho}
    Each vertex of $\cT$ is labeled by a matroid of rank less than $\rho$.
\end{sublemma}

Suppose that some vertex $V$ of $\cT$ is labeled by a matroid with $r(V)\geq \rho$. Since (\romannum{2}) fails, $V$ is labeled by a circuit or a cocircuit. As $E(V)$ is a $2$-connected graphic flat of $V$, Lemma~\ref{lem: cleaning 2-connected leaves} implies that $M$ has a $2$-connected graphic flat of rank at least $\rho\geq k$, a contradiction. Thus~\ref{sublem: every vertex rank less than rho} holds.\\

Next we prove the following.

\begin{sublemma}\label{sublem: T has at least v vertices}
    The tree $\cT$ has at least $v$ vertices.
\end{sublemma}

Recall that
\[
    r(M)=\sum_{N\in V(\cT)}r(N)-|E(\cT)|.
\]
By~\ref{sublem: every vertex rank less than rho},
\[
    r(M)\leq (\rho-1)|V(\cT)|-|E(\cT)|
    =\rho |V(\cT)|-2|V(\cT)|+1.
\]
Suppose that $|V(\cT)|<v$. Then, as $\rho\geq 3$,
\[
    r(M)<\rho v-2v+1,
\]
a contradiction. Therefore $\cT$ has at least $v$ vertices, that is,~\ref{sublem: T has at least v vertices} holds.\\

Recall that
\[
    v=f_{\ref{lem: long path or high degree vertex}}(2^{\rho-1},2k
    -3).
\]
By Lemma~\ref{lem: long path or high degree vertex}, the tree $\cT$ has a vertex of degree at least $2^{\rho-1}$ or a path of length at least $2k-3$. If $\cT$ has a path of length at least $2k-3$, then Lemma~\ref{lem: long paths implies large diameter} implies that the diameter of $M$ is at least $k+1$. Hence, by Lemma~\ref{lem: diameter implies graphic flat}, $M$ has a $2$-connected graphic flat of rank at least $k$, contradicting the fact that (\romannum{1}) fails.

We may now assume that $\cT$ has a vertex $P$ of degree at least $2^{\rho-1}$. Since each edge incident with $P$ corresponds to a distinct element of $E(P)$, we have
\[
    |E(P)|\geq 2^{\rho-1}.
\]

Next we prove the following.

\begin{sublemma}\label{sublem: P is 3-connected}
    The vertex $P$ is labeled by a simple $3$-connected regular matroid.
\end{sublemma}

Suppose first that $P$ is labeled by a circuit or a cocircuit. Since $|E(P)|\geq 2^{\rho-1}$ and $\rho\geq 3$, the vertex $P$ is labeled by a cocircuit, or by a circuit with at least four elements. If $P$ is a circuit, then, since $\rho\geq 3$,
\[
    r(P)=|E(P)|-1\geq 2^{\rho-1}-1\geq\rho,
\]
contradicting~\ref{sublem: every vertex rank less than rho}. Thus $P$ is labeled by a cocircuit.

Let $D_1,D_2,\dots,D_n$ be the $2$-connected matroids corresponding to the components of $\cT-P$. Then $M$ can be obtained by taking $2$-sums of $P$ with $D_1,D_2,\dots,D_n$, in any order. For each $i\in\{1,2,\dots,n\}$, let $d_i$ be the unique element of $E(D_i)\cap E(P)$, and let $O_i$ be a smallest circuit of $D_i$ containing $d_i$. Put
\[
    Q_i=\cl_{D_i}(O_i).
\]
By Proposition~\ref{prop: spanning circuit containing e}, the set $Q_i$ is a $2$-connected graphic flat of $D_i$ containing $d_i$. Let
\[
    Q=\left(\bigcup_{1\leq i\leq n}Q_i\right)
        \cup \left(E(P)-\{d_1,d_2,\dots,d_n\}\right).
\]
Then $Q$ is a flat of $M$, and $M|Q$ can be obtained by taking $2$-sums, in any order, of $P$ with each $D_i|Q_i$, for $i\in\{1,2,\dots,n\}$. Since each $D_i|Q_i$ and $P$ are $2$-connected and graphic, so is $M|Q$. Moreover, since $P$ is not adjacent in $\cT$ to a vertex labeled by a cocircuit, each circuit $O_i$ has at least three elements. Hence each $D_i|Q_i$ has rank at least two. As $n$ is the degree of $P$ in $\cT$ and $n\geq 2^{\rho-1}$, we have $r(M|Q)\geq 2^{\rho-1}+1>k$. Thus $M$ has a $2$-connected graphic flat of rank at least $k$, a contradiction. Hence~\ref{sublem: P is 3-connected} holds.\\

As $P$ is a simple $3$-connected regular matroid with at least $2^{\rho-1}$ elements, Lemma~\ref{lem: binary projective geometry} implies that $r(P)\geq \rho$. Therefore (\romannum{2}) holds, a contradiction. This completes the proof of Theorem~\ref{thm: high-rank 3-connected vertex or high-rank graphic flat}.
\end{proof}

We note that the integer $f_{\ref{thm: high-rank 3-connected vertex or high-rank graphic flat}}(k,\rho)$ used in the previous proof can easily be improved. One way to do so is to replace Lemma~\ref{lem: binary projective geometry} with a sharper density bound for regular matroids. Since we only need the existence of such an integer $f_{\ref{thm: high-rank 3-connected vertex or high-rank graphic flat}}(k,\rho)$, we use Lemma~\ref{lem: binary projective geometry}, which may be more familiar to most readers.

As proved by Theorem~\ref{thm: high-rank 3-connected vertex or high-rank graphic flat}, if a large-rank $2$-connected matroid $M$ has no large-rank $2$-connected graphic flat, then the canonical tree decomposition $\cT$ of $M$ must have a vertex labeled by a large-rank $3$-connected matroid $N$. We note that a flat $F$ of $N$ need not be a flat of $M$, since some elements of $F$ may be involved in the $2$-sums corresponding to the edges of $\cT$ incident with the vertex labeled by $N$. However, if $F$ is a $2$-connected graphic flat of $N$, then Lemma~\ref{lem: cleaning 2-connected leaves} guarantees a $2$-connected graphic flat $F'$ of $M$ such that $r(M|F')\geq r(N|F)$. Therefore, it remains to find a large-rank $2$-connected graphic flat in a $3$-connected regular matroid. In the next section, we establish several tools for this purpose.

\section{The enlarged-wheel replacement theorem}\label{sec: downgrade 3-sums}

In this section, we introduce tools for simplifying the structure on one side of a $3$-sum. The main result of this section is Theorem~\ref{thm: a graph that keeps a triangle}. Before presenting it, we prove several useful lemmas.

Let $e$ and $f$ be distinct elements in a matroid $M$. We say $e$ is {\it freer} than $f$ if $f$ is contained in the closure of every circuit containing $e$. If neither $e$ nor $f$ is freer than the other, then $e$ and $f$ are {\it incomparable}. We omit the elementary proof of the first two of the following results.

\begin{proposition}\label{prop:freer}
    Let $e$ and $f$ be distinct elements in a matroid $M$ such that $e$ is freer than $f$. If $N$ is a minor of $M$ such that $\{e,f\}\subseteq E(N)$, then $e$ is freer than $f$ in $N$.
\end{proposition}

\begin{proposition}\label{prop: K4 incomparable}
    If $e$ and $f$ are distinct elements in $M(K_4)$, then $e$ and $f$ are incomparable.
\end{proposition}

Oxley~\cite[Corollary~(3.7)]{Oxley_nonbinary} proved the following result.

\begin{lemma}\label{lem: K4 use a triangle}
    Let $M$ be a $3$-connected binary matroid having rank and corank at least three and suppose that $\{a,b,c\}\subseteq E(M)$. Then $M$ has a minor $N$ isomorphic to $M(K_4)$ such that $\{a,b,c\}\subseteq E(N)$.
\end{lemma}

The following result is used frequently in our arguments.

\begin{lemma}\label{lem: circuit contains only one in its closure}
Let $M$ be a $2$-connected binary matroid and let $\{a,b,c\}$ be a triangle of $M$. If $\{a,b,c\}$ is contained in a $3$-connected minor of $M$ with at least four elements, then, for each $x\in\{a,b,c\}$, there is a circuit $C$ of $M$ such that $x\in C$ and $\cl(C)\cap \{a,b,c\}=\{x\}$.
\end{lemma}

\begin{proof}
By symmetry, we may assume that $x=a$. Because a $3$-connected binary matroid with at least four elements has rank and corank at least three, Lemma~\ref{lem: K4 use a triangle} implies that $M$ has a $M(K_4)$-minor $N$ such that $\{a,b,c\}\subseteq E(N)$. By Proposition~\ref{prop: K4 incomparable}, the elements $a,b,c$ are mutually incomparable in $N$. Hence they are also mutually incomparable in $M$ by Proposition~\ref{prop:freer}. Therefore, there is a circuit $C$ of $M$ with $a\in C$ and $b\notin\cl(C)$. Since $\{a,b,c\}$ is a triangle, we have $b\in\cl(\{a,c\})$, so $c\notin\cl(C)$. Hence $\cl(C)\cap\{a,b,c\}=\{a\}$.
\end{proof}

For $n\geq 2$, a {\it wheel} is a graph that is obtained from an $n$-cycle (the {\it rim}) by adding a vertex $v$ and joining $v$ to every vertex of the rim by a single edge (a {\it spoke}). An {\it enlarged wheel} is a graph obtained from a wheel graph by adding parallel edges to spoke edges, and subdividing rim edges. The {\it rim} of an enlarged wheel is the potentially subdivided rim of the initial wheel. Figure~\ref{fig:broken wheels} shows some examples of enlarged wheels.

\begin{figure}[htb]
    \centering
    \resizebox{11cm}{!}{\input{figures/fig_broken_wheels}}%
    \caption{Examples of enlarged wheels.}
    \label{fig:broken wheels}
\end{figure}
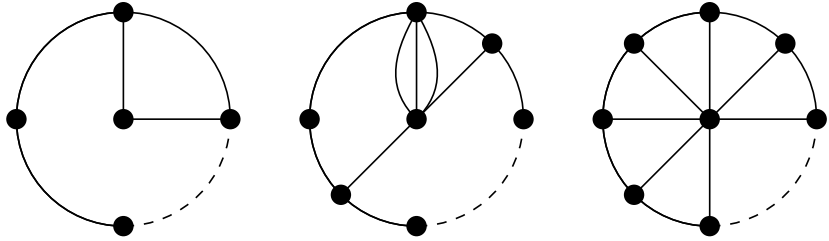

The following result gives a sufficient condition under which a regular matroid is graphic.

\begin{lemma}\label{lem: a regular matroid having circuit hyperplane is graphic}
    If $M$ is a regular matroid with a circuit-hyperplane $C$, then $M$ is graphic.
\end{lemma}

\begin{proof}
    Suppose that $M$ is not graphic. We first show that

    \begin{sublemma}\label{sublem: if cographic then graphic}
        $M$ is not cographic.
    \end{sublemma}

    Suppose that $M=M^*(H)$ for some graph $H$. Since $C$ is a circuit-hyperplane of $M$, we have $M/C\cong U_{1,n}$ for some positive integer $n$. Hence $M(H\backslash C)=(M/C)^*\cong U_{n-1,n}$. Therefore, $H\backslash C$ consists of an $n$-cycle together with an isolated vertex. It follows that $C$ is a vertex bond of $H$ and $H$ is an enlarged wheel. In particular, $H$ is planar and hence $M$, which is $M^*(H)$, is graphic, a contradiction. Therefore~\ref{sublem: if cographic then graphic} holds.\\

    Now we proceed by induction on $|E(M)|$. By Theorem~\ref{thm: graphic regular matroids}, the statement holds for all regular matroids with at most eight elements. We may now assume that $|E(M)|\geq 9$ and that the statement holds for all regular matroids with at most $|E(M)|-1$ elements.

    If $M$ is not $2$-connected, then there are regular matroids $M_1$ and $M_2$ such that $M=M_1\oplus M_2$. Because $C$ is a circuit, we may assume that $C\subseteq E(M_1)$. Therefore, $r(M_2)\leq 1$. Because $C$ is a hyperplane, $\cl(C)\cap E(M_2)=\emptyset$. Thus $r(M_2)=1$. Because $C$ is a circuit-hyperplane, $E(M_1)=C$ and hence $M_1$ is graphic. Clearly $M_2$ is also graphic, so $M$ is graphic, a contradiction.

    We may now assume that $M$ is $2$-connected. If $M$ is not $3$-connected, then there are regular matroids $M_3$ and $M_4$ such that $M=M_3\oplus_2 M_4$ at the basepoint $p$. Since $C$ is a circuit-hyperplane,
    \[|C|=r(M)=r(M_3)+r(M_4)-1.\]
    Therefore, we may assume that $C$ is of the form $(C_3\cup C_4)-\{p\}$ where $C_3$ is a spanning circuit of $M_3$, and $C_4$ is a circuit of size $r(M_4)$ in $M_4$. Moreover, as $C$ is a hyperplane of $M$, we know that $E(M_3)=C_3$ and $C_4$ is a circuit-hyperplane of $M_4$. By the inductive hypothesis, $M_4$ is graphic. Since $M_3$ is clearly graphic, it follows that $M$ is graphic, a contradiction.

    We may now assume that $M$ is $3$-connected. By~\cite[p.~656]{Oxl11}, every circuit of the rank-$5$ matroid $R_{10}$ has four or six elements. Thus $R_{10}$ has no circuit-hyperplanes and $M$ is not $R_{10}$. By Theorem~\ref{thm: Aprile and Fiorini}, $M$ has a $\triangle$-tree decomposition $\cT$. If $\cT$ has only one vertex, then $M$ is either graphic or cographic. By~\ref{sublem: if cographic then graphic}, this is a contradiction. Therefore, we may assume that $|V(T)|\geq 2$. In particular, there are two regular matroids $M_5$ and $M_6$ (not necessarily graphic or cographic) such that $M=M_5\oplus_3 M_6$ on a common triangle $\{x,y,z\}$. Since $C$ is a circuit-hyperplane,
    \[|C|=r(M)=r(M_5)+r(M_6)-2.\]
    As $M$ is simple and $\min\{|E(M_5)|,|E(M_6)|\}\geq 7$, it follows that $\min\{r(M_5),r(M_6)\}\geq 3$. Therefore, $C$ is of the form $(C_5\cup C_6)-\{t\}$ where $C_i$ is a circuit of $M_i$ for each $i\in\{5,6\}$, and $t$ is an element of $\{x,y,z\}$. Without loss of generality, we assume that $t=x$. By symmetry, there are two possibilities:
    \begin{enumerate}
        \item[(\romannum{1})] $|C_5|=r(M_5)+1$ and $|C_6|=r(M_6)-1$, or
        \item[(\romannum{2})] $|C_5|=r(M_5)$ and $|C_6|=r(M_6)$.
    \end{enumerate}

    Suppose (\romannum{1}) holds. As $C_5$ is a spanning circuit of $M_5$ and $C$ is a circuit-hyperplane of $M$, it is straightforward to deduce that $E(M_5)=C_5\cup\{y,z\}$. Since $|E(M_5)|\geq 7$, it follows that $|C_5|\geq 5$ and hence $M_5$ has a series pair $\{e,f\}$ such that $\{e,f\}\cap \{x,y,z\}=\emptyset$. Moreover, $\{e,f\}$ is a series pair of $M$, contradicting the fact that $M$ is $3$-connected. Therefore (\romannum{1}) fails.

    We may now assume that (\romannum{2}) holds. Because $C$ is a circuit-hyperplane of $M$, we know that $\cl_{M_i}(C_i)-\{y,z\}=C_i$ for each $i\in\{5,6\}$. Suppose that $y\in\cl_{M_5}(C_5)$. It follows immediately that $\{y,z\}\subseteq \cl_{M_5}(C_5)$. Then $y$ is in the closure of $C_5-\{x\}$ and hence $M_5$ has a circuit $D_5$ such that $y\in D_5\subseteq (C_5-\{x\})\cup \{y\}$. Moreover, $|D_5|<|C_5|$ since $M$ is binary. If $y\in\cl_{M_6}(C_6)$, then $M_6$ has a circuit $D_6$ such that $y\in D_6\subseteq (C_6-\{x\})\cup\{y\}$. However, $(D_5\cup D_6)-\{y\}$ is a dependent set of $M$ that is properly contained in $C$, a contradiction. Thus, $y\notin \cl_{M_6}(C_6)$. Because $\{x,y,z\}$ is a triangle and $x\in C_6$, it is easy to deduce that $\{y,z\}\cap \cl_{M_6}(C_6)=\emptyset$ and hence $C_6$ is a circuit-hyperplane of $M_6$. Moreover, as $\{x,y,z\}$ is a coindependent triangle of $M_6$, there is an element $w\in E(M_6)-(C_6\cup\{y,z\})$. However, as $\{y,z\}\subseteq \cl_{M_5}(C_5)$ and $w\in\cl_{M_6}(C_6\cup\{y,z\})$, it follows that $w\in \cl_{M}(C)$, contradicting the fact that $C$ is a circuit-hyperplane. Therefore $\{y,z\}\cap \cl_{M_5}(C_5)=\emptyset$ and, by symmetry, $\{y,z\}\cap \cl_{M_6}(C_6)=\emptyset$. Thus, $C_5$ and $C_6$ are circuit-hyperplanes of $M_5$ and $M_6$, respectively. By the inductive hypothesis, both $M_5$ and $M_6$ are graphic. Hence so is their $3$-sum $M$. This contradiction completes the proof.    
\end{proof}

\begin{lemma}\label{lem: enlarged wheels}
    Every loopless regular matroid $M$ that has a circuit-hyperplane $C$ is the cycle matroid of an enlarged wheel $W$. Moreover, $C$ is the rim of $W$.
\end{lemma}

\begin{proof}
    By Lemma~\ref{lem: a regular matroid having circuit hyperplane is graphic}, the matroid $M$ is graphic. Thus $M=M(W)$ for some connected graph $W$. Let $C$ be a circuit-hyperplane of $M$. Then $C$ is a cycle of $W$ such that $|V(W)-V(C)|=1$. Let $v$ be the vertex in $V(W)-V(C)$. As $C$ is a hyperplane of $M(W)$ and $M$ is loopless, every edge of $W$ not in $C$ joins $v$ to a vertex of the cycle $C$. Thus, $W$ is an enlarged wheel and $C$ is the rim of $W$.
\end{proof}

We now have all the ingredients needed to prove the enlarged-wheel replacement theorem, which is the main result of this section.

\begin{theorem}\label{thm: a graph that keeps a triangle}
    Let $M$ be a $3$-connected regular matroid having at least four elements, and let $\{a,b,c\}$ be a triangle of $M$. Then there is a flat $F$ of $M$ such that
    \begin{enumerate}
        \item[(\romannum{1})] $M|F$ is the cycle matroid of an enlarged wheel $W$, and
        \item[(\romannum{2})] in $W$, the edge $a$ is a rim edge, the edges $b$ and $c$ are spokes, and $a$, $b$, and $c$ form a triangle.
    \end{enumerate}
\end{theorem}

\begin{proof}
    By Lemma~\ref{lem: circuit contains only one in its closure}, there is a circuit $A$ of $M$ such that $a\in A$ and $\cl(A)\cap\{a,b,c\}=\{a\}$. Choose such a circuit $A$ with minimum size. We first show that

    \begin{sublemma}\label{sublem: A is a flat}
        $\cl(A)=A$.
    \end{sublemma}
    
    Suppose that there is an element $e\in \cl(A)-A$. Since $A$ is a circuit, the element $e$ is also in $\cl(A-\{a\})$. Hence there is a circuit $C$ such that $e\in C\subseteq (A-\{a\})\cup\{e\}$. Since $M$ is binary, the symmetric difference $A\triangle C$ is a disjoint union of circuits. Since $a\in A\triangle C$, there is a circuit $A'$ such that $a\in A'\subseteq A\triangle C$. As $M$ is $3$-connected, $|C|\geq 3$. Thus $|A'|\leq |A\triangle C|\leq |A|-1$. Moreover, since $\cl(A')\subseteq \cl(A)$, we have $\cl(A')\cap\{a,b,c\}=\{a\}$. Therefore $A'$ is a circuit such that $a\in A'$ and $\cl(A')\cap\{a,b,c\}=\{a\}$, with $|A'|<|A|$, a contradiction. Thus~\ref{sublem: A is a flat} holds.\\

    Let $F=\cl(A\cup\{b,c\})$. It remains to show that $F$ satisfies (\romannum{1}) and (\romannum{2}). Since $\{a,b,c\}$ is a triangle and $a\in A$, we have $r(F)=r(A)+1$. Hence, by~\ref{sublem: A is a flat}, the set $A$ is a circuit-hyperplane of $M|F$. Therefore, (\romannum{1}) follows immediately from Lemma~\ref{lem: enlarged wheels}. Let $W$ be an enlarged wheel such that $M|F=M(W)$. By Lemma~\ref{lem: enlarged wheels}, the set $A$ is the rim of $W$. Thus the edge $a$ is a rim edge, and the edges $b$ and $c$ are spokes. Moreover, since $\{a,b,c\}$ is a triangle of $M$, it is also a triangle of $M|F$. Hence (\romannum{2}) holds.
\end{proof}

    Suppose that a $3$-connected matroid $M$ is the $3$-sum $M_1\oplus_3 M_2$ across a common triangle $T$. Theorem~\ref{thm: a graph that keeps a triangle} allows us to reduce one side of this $3$-sum to a graphic flat while retaining the common triangle $T$, so that we may apply a quasi-$3$-sum across this triangle. We shall see in the next section that, in some circumstances, the structure of an enlarged wheel allows us to preserve $2$-connectivity.
\section{Graphic flats in $3$-connected regular matroids}\label{sec: 3-con cases}

In this section, we prove the following.

\begin{theorem}\label{thm: large 3-connected (main)}
    For every positive integer $k$, there is an integer $f_{\ref{thm: large 3-connected (main)}}(k)$ such that every $3$-connected regular matroid $M$ with $r(M)\geq f_{\ref{thm: large 3-connected (main)}}(k)$ has a $2$-connected graphic flat $F$ with $r(F)\geq k$.
\end{theorem}

We prove Theorem~\ref{thm: large 3-connected (main)} in two steps using the structure of a $\triangle$-tree decomposition of $M$. In the next two subsections, we treat separately the case when the $\triangle$-tree has many vertices, and the case when it has a vertex labeled by a large-rank matroid. The next lemma ensures that one of these cases occurs.

\begin{lemma}\label{lem: large delta-tree or large vertex}
    For every pair $v$ and $\rho$ of positive integers, there is an integer
    $f_{\ref{lem: large delta-tree or large vertex}}(v,\rho)$ such that every
    $3$-connected regular matroid $M$ with
    \[
        r(M)\geq f_{\ref{lem: large delta-tree or large vertex}}(v,\rho)
    \]
    has a $\triangle$-tree decomposition $\cT$ in which every vertex label has rank at least three, and
    \begin{enumerate}
        \item[(\romannum{1})] $\cT$ has at least $v$ vertices; or
        \item[(\romannum{2})] some vertex label of $\cT$ has rank at least $\rho$.
    \end{enumerate}
\end{lemma}

\begin{proof}
    We show that one may take
    \[
        f_{\ref{lem: large delta-tree or large vertex}}(v,\rho)
        =\max\{6,\rho v-3v-\rho+6\}.
    \]
    Let $M$ be a $3$-connected regular matroid with
    \[
        r(M)\geq f_{\ref{lem: large delta-tree or large vertex}}(v,\rho).
    \]
    Since $r(M)\geq 6$, the matroid $M\not\cong R_{10}$. Hence, by Proposition~\ref{prop: vertex label rank at least three}, $M$ has a $\triangle$-tree $\cT$ in which every vertex is labeled by a matroid of rank at least three.

    Suppose that both (\romannum{1}) and (\romannum{2}) fail. Since every vertex label has rank at least three, we may assume that $\rho\geq 4$. By Proposition~\ref{prop: rank of 3-sum},
    \begin{align*}
        r(M)
        &=\sum_{N\in V(\cT)}r(N)-2|E(\cT)|\\
        &\leq (\rho-1)|V(\cT)|-2(|V(\cT)|-1)\\
        &= (\rho-3)|V(\cT)|+2\\
        &\leq (\rho-3)(v-1)+2\\
        &= \rho v-3v-\rho+5.
    \end{align*}
    This contradiction to the assumption that $r(M)\geq \rho v-3v-\rho+6$ proves the lemma.
\end{proof}

\subsection{A large $\triangle$-tree}
Recall that the distance between two distinct elements is the size of a smallest circuit containing both of them. We now extend this notion to sets of elements. Let $A$ and $B$ be disjoint nonempty subsets of $E(M)$. The {\it distance} $\dist_M(A,B)$ between $A$ and $B$ is the minimum size of a circuit containing some element of $A$ and some element of $B$. We omit the subscript when the host matroid is clear from the context, and we omit braces around one-element sets. The next result describes how distance behaves under $3$-sums.

\begin{proposition}\label{prop: distance across 3-sums}
    Suppose that $M$ is the $3$-sum of two binary matroids $M_1$ and $M_2$ across a common triangle $T$. If $e\in E(M_1)-T$ and $f\in E(M_2)-T$, then
    \[
        \dist_M(e,f)\geq \dist_{M_1}(e,T)+\dist_{M_2}(f,T)-2.
    \]
\end{proposition}

\begin{proof}
    If $M$ has no circuit containing $\{e,f\}$, then $\dist_M(e,f)=\infty$, and the result holds. Thus, we may take a minimum-sized circuit $C$ of $M$ containing $\{e,f\}$. Since $C$ meets both $E(M_1)-T$ and $E(M_2)-T$, there are circuits $C_1$ and $C_2$ of $M_1$ and $M_2$, respectively, such that $C_1\cap T=C_2\cap T$ and $C=C_1\triangle C_2$. Moreover, by the definition of a $3$-sum, $|C_1\cap C_2|=1$. Hence
    \[\dist_M(e,f)=|C|=|C_1|+|C_2|-2\geq \dist_{M_1}(e,T)+\dist_{M_2}(f,T)-2.\qedhere\]
\end{proof}

Let $M$ be a regular matroid, and let $T_1$ and $T_2$ be disjoint triangles of $M$. We say that $T_1$ and $T_2$ are {\it parallel} if $r(T_1\cup T_2)=2$, that $T_1$ and $T_2$ are {\it intersecting} if $r(T_1\cup T_2)=3$, and that $T_1$ and $T_2$ are {\it skew} if $r(T_1\cup T_2)=4$. Figure~\ref{fig:relations of triangles} shows geometric representations of the three possibilities. It is clear that $\dist_M(T_1,T_2)=2$ if $T_1$ and $T_2$ are parallel or intersecting, and that $\dist_M(T_1,T_2)\geq 3$ if $T_1$ and $T_2$ are skew.

\begin{figure}[htb]
    \centering
    \resizebox{11cm}{!}{\input{figures/fig_relations_of_trianlges}}%
    \caption{Geometric representations of parallel triangles, intersecting triangles, and skew triangles.}
    \label{fig:relations of triangles}
\end{figure}
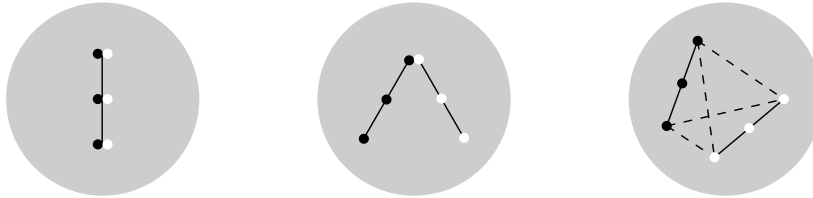

Let $M$ be a $3$-connected regular matroid, and let $\cP$ be a path decomposition of $M$ that is not necessarily a $\triangle$-tree. Clearly, every edge of $\cP$ corresponds to a $3$-sum of its two incident vertex labels. Let $P_0,P_1,\ldots,P_l$ be the vertex labels of $\cP$, and, for each $i\in\{1,2,\ldots,l\}$, let $T_i$ be the common triangle $E(P_{i-1})\cap E(P_i)$. Then, for each $i\in\{1,2,\ldots,l-1\}$, the matroid $P_i$ has two {\it distinguished} triangles $T_i$ and $T_{i+1}$ corresponding to the two edges of $\cP$ incident with $P_i$. We call $P_1,P_2,\dots, P_{l-1}$ the {\it internal vertices} of $\cP$; an edge of $\cP$ is {\it internal} if it joins two internal vertices. 

Suppose that $U$ and $V$ are adjacent internal vertices of a path decomposition $\cP$ of a $3$-connected regular matroid. Let the distinguished triangles of $U$ be $A$ and $B$, and let the distinguished triangles of $V$ be $B$ and $C$. Assume that $A$ and $B$ are intersecting in $U$, and that $B$ and $C$ are intersecting in $V$. We say that the edge of $\cP$ joining $U$ and $V$ is {\it separating} if $A$ and $C$ are skew in $U\oplus_3V$; otherwise, $A$ and $C$ are intersecting in $U\oplus_3V$, and the edge joining $U$ and $V$ is {\it nonseparating}. Figure~\ref{fig: separating/nonseparating} shows examples of separating and nonseparating edges.

\begin{figure}[htb]
    \centering
    \resizebox{11cm}{!}{\input{figures/fig_bundle}}%
    \caption{(a) A separating edge and (b) a nonseparating edge.}
    \label{fig: separating/nonseparating}
\end{figure}
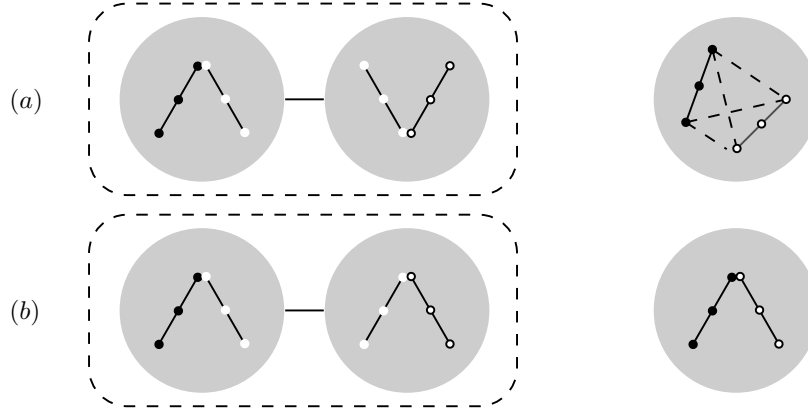

We next prove several useful results concerning the internal vertices of path decompositions of $3$-connected regular matroids.

\begin{lemma}\label{lem:many parallel triangles}
    Let $k$ be a positive integer. Suppose that $M$ is a $3$-connected regular matroid and that $\cP$ is a path decomposition of $M$ with at least $k$ vertices such that
    \begin{enumerate}
        \item[(\romannum{1})] each vertex of $\cP$ is labeled by a matroid of rank at least three, and
        \item[(\romannum{2})] for each internal vertex of $\cP$, the distinguished triangles are parallel.
    \end{enumerate}
    Then $M$ has a $2$-connected graphic flat of rank at least $k$.
\end{lemma}

\begin{proof}
    Let $P_0,P_1,\dots,P_l$ be the vertices of $\cP$ in order, and let $T_i=E(P_{i-1})\cap E(P_i)$ for each $i\in\{1,2,\dots,l\}$. Then $l\geq k-1$. Choose $t_1$ arbitrarily in $T_1$. For each $i\in\{2,3,\dots,l\}$, let $t_i$ be the element of $T_i$ that is parallel to $t_{i-1}$ in $P_{i-1}$ (see Figure~\ref{fig: parallel triangles}).

    Since each vertex of $\cP$ has rank at least three, Proposition~\ref{prop: almost 3-con} implies that the simplification of each $P_i$ is a $3$-connected matroid with at least four elements. Hence, by Lemma~\ref{lem: circuit contains only one in its closure}, for each $i\in\{0,1,\dots,l\}$ and each triangle $\{a,b,c\}$ of $P_i$, there is a circuit $C$ of $P_i$ such that $a\in C$ and $\cl_{P_i}(C)\cap\{a,b,c\}=\{a\}$.
    
    For each $j\in\{1,2,\dots,l\}$, let $C_j$ be a minimum-sized circuit of $P_j$ such that $t_j\in C_j$ and $\cl_{P_j}(C_j)\cap T_j=\{t_j\}$. Let $C_0$ be a minimum-sized circuit of $P_0$ such that $t_1\in C_0$ and $\cl_{P_0}(C_0)\cap T_1=\{t_1\}$. For each $j\in\{0,1,\dots,l\}$, let $F_j=\cl_{P_j}(C_j)$. Clearly, $P_j|F_j$ is $2$-connected for each $j\in\{0,1,\dots,l\}$. By Proposition~\ref{prop: spanning circuit containing e}, $P_j|F_j$ is also graphic. Therefore, by Lemma~\ref{lem: flats in 3-sum},
    \[
        (F_0\cup F_1\cup\dots\cup F_l)-\{t_1,t_2,\dots,t_l\}
    \]
    is a flat $F$ of $M$ such that
    \[
        M|F=P_0|F_0\oplus_2 P_1|F_1\oplus_2\dots\oplus_2 P_l|F_l.
    \]
    Thus $M|F$ is both $2$-connected and graphic. Moreover, since $M$ is $3$-connected, at most one of $C_0,C_1,\dots,C_l$ has size two. It follows that $r(P_j|F_j)\geq 2$ for all but at most one $j\in\{0,1,\dots,l\}$. Therefore,
    \[
        r(M|F)=\sum_{j=0}^l r(P_j|F_j)-l\geq l+1\geq k,
    \]
    and the proof is complete.
\end{proof}

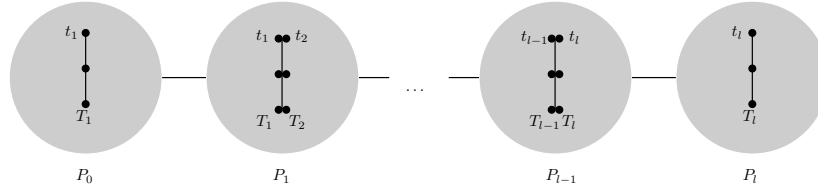
\begin{figure}[htb]
    \centering
    \resizebox{11cm}{!}{\input{figures/fig_many_parallel_traingles}}%
    \caption{A path decomposition with all internal vertices having parallel distinguished triangles.}
    \label{fig: parallel triangles}
\end{figure}

\begin{lemma}\label{lem: many intersecting}
    Let $k$ be a positive integer. Suppose that $M$ is a $3$-connected regular matroid with at least four elements, and that $\cP$ is a path decomposition of $M$ with at least $k+1$ vertices such that
    \begin{enumerate}
        \item[(\romannum{1})] every vertex of $\cP$ is labeled by a matroid of rank at least three, and
        \item[(\romannum{2})] every internal edge of $\cP$ is nonseparating.
    \end{enumerate}
    Then $M$ has a $2$-connected graphic flat of rank at least $k$.
\end{lemma}

Before the proof, we recall a theorem of Bixby and Cunningham~\cite[Theorem~3]{Bixby-Cunningham}.

\begin{theorem}\label{thm: connected hyperplane}
    For every element $e$ of a $3$-connected binary matroid $M$ with at least four elements, there are at least two $2$-connected hyperplanes of $M$ that avoid $e$.
\end{theorem}

\begin{proof}[Proof of Lemma~\ref{lem: many intersecting}.]
    Let $P_0,P_1,\dots,P_l$ be the vertices of $\cP$ in order. Then $l\geq k$. For each $i\in\{1,2,\dots,l\}$, let
    \[
        T_i=E(P_{i-1})\cap E(P_i).
    \]
    Since every internal edge of $\cP$ is nonseparating, the triangles $T_i$ and $T_{i+1}$ are intersecting in $P_i$ for each $i\in\{1,2,\dots,l-1\}$. Hence there are elements $t_1,t_2,\dots,t_l$, with $t_i\in T_i$ for each $i$, such that $t_j$ and $t_{j+1}$ are parallel in $P_j$ for each $j\in\{1,2,\dots,l-1\}$; see Figure~\ref{fig: interrsecting triangles}.

    As $\cP$ is a path decomposition, $M$ can be constructed from $P_0,P_1,\dots,P_l$ by applying the $3$-sums corresponding to $T_1,T_2,\dots,T_l$, in this order. When the last $3$-sum is performed in the construction of $M$, before the basepoint $t_{l}$ is deleted, there may or may not be an element $s$ in parallel to $t_{l}$. If there is, let $M'=M$. If there is not, let $M'$ be obtained by adding such an element $s$. Then $M'$ is certainly $3$-connected. Thus, by Theorem~\ref{thm: connected hyperplane}, $M'$ has a $2$-connected hyperplane $H$ that avoids $s$. Then $H$ is also a $2$-connected hyperplane of $M$.

    By Lemma~\ref{lem: flats in 3-sum}, the hyperplane $H$ has the form
    \[
        (H_0\cup H_1\cup\dots\cup H_l)-(T_1\cup T_2\cup\dots\cup T_l),
    \]
    where $H_i$ is a hyperplane of $P_i$ for each $i\in\{0,1,\dots,l\}$. Choose $e\in H_0-T_1$ and $f\in H_l-T_l$. For each $i\in\{1,2,\dots,l\}$, let
    \[
        T_i^-=T_i-\{t_i\}.
    \]
    Since $H$ avoids the point $s$,
    \begin{align*}
        \dist_{M|H}(e,f)
        &\ge \dist_{P_0}(e,T_1^-)+\dist_{P_l}(f,T_l^-)
        +\sum_{i=1}^{l-1}\dist_{P_i}(T_i^-,T_{i+1}^-)-2l\\
        &\ge 4+3(l-1)-2l\\
        &= l+1\\
        &\ge k+1.
    \end{align*}
    Thus $M|H$ has diameter at least $k+1$. By Lemma~\ref{lem: diameter implies graphic flat}, $M|H$ has a $2$-connected graphic flat of rank at least $k$. Since every flat of $M|H$ is also a flat of $M$, the result follows.
\end{proof}

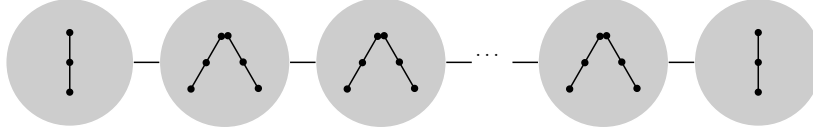
\begin{figure}[htb]
    \centering
    \resizebox{11cm}{!}{\input{figures/fig_many_intersecting_triangles}}%
    \caption{A path decomposition in which every internal edge is nonseparating.}
    \label{fig: interrsecting triangles}
\end{figure}

The next result proves the existence of a large-rank $2$-connected graphic flat when a $\triangle$-tree of a $3$-connected regular matroid contains a long path.

\begin{theorem}\label{thm: when Delta-tree has a long path}
    For every positive integer $k$, there is an integer $f_{\ref{thm: when Delta-tree has a long path}}(k)$ such that if a $3$-connected regular matroid $M$ has a $\triangle$-tree $\cT$ in which
    \begin{enumerate}
        \item[(\romannum{1})] each vertex of $\cT$ is labeled by a matroid of rank at least three, and
        \item[(\romannum{2})] $\cT$ contains a path of length $f_{\ref{thm: when Delta-tree has a long path}}(k)$,
    \end{enumerate}
    then $M$ has a $2$-connected graphic flat $F$ with $r(F)\ge k$.
\end{theorem}

\begin{proof}
    First we show that one may take $f_{\ref{thm: when Delta-tree has a long path}}(1)=f_{\ref{thm: when Delta-tree has a long path}}(2)=1$. It is clear that $r(M)\geq 3$, so the closure of a minimum-sized circuit of $M$ is a circuit. As this circuit has at least three elements, it is a $2$-connected graphic flat of rank at least two.
    
    We may now assume that $k\geq 3$. We show that one may take
    \[
        f_{\ref{thm: when Delta-tree has a long path}}(k)=2k^3-10k^2+17k-9.
    \]
    Suppose that $M$ has a $\triangle$-tree $\cT$ containing a path of length $f_{\ref{thm: when Delta-tree has a long path}}(k)$, and suppose that $M$ has no $2$-connected graphic flat of rank at least $k$. Let $\cP$ be a longest path of $\cT$, with vertices labeled by $M_0,M_1,\ldots,M_l$. Then
    \[
        l\ge 2k^3-10k^2+17k-9.
    \]
    
    We first perform all $3$-sums in $\cT$ corresponding to edges not belonging to $\cP$, as illustrated in Figure~\ref{fig:long path in delta tree}. The resulting path $\cP'$ has the same edge set as $\cP$, and we label its vertices by $P_0,P_1,\ldots,P_l$. Then $\cP'$ is a path decomposition of $M$, since $M$ can be obtained from $P_0,P_1,\ldots,P_l$ by taking the $3$-sums corresponding to the edges of $\cP'$. Note that $\cP'$ need not be a $\triangle$-tree, since its vertices may be labeled by matroids that are neither graphic nor cographic.
    
    \begin{figure}[htb]
        \centering
        \resizebox{11cm}{!}{\input{figures/fig_long_path}}%
        \caption{Absorbing the branches of $\cT$ into a long path $\cP'$.}
        \label{fig:long path in delta tree}
    \end{figure}
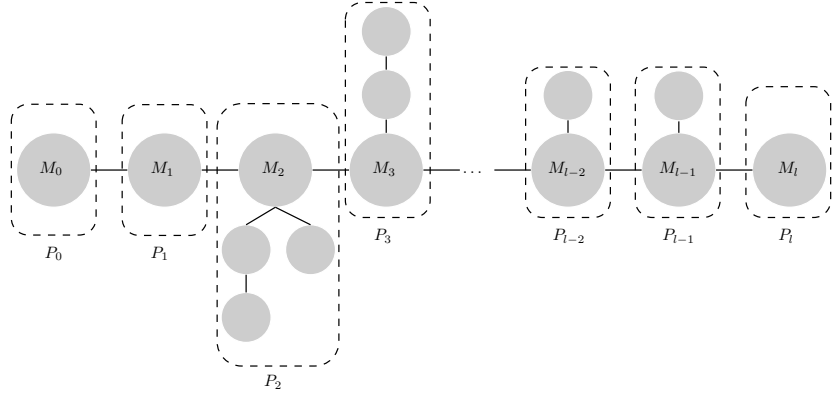
    
    Let $I$ be an internal vertex of $\cP'$ in which the distinguished triangles are parallel. Perform the $3$-sum corresponding to an edge incident with $I$. This operation decreases the number of internal vertices whose distinguished triangles are parallel. Hence we may repeat this operation until the resulting path decomposition has no such internal vertex; see Figure~\ref{fig: the first step}. Let $\cQ$ be the resulting path decomposition of $M$. We show the following.

    \begin{figure}[htb]
        \centering
        \resizebox{11cm}{!}{\input{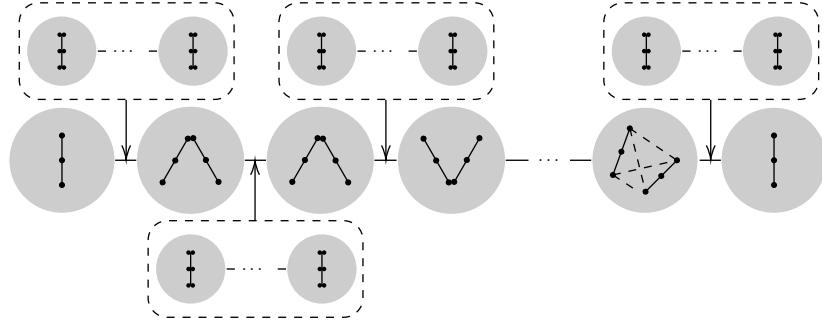}}%
        \caption{Performing $3$-sums involving an internal vertex in which the distinguished triangles are parallel.}
        \label{fig: the first step}
    \end{figure}

    \begin{sublemma}\label{sublem: Q has many vertices}
    The path decomposition $\cQ$ has at least $2k^2-6k+6$ vertices.
    \end{sublemma}
    
    Suppose that $\cQ$ has at most $2k^2-6k+5$ vertices. As illustrated in Figure~\ref{fig: the first step}, the internal vertices of $\cP'$ in which the distinguished triangles are parallel must be squeezed into at most $2k^2-6k+4$ positions.
    
    Suppose that at least $k-2$ such internal vertices are squeezed into a single position. Then $\cP'$ contains a subpath $\cI$ with at least $k$ vertices such that, for every internal vertex of $\cI$, the distinguished triangles are parallel. After performing all the $3$-sums corresponding to edges in $E(\cP')-E(\cI)$, we obtain a path decomposition satisfying the hypotheses of Lemma~\ref{lem:many parallel triangles}. Hence $M$ has a $2$-connected graphic flat of rank at least $k$, a contradiction. Therefore, there are at most $k-3$ internal vertices squeezed in each position, and $P'$ has at most
    \[
        (2k^2-6k+5)+(2k^2-6k+4)(k-3)=2k^3-10k^2+16k-7
    \]
    vertices. But, for $k\ge 3$,
    \[
        2k^3-10k^2+16k-7<2k^3-10k^2+17k-9,
    \]
    contradicting the choice of $\cP'$. Therefore~\ref{sublem: Q has many vertices} holds.\\

    We then perform all $3$-sums in $\cQ$ corresponding to nonseparating edges, in an arbitrary order. Performing one such $3$-sum does not affect the status of the other internal edges. After all such $3$-sums are performed, we obtain a path decomposition $\cR$ of $M$; see Figure~\ref{fig: the second step}. We next prove the following.

    \begin{sublemma}\label{sublem: R has many vertices}
        The path decomposition $\cR$ has at least $2k$ vertices.
    \end{sublemma}
    
    \begin{figure}[htb]
        \centering
        \resizebox{11cm}{!}{\input{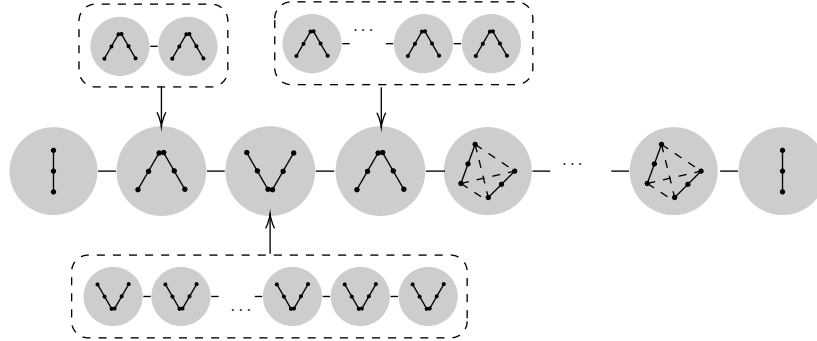}}%
        \caption{Performing all $3$-sums in $\cQ$ corresponding to nonseparating edges.}
        \label{fig: the second step}
    \end{figure}
    
    Suppose that $\cR$ has at most $2k-1$ vertices. As Figure~\ref{fig: the second step} shows, each internal vertex of $\cR$ corresponds to a set of internal vertices in $\cQ$. Since $M$ has no $2$-connected graphic flat of rank at least $k$, Lemma~\ref{lem: many intersecting} implies that each such set contains at most $k-2$ internal vertices of $\cQ$. Hence $\cQ$ has at most
    \[
        2+(2k-3)(k-2)=2k^2-7k+8
    \]
    vertices. But, for $k\ge 3$,
    \[
        2k^2-7k+8<2k^2-6k+6,
    \]
    contradicting~\ref{sublem: Q has many vertices}. Therefore~\ref{sublem: R has many vertices} holds.\\

    If $\cR$ has more than $2k$ vertices, then we perform some $3$-sums to reduce the number of vertices to exactly $2k$, obtaining a new path decomposition $\cR'$ with vertices $R_0,R_1,\dots,R_{2k-1}$ in this order. This reduction creates neither a nonseparating edge nor an internal vertex in which the distinguished triangles are parallel. Let $U$ and $V$ be two adjacent internal vertices of $\cR'$. Since $R'$ has no nonseparating edge, either the edge $UV$ is separating, or, in at least one of $U$ and $V$, the two distinguished triangles are skew. In either case, performing the $3$-sum corresponding to $UV$ produces a new internal vertex $U\oplus_3 V$ in which the two distinguished triangles are skew; see Figure~\ref{fig: pair-up}.

    \begin{figure}[htb]
        \centering
        \resizebox{11cm}{!}{\input{figures/fig_pair-up}}%
        \caption{Pairing adjacent internal vertices of $\cR'$.}
        \label{fig: pair-up}
    \end{figure}
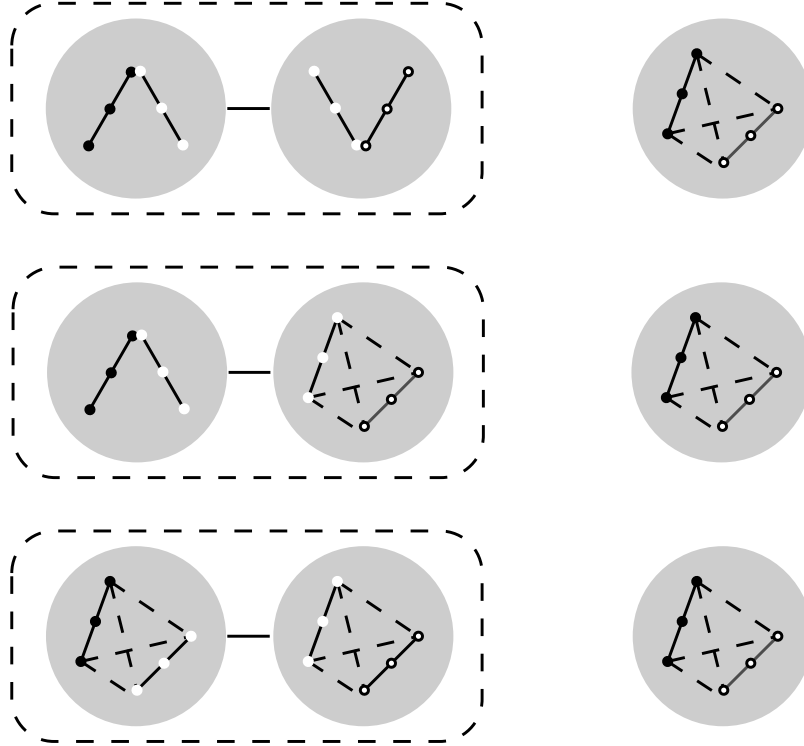

    We then perform the $3$-sum between $R_i$ and $R_{i+1}$ for each $i\in\{1,3,5,\dots,2k-3\}$. Let $\cS$ be the resulting path decomposition of $M$, with vertices $S_0,S_1,\dots,S_k$. By construction, in every internal vertex of $\cS$, the two distinguished triangles are skew.
    
    For each $i\in\{1,2,\dots,k\}$, let $T_i=E(S_{i-1})\cap E(S_i)$. Choose $e\in E(S_0)-T_1$ and $f\in E(S_k)-T_k$. By Proposition~\ref{prop: distance across 3-sums},
    \begin{align*}
        \dist_M(e,f)
        &\ge \dist_{S_0}(e,T_1)+\dist_{S_k}(f,T_k)
            +\sum_{i=1}^{k-1}\dist_{S_i}(T_i,T_{i+1})-2k \\
        &\ge 4+3(k-1)-2k \\
        &= k+1.
    \end{align*}
    Thus the diameter of $M$ is at least $k+1$. By Lemma~\ref{lem: diameter implies graphic flat}, the matroid $M$ has a $2$-connected graphic flat of rank at least $k$, a contradiction. This completes the proof of Theorem~\ref{thm: when Delta-tree has a long path}.    
\end{proof}

A {\it star} is a graph isomorphic to $K_{1,t}$ for some $t\geq 0$. If $t=0$, then the unique vertex is called the {\it center}. If $t=1$, then either vertex may be designated as the {\it center}. If $t\geq 2$, then the {\it center} is the unique vertex of degree $t$. The next result concerns decompositions whose underlying tree is a star.

\begin{lemma}\label{lem: a star with many parallel triangles}
    Let $k$ be a positive integer. Let $M$ be a $3$-connected regular matroid, and let $\cS$ be a tree decomposition of $M$ such that $\cS\cong K_{1,k-1}$. Let $C$ be the center of $\cS$. Suppose that each vertex of $\cS$ is labeled by a matroid of rank at least three, and that the triangles of $C$ involved in $3$-sums are pairwise parallel. Then $M$ has a $2$-connected graphic flat of rank at least $k$.
\end{lemma}

\begin{proof}
    Let $N_1,N_2,\dots,N_{k-1}$ be the vertices of $\cS$ other than $C$, and let
    $T_i=E(N_i)\cap E(C)$ for each $i\in\{1,2,\dots,k-1\}$. Choose an element
    $t_1\in T_1$. For each $i\in\{2,3,\dots,k-1\}$, let $t_i$ be the element of
    $T_i$ that is parallel to $t_1$ in $C$.

    By Lemma~\ref{lem: circuit contains only one in its closure}, for each
    $i\in\{1,2,\dots,k-1\}$, there is a minimum-sized circuit $C_i$ of $N_i$ such that $t_i\in C_i$ and
    \[
        \cl_{N_i}(C_i)\cap T_i=\{t_i\}.
    \]
    Furthermore, there is a minimum-sized circuit $C_0$ of $C$ such that
    $t_1\in C_0$ and
    \[
        \cl_C(C_0)\cap (T_1\cup T_2\cup\dots\cup T_{k-1})
        =\{t_1,t_2,\dots,t_{k-1}\}.
    \]

    For each $i\in\{1,2,\dots,k-1\}$, let $F_i=\cl_{N_i}(C_i)$, and let
    $F_0=\cl_C(C_0)$. By Proposition~\ref{prop: spanning circuit containing e}, the
    restrictions $C|F_0,N_1|F_1,\dots,N_{k-1}|F_{k-1}$ are $2$-connected graphic
    flats of $C,N_1,\dots,N_{k-1}$, respectively. Moreover,
    \[
        (F_0\cup F_1\cup\dots\cup F_{k-1})-\{t_1,t_2,\dots,t_{k-1}\}
    \]
    is a flat $F$ of $M$, and, by Lemma~\ref{lem: flats in 3-sum}, the restriction
    $M|F$ can be obtained by repeatedly taking $2$-sums of $C|F_0$ with
    $N_1|F_1,N_2|F_2,\dots,N_{k-1}|F_{k-1}$.

    Since $M$ is $3$-connected, at most one of the circuits
    $C_0,C_1,\dots,C_{k-1}$ has size two. Therefore
    \[
        r(M|F)=r(C|F_0)+\sum_{i=1}^{k-1}r(N_i|F_i)-(k-1)\geq k.
    \]
    Thus $F$ is a $2$-connected graphic flat of $M$ with rank at least $k$.
\end{proof}

From now on, a decomposition tree whose underlying tree is a star will be called a {\it star decomposition}. We conclude this subsection with the following result.

\begin{theorem}\label{thm: large graphic flat or a high-rank vertex}
    For every pair of positive integers $k$ and $\rho$, there is an integer
    $f_{\ref{thm: large graphic flat or a high-rank vertex}}(k,\rho)$ such that every $3$-connected regular matroid $M$ with $r(M)\geq f_{\ref{thm: large graphic flat or a high-rank vertex}}(k,\rho)$ has
    \begin{enumerate}
        \item[(\romannum{1})] a $2$-connected graphic flat of rank at least $k$; or
        \item[(\romannum{2})] a $\triangle$-tree decomposition $\cT$ in which every vertex label has rank at least three and some vertex label has rank at least $\rho$.
    \end{enumerate}
\end{theorem}

Before proving this theorem, we recall an elementary property of binary matroids; see, for example, \cite[Proposition~6.1.4]{Oxl11}.

\begin{proposition}\label{prop: max number of rank-2 flats}
    For every positive integer $r$, a rank-$r$ binary matroid has at most $\frac{(2^r-1)(2^{r-1}-1)}{3}$ rank-$2$ flats.
\end{proposition}

\begin{proof}[Proof of Theorem~\ref{thm: large graphic flat or a high-rank vertex}]
    The result holds trivially for $k=1$ or $2$, so we may assume that $k\geq 3$. Let
    \[
        d=(k-2)\frac{(2^{\rho-1}-1)(2^{\rho-2}-1)}{3}+1
    \]
    and let
    \[
        v=f_{\ref{lem: long path or high degree vertex}}
        \left(d,f_{\ref{thm: when Delta-tree has a long path}}(k)\right).
    \]
    We show that one may take
    \[
        f_{\ref{thm: large graphic flat or a high-rank vertex}}(k,\rho)
        =
        f_{\ref{lem: large delta-tree or large vertex}}(v,\rho).
    \]
    Let $M$ be a $3$-connected regular matroid with
    \[
        r(M)\geq f_{\ref{thm: large graphic flat or a high-rank vertex}}(k,\rho).
    \]
    Suppose that both (\romannum{1}) and (\romannum{2}) fail. By the choice of $f_{\ref{thm: large graphic flat or a high-rank vertex}}(k,\rho)$, Lemma~\ref{lem: large delta-tree or large vertex} implies that $M$ has a $\triangle$-tree $\cT$ such that every vertex of $\cT$ is labeled by a matroid of rank at least three, and $\cT$ has at least $v$ vertices.

    If $\cT$ has a path of length at least $f_{\ref{thm: when Delta-tree has a long path}}(k)$, then Theorem~\ref{thm: when Delta-tree has a long path} implies that $M$ has a $2$-connected graphic flat of rank at least $k$, contradicting the assumption that (\romannum{1}) fails. Therefore, $\cT$ has a vertex $J$ of degree at least $d$.

    Perform all $3$-sums corresponding to edges of $\cT$ that are not incident
    with $J$. This gives a star decomposition of $M$ with center $J$. Since
    (\romannum{2}) fails, we have $r(J)<\rho$. By Proposition~\ref{prop: max number of rank-2 flats} and the choice of $d$, the matroid $J$ has a collection of at least $k-1$ pairwise parallel triangles that are involved in $3$-sums. Lemma~\ref{lem: a star with many parallel triangles}
    implies that $M$ has a $2$-connected graphic flat of rank at least $k$,
    contradicting the assumption that (\romannum{1}) fails. This completes the proof of Theorem~\ref{thm: large graphic flat or a high-rank vertex}.
\end{proof}

\subsection{A large-rank vertex in a $\triangle$-tree}

We now turn to $\triangle$-trees with a large-rank vertex. First we recall the following result.

\begin{proposition}{\cite[(4.3)]{Seymour1980}}\label{prop: possible parallel on the common triangle}
    Suppose that $M$ is the $3$-sum of $M_1$ and $M_2$, and that $M$ is $3$-connected. If $(Y_1,Y_2)$ is a $2$-separation of $M_1$, then some $Y_i$ is a parallel pair $\{x,z\}$ of $M_1$, where $x\in E(M_1)-E(M_2)$ and $z\in E(M_1)\cap E(M_2)$.
\end{proposition}

A {\it quasi-$\triangle$-star} of a $3$-connected matroid $M$ is a matroid-labeled star $\cS$ satisfying all of the following:
\begin{enumerate}
    \item[(\romannum{1})] every edge of $\cS$ corresponds to a quasi-$3$-sum of its two incident vertex labels;
    \item[(\romannum{2})] the center of $\cS$ is labeled by an almost $3$-connected matroid that is graphic or cographic and has rank at least three;
    \item[(\romannum{3})] every non-center vertex of $\cS$ is labeled by a $3$-connected matroid of rank at least three; and
    \item[(\romannum{4})] $M$ can be obtained from $\cS$ by performing all quasi-$3$-sums corresponding to edges of $\cS$.
\end{enumerate}

\begin{proposition}\label{prop: quasi-delta-star}
    Let $\cT$ be a $\triangle$-tree decomposition of a $3$-connected regular matroid $M$ such that every vertex of $\cT$ is labeled by a matroid of rank at least three. Let $J$ be a vertex label of $\cT$. Then $M$ has a quasi-$\triangle$-star decomposition whose center label is a matroid $J'$ for which $\si(J')=\si(J)$.
\end{proposition}

\begin{proof}
    Perform all $3$-sums corresponding to edges of $\cT$ that are not incident with $J$. This gives a star decomposition $\cS$ of $M$ whose center is labeled by $J$. Since $\cT$ is a $\triangle$-tree, the matroid $J$ is graphic or cographic.

    Let $N_1,N_2,\dots,N_n$ be the labels of the non-center vertices of $\cS$. For each $i\in\{1,2,\dots,n\}$, let $T_i$ be the common triangle of $N_i$ and $J$. By Proposition~\ref{prop: possible parallel on the common triangle}, each $N_i$ is almost $3$-connected, and all parallel elements of $N_i$ are parallel to elements of $T_i$.

    Suppose that $e\in E(N_i)-T_i$ and $e$ is parallel in $N_i$ to an element $t\in T_i$. We delete $e$ from $N_i$ and add $e$ to $J$ as an element parallel to $t$. Repeating this operation for every such element and every $i\in\{1,2,\dots,n\}$ gives a matroid-labeled star $\cS'$. Let $J'$ be the label of the center of $\cS'$, and let $N_1',N_2',\dots,N_n'$ be the labels of its non-center vertices.

    By construction, each $N_i'$ is $3$-connected, while $J'$ is almost $3$-connected. Moreover, $\si(J')=\si(J)$, and $J'$ is still graphic or cographic. The ordinary $3$-sum between $J'$ and some $N_i'$ may no longer be defined, since deleting parallel elements from $N_i$ may reduce the size of its ground set. However, the triangle $T_i$ remains a triangle of $N_i'$, and hence the quasi-$3$-sum of $J'$ and $N_i'$ across $T_i$ is well defined.

    Finally, moving such parallel elements from the non-center labels to the center label does not change the matroid obtained after performing the corresponding sums. Thus $M$ is obtained from $\cS'$ by performing all quasi-$3$-sums corresponding to the edges of $\cS'$. Hence $\cS'$ is a quasi-$\triangle$-star decomposition of $M$ with center label $J'$ such that $\si(J')=\si(J)$.
\end{proof}

\begin{theorem}\label{thm: high-rank vertex in Delta-tree}
    For every positive integer $k$, there is an integer $f_{\ref{thm: high-rank vertex in Delta-tree}}(k)$ such that if $M$ is a $3$-connected regular matroid with a $\triangle$-tree decomposition $\cT$ in which every vertex label has rank at least three and some vertex label has rank at least $f_{\ref{thm: high-rank vertex in Delta-tree}}(k)$, then $M$ has a $2$-connected graphic flat of rank at least $k$.
\end{theorem}

\begin{proof}
    The statement holds trivially for $k=1$ and $k=2$, so we may assume that $k\geq 3$. We show that one may take
    \[
        f_{\ref{thm: high-rank vertex in Delta-tree}}(k)
        =
        f_{\ref{cor: graphic and cographic in one}}(k).
    \]
    Let $J$ be a vertex label of $\cT$ such that
    \[
        r(J)\geq f_{\ref{thm: high-rank vertex in Delta-tree}}(k),
    \]
    and suppose that $M$ has no $2$-connected graphic flat of rank at least $k$. By Proposition~\ref{prop: quasi-delta-star}, $M$ has a quasi-$\triangle$-star decomposition $\cS$ whose center is labeled by a matroid $J'$ such that $\si(J')=\si(J)$. Let $N_1,N_2,\dots,N_n$ be the labels of the non-center vertices of $\cS$, and let $T_i$ be the common triangle of $J'$ and $N_i$.

    By the choice of $f_{\ref{thm: high-rank vertex in Delta-tree}}(k)$, Corollary~\ref{cor: graphic and cographic in one} implies that $J'$ has a flat $F_0$ such that $J'|F_0=M(H)$ for some graph $H$. Moreover, the simplification of $H$ is one of the following:
    \begin{enumerate}
        \item[(\romannum{1})] $K_{k+1}$;
        \item[(\romannum{2})] a subdivision of $K_{2,k+1}$;
        \item[(\romannum{3})] a graph obtained from a subdivision of $K_{2,k+1}$ by adding an edge joining the two vertices of degree $k+1$; or
        \item[(\romannum{4})] a clean ladder with at least $k+1$ vertices.
    \end{enumerate}

    We identify the elements of $F_0$ with the edges of $H$. Choose a set $R\subseteq F_0$, and let $H_R$ be the subgraph of $H$ with edge set $R$ and no isolated vertices, as follows:
    \begin{enumerate}
        \item[(\romannum{1})] if the simplification of $H$ is $K_{k+1}$, choose $R$ so that $H_R$ is a cycle with $k+1$ edges;
        \item[(\romannum{2})] if the simplification of $H$ is a subdivision of $K_{2,k+1}$, choose $R$ so that $H_R$ is a subdivision of $K_{2,k+1}$;
        \item[(\romannum{3})] if the simplification of $H$ is obtained from a subdivision of $K_{2,k+1}$ by adding an edge joining the two vertices of degree $k+1$, choose $R$ so that $H_R$ is the underlying subdivision of $K_{2,k+1}$; and
        \item[(\romannum{4})] if the simplification of $H$ is a clean ladder with at least $k+1$ vertices, choose $R$ so that $H_R$ is a cycle with at least $k+1$ edges.
    \end{enumerate}
    In each case, $H_R$ is $2$-connected, $r(M(H_R))\geq k$, and $J'|R$ has no triangle.

    For each $i\in\{1,2,\dots,n\}$, we now choose a flat $F_i$ of $N_i$. Since $F_0$ is a flat of $J'$ and $T_i$ is a triangle, the set $F_0\cap T_i$ has size $0$, $1$, or $3$. If $F_0\cap T_i=\emptyset$, let $F_i=\emptyset$. Suppose next that $F_0\cap T_i=\{t_i\}$ for some $t_i\in T_i$. Let $F_i$ be a minimum-sized circuit of $N_i$ such that
    \[
        t_i\in F_i
        \quad\text{and}\quad
        \cl_{N_i}(F_i)\cap T_i=\{t_i\}.
    \]
    Since $N_i$ is $3$-connected and $r(N_i)\geq 3$, such a circuit exists by Lemma~\ref{lem: circuit contains only one in its closure}. By the minimality of $F_i$, the circuit $F_i$ is a flat of $N_i$. Moreover, $|F_i|\geq 3$.

    It remains to consider the case in which
    \[
        F_0\cap T_i=T_i=\{a_i,b_i,c_i\}.
    \]
    Since $J'|R$ has no triangle, the set $R\cap T_i$ has size at most two. If $R\cap T_i=\emptyset$, let $F_i=T_i$. If $|R\cap T_i|=1$, we may assume that $R\cap T_i=\{a_i\}$. Then, by Theorem~\ref{thm: a graph that keeps a triangle}, the matroid $N_i$ has a flat $F_i$ such that $N_i|F_i=M(W_i)$ for some enlarged wheel $W_i$ in which $a_i$ is a rim edge and $b_i,c_i$ are spokes. If $|R\cap T_i|=2$, we may assume that $R\cap T_i=\{a_i,b_i\}$. Then, again by Theorem~\ref{thm: a graph that keeps a triangle}, the matroid $N_i$ has a flat $F_i$ such that $N_i|F_i=M(W_i)$ for some enlarged wheel $W_i$ in which $c_i$ is a rim edge and $a_i,b_i$ are spokes. We note that, as $N_i$ is a $3$-connected matroid of rank at least three, the rim of $W_i$ has at least three edges.

    Let $F$ be the set
    \[
        (F_0\cup F_1\cup F_2\cup\dots\cup F_n)-(T_1\cup T_2\cup\dots\cup T_n).
    \]
    By the choices of the sets $F_i$, the set $F$ is a flat of $M$. Moreover, by Lemma~\ref{lem: flats in 3-sum}, the restriction $M|F$ is obtained from $J'|F_0$ and the matroids $N_i|F_i$ by applying the corresponding $1$-sums, $2$-sums, and quasi-$3$-sums. Since $J'|F_0$ and all the matroids $N_i|F_i$ are graphic, it follows that $M|F$ is graphic.

    It remains to find a $2$-connected flat of rank at least $k$ in $M|F$. We argue graphically, tracking the subgraph $H_R$ as the sums with the matroids $N_i|F_i$ are performed.

    A $1$-sum with $N_i|F_i$ does not affect the part of the graph corresponding to $R$. Now suppose that the sum with $N_i|F_i$ is a $2$-sum with basepoint $t_i$. Then $N_i|F_i$ is a circuit containing $t_i$. If $t_i\notin R$, this sum does not affect $H_R$. If $t_i\in R$, then the $2$-sum replaces the edge $t_i$ of $H_R$ by a path, namely the path corresponding to $F_i-\{t_i\}$. Thus a $2$-sum with such a circuit only subdivides an edge of $H_R$, and hence preserves $2$-connectivity.

    It remains to consider the quasi-$3$-sums. By Proposition~\ref{prop: graph of quasi-3-sum}, the quasi-$3$-sum of two graphic matroids across a common triangle is the cycle matroid of the corresponding graph quasi-$3$-sum. Thus, when $F_0\cap T_i=T_i$, we may view the operation as gluing the graph $H$ to the enlarged wheel $W_i$ along the triangle $T_i$ and then deleting the three edges of $T_i$.

    If $R\cap T_i=\emptyset$, then the quasi-$3$-sum does not change the subgraph that we are tracking. Suppose next that $|R\cap T_i|=1$. We may assume that $R\cap T_i=\{a_i\}$. By the choice of $F_i$, the edge $a_i$ is a rim edge of $W_i$, while $b_i$ and $c_i$ are spokes. In the graph obtained after the quasi-$3$-sum, we focus on the subgraph obtained from $H_R$ by deleting $a_i$ and adding the remaining rim path of $W_i$. Thus the edge $a_i$ of $H_R$ is replaced by a path with the same ends.

    Finally, suppose that $|R\cap T_i|=2$. We may assume that $R\cap T_i=\{a_i,b_i\}$. By the choice of $F_i$, the edge $c_i$ is a rim edge of $W_i$, while $a_i$ and $b_i$ are spokes. In the graph obtained after the quasi-$3$-sum, we pass to the subgraph obtained from $H_R$ by deleting the two-edge path using $a_i$ and $b_i$ and adding the remaining rim path of $W_i$. Thus the tracked subgraph is obtained by replacing that two-edge path by a path with the same ends. Since $N_i$ is $3$-connected of rank at least three, the wheel $W_i$ is simple and its rim has at least three edges, so this replacement path has at least two edges. For our choice of $R$, the original two-edge path is either a segment of a cycle with at least $k+1$ edges, or a branch joining the two degree-$(k+1)$ vertices in the subdivision of $K_{2,k+1}$. Replacing the two-edge path using $\{a_i,b_i\}$ by another path with the same ends and at least two edges preserves $2$-connectivity and does not decrease the number of vertices.

    It follows that, after all sums are performed, there is a graph $G$ with $M|F=M(G)$ such that $G$ contains a $2$-connected subgraph $H'$ obtained from $H_R$ by subdividing edges and replacing some two-edge paths by paths with the same ends. In particular,
    \[
        r(M(H'))\geq r(M(H_R))\geq k.
    \]
    Let $F'=\cl_{M|F}(E(H'))$. Since $H'$ is $2$-connected and $M|F$ is graphic, the set $F'$ is a $2$-connected graphic flat of $M|F$. Moreover, $r(M|F')=r(M(H'))\geq k$. Since $F$ is a flat of $M$, the flat $F'$ of $M|F$ is also a flat of $M$. Thus $M$ has a $2$-connected graphic flat of rank at least $k$, a contradiction. This completes the proof.
\end{proof}

We now have all the ingredients needed to prove the main result of this section.

\begin{proof}[Proof of Theorem~\ref{thm: large 3-connected (main)}.]
    We show that one may take
    \[
        f_{\ref{thm: large 3-connected (main)}}(k)
        =
        f_{\ref{thm: large graphic flat or a high-rank vertex}}
        \bigl(k,f_{\ref{thm: high-rank vertex in Delta-tree}}(k)\bigr).
    \]
    Suppose that $M$ has no $2$-connected graphic flat of rank at least $k$. By Theorem~\ref{thm: large graphic flat or a high-rank vertex}, the matroid $M$ has a $\triangle$-tree decomposition $\cT$ in which every vertex label has rank at least three and some vertex label has rank at least $f_{\ref{thm: high-rank vertex in Delta-tree}}(k)$. By Theorem~\ref{thm: high-rank vertex in Delta-tree}, the matroid $M$ has a $2$-connected graphic flat of rank at least $k$, a contradiction.
\end{proof}

\section{Proofs of the main theorems}

In this section, we prove the main results of the paper.

\begin{proof}[Proof of Theorem~\ref{thm:large_graphic_flat}.]
    First suppose that $k\leq 2$. In this case, take
    \[
        f_{\ref{thm:large_graphic_flat}}(k)=2.
    \]
    Let $M$ be a $2$-connected regular matroid with $r(M)\geq 2$. Since $M$ is connected and has rank at least two, it has a circuit of size at least three. Hence $M$ has two elements $e$ and $f$ such that $\dist(e,f)\geq 3$. Thus the diameter of $M$ is at least three. By Lemma~\ref{lem: diameter implies graphic flat}, the matroid $M$ has a $2$-connected graphic flat of rank at least two.

    We may now assume that $k\geq 3$. We show that one may take
    \[
        f_{\ref{thm:large_graphic_flat}}(k)
        =
        f_{\ref{thm: high-rank 3-connected vertex or high-rank graphic flat}}
        \bigl(k,f_{\ref{thm: large 3-connected (main)}}(k)\bigr).
    \]
    Let $M$ be a $2$-connected regular matroid such that $r(M)\geq f_{\ref{thm:large_graphic_flat}}(k)$. Suppose that $M$ has no $2$-connected graphic flat of rank at least $k$. By Theorem~\ref{thm: high-rank 3-connected vertex or high-rank graphic flat}, the canonical tree decomposition $\cT$ of $M$ has a vertex labeled by a $3$-connected matroid $N$ such that $r(N)\geq f_{\ref{thm: large 3-connected (main)}}(k)$. By Theorem~\ref{thm: large 3-connected (main)}, the matroid $N$ has a $2$-connected graphic flat $F$ of rank at least $k$. By Lemma~\ref{lem: cleaning 2-connected leaves}, the matroid $M$ has a $2$-connected graphic flat $F'$ such that
    \[
        r(M|F')\geq r(N|F)\geq k.
    \]
    This contradiction completes the proof.
\end{proof}

We conclude by obtaining Theorem~\ref{thm: explicit characterization} as a consequence of Theorem~\ref{thm:large_graphic_flat}.

\begin{proof}[Proof of Theorem~\ref{thm: explicit characterization}.]
    We show that one may take
    \[
        f_{\ref{thm: explicit characterization}}(k)
        =
        f_{\ref{thm:large_graphic_flat}}
        (f_{\ref{cor: graphic case}}(k)).
    \]
    Let $M$ be a simple $2$-connected regular matroid such that $r(M)\geq f_{\ref{thm: explicit characterization}}(k)$. By Theorem~\ref{thm:large_graphic_flat}, the matroid $M$ has a $2$-connected graphic flat $F$ such that $r(M|F)\geq f_{\ref{cor: graphic case}}(k)$. Since $M$ is simple, every restriction of $M$ is also simple. By Corollary~\ref{cor: graphic case}, the matroid $M|F$ has a $2$-connected graphic flat $F'$ such that $M|F'=M(H)$ for some graph $H$ as described in Theorem~\ref{thm: explicit characterization}.
\end{proof}

\section*{Acknowledgments}
The author thanks James Oxley for suggesting the topic of this paper and for his guidance throughout the project. His comments substantially improved the exposition and the arguments.

\input{reference}
\end{document}

%% file: figures/fig_clean_laddar.tex
\begin{tikzpicture}[
    scale=0.55,
    every node/.style={circle, fill=black, inner sep=1.5pt},
    edge/.style={line width=1pt}
]

\foreach \x in {0,2.2,3.4,5.5,6.8,7.9,9.8,10.9,12.0,12.9,15.9}
    \node (t\x) at (\x,2) {};

\foreach \x in {0,0.6,2.0,2.7,5.5,6.8,7.9,9.0,10.0,10.9,11.7,12.6,13.5,15.1,15.9}
    \node (b\x) at (\x,0) {};

\draw[edge] (0,0) -- (15.9,0) -- (15.9,2) -- (0,2) -- cycle;

\draw[edge] (0.6,0) -- (2.2,2);
\draw[edge] (2.0,0) -- (2.2,2);
\draw[edge] (2.2,2) -- (2.7,0);
\draw[edge] (3.4,2) -- (2.7,0);

\draw[edge] (5.5,2) -- (6.8,0);
\draw[edge] (6.8,2) -- (5.5,0);

\draw[edge] (7.9,0) -- (7.9,2);

\draw[edge] (9.8,2) -- (9.0,0);
\draw[edge] (9.8,2) -- (10.0,0);
\draw[edge] (10.9,2) -- (10.0,0);
\draw[edge] (10.9,2) -- (10.9,0);

\draw[edge] (12.0,2) -- (11.7,0);
\draw[edge] (12.0,2) -- (12.6,0);
\draw[edge] (12.9,2) -- (12.6,0);
\draw[edge] (12.9,2) -- (13.5,0);

\draw[edge] (15.1,0) -- (15.9,2);

\end{tikzpicture}

%% file: figures/fig_fan.tex
\tikzset{every picture/.style={line width=0.75pt}} 

\begin{tikzpicture}[x=0.75pt,y=0.75pt,yscale=-1,xscale=1]

\draw    (133.22,50.81) -- (64.38,137.22) ;
\draw  [fill={rgb, 255:red, 0; green, 0; blue, 0 }  ,fill opacity=1 ] (130.03,50.81) .. controls (130.03,49.04) and (131.46,47.61) .. (133.22,47.61) .. controls (134.98,47.61) and (136.41,49.04) .. (136.41,50.81) .. controls (136.41,52.57) and (134.98,54) .. (133.22,54) .. controls (131.46,54) and (130.03,52.57) .. (130.03,50.81) -- cycle ;
\draw  [fill={rgb, 255:red, 0; green, 0; blue, 0 }  ,fill opacity=1 ] (61.19,137.22) .. controls (61.19,135.46) and (62.62,134.03) .. (64.38,134.03) .. controls (66.14,134.03) and (67.57,135.46) .. (67.57,137.22) .. controls (67.57,138.99) and (66.14,140.42) .. (64.38,140.42) .. controls (62.62,140.42) and (61.19,138.99) .. (61.19,137.22) -- cycle ;
\draw  [fill={rgb, 255:red, 0; green, 0; blue, 0 }  ,fill opacity=1 ] (88.73,137.22) .. controls (88.73,135.46) and (90.16,134.03) .. (91.92,134.03) .. controls (93.68,134.03) and (95.11,135.46) .. (95.11,137.22) .. controls (95.11,138.99) and (93.68,140.42) .. (91.92,140.42) .. controls (90.16,140.42) and (88.73,138.99) .. (88.73,137.22) -- cycle ;
\draw  [fill={rgb, 255:red, 0; green, 0; blue, 0 }  ,fill opacity=1 ] (116.27,137.22) .. controls (116.27,135.46) and (117.7,134.03) .. (119.46,134.03) .. controls (121.22,134.03) and (122.65,135.46) .. (122.65,137.22) .. controls (122.65,138.99) and (121.22,140.42) .. (119.46,140.42) .. controls (117.7,140.42) and (116.27,138.99) .. (116.27,137.22) -- cycle ;
\draw  [fill={rgb, 255:red, 0; green, 0; blue, 0 }  ,fill opacity=1 ] (143.81,137.22) .. controls (143.81,135.46) and (145.24,134.03) .. (147,134.03) .. controls (148.76,134.03) and (150.19,135.46) .. (150.19,137.22) .. controls (150.19,138.99) and (148.76,140.42) .. (147,140.42) .. controls (145.24,140.42) and (143.81,138.99) .. (143.81,137.22) -- cycle ;
\draw  [fill={rgb, 255:red, 0; green, 0; blue, 0 }  ,fill opacity=1 ] (171.35,137.22) .. controls (171.35,135.46) and (172.78,134.03) .. (174.54,134.03) .. controls (176.3,134.03) and (177.73,135.46) .. (177.73,137.22) .. controls (177.73,138.99) and (176.3,140.42) .. (174.54,140.42) .. controls (172.78,140.42) and (171.35,138.99) .. (171.35,137.22) -- cycle ;
\draw  [fill={rgb, 255:red, 0; green, 0; blue, 0 }  ,fill opacity=1 ] (198.87,137.22) .. controls (198.87,135.46) and (200.3,134.03) .. (202.06,134.03) .. controls (203.83,134.03) and (205.26,135.46) .. (205.26,137.22) .. controls (205.26,138.99) and (203.83,140.42) .. (202.06,140.42) .. controls (200.3,140.42) and (198.87,138.99) .. (198.87,137.22) -- cycle ;

\draw    (133.22,50.81) -- (91.92,137.22) ;
\draw    (133.22,50.81) -- (119.46,137.22) ;
\draw    (133.22,50.81) -- (147,137.22) ;
\draw    (133.22,50.81) -- (174.54,137.22) ;
\draw    (133.22,50.81) -- (202.06,137.22) ;
\draw    (64.38,137.22) -- (91.92,137.22) ;
\draw    (91.92,137.22) -- (119.46,137.22) ;
\draw    (147,137.22) -- (174.54,137.22) ;

\draw    (175.54,137.22) -- (203.06,137.22) ;

\draw    (318.63,50.85) -- (249.78,137.26) ;
\draw  [fill={rgb, 255:red, 0; green, 0; blue, 0 }  ,fill opacity=1 ] (315.43,50.85) .. controls (315.43,49.08) and (316.86,47.65) .. (318.63,47.65) .. controls (320.39,47.65) and (321.82,49.08) .. (321.82,50.85) .. controls (321.82,52.61) and (320.39,54.04) .. (318.63,54.04) .. controls (316.86,54.04) and (315.43,52.61) .. (315.43,50.85) -- cycle ;
\draw  [fill={rgb, 255:red, 0; green, 0; blue, 0 }  ,fill opacity=1 ] (246.59,137.26) .. controls (246.59,135.5) and (248.02,134.07) .. (249.78,134.07) .. controls (251.55,134.07) and (252.97,135.5) .. (252.97,137.26) .. controls (252.97,139.03) and (251.55,140.46) .. (249.78,140.46) .. controls (248.02,140.46) and (246.59,139.03) .. (246.59,137.26) -- cycle ;
\draw  [fill={rgb, 255:red, 0; green, 0; blue, 0 }  ,fill opacity=1 ] (274.13,137.26) .. controls (274.13,135.5) and (275.56,134.07) .. (277.32,134.07) .. controls (279.09,134.07) and (280.51,135.5) .. (280.51,137.26) .. controls (280.51,139.03) and (279.09,140.46) .. (277.32,140.46) .. controls (275.56,140.46) and (274.13,139.03) .. (274.13,137.26) -- cycle ;
\draw  [fill={rgb, 255:red, 0; green, 0; blue, 0 }  ,fill opacity=1 ] (301.67,137.26) .. controls (301.67,135.5) and (303.1,134.07) .. (304.86,134.07) .. controls (306.63,134.07) and (308.05,135.5) .. (308.05,137.26) .. controls (308.05,139.03) and (306.63,140.46) .. (304.86,140.46) .. controls (303.1,140.46) and (301.67,139.03) .. (301.67,137.26) -- cycle ;
\draw  [fill={rgb, 255:red, 0; green, 0; blue, 0 }  ,fill opacity=1 ] (329.21,137.26) .. controls (329.21,135.5) and (330.64,134.07) .. (332.4,134.07) .. controls (334.17,134.07) and (335.59,135.5) .. (335.59,137.26) .. controls (335.59,139.03) and (334.17,140.46) .. (332.4,140.46) .. controls (330.64,140.46) and (329.21,139.03) .. (329.21,137.26) -- cycle ;
\draw  [fill={rgb, 255:red, 0; green, 0; blue, 0 }  ,fill opacity=1 ] (356.75,137.26) .. controls (356.75,135.5) and (358.18,134.07) .. (359.94,134.07) .. controls (361.71,134.07) and (363.13,135.5) .. (363.13,137.26) .. controls (363.13,139.03) and (361.71,140.46) .. (359.94,140.46) .. controls (358.18,140.46) and (356.75,139.03) .. (356.75,137.26) -- cycle ;
\draw  [fill={rgb, 255:red, 0; green, 0; blue, 0 }  ,fill opacity=1 ] (384.28,137.26) .. controls (384.28,135.5) and (385.71,134.07) .. (387.47,134.07) .. controls (389.23,134.07) and (390.66,135.5) .. (390.66,137.26) .. controls (390.66,139.03) and (389.23,140.46) .. (387.47,140.46) .. controls (385.71,140.46) and (384.28,139.03) .. (384.28,137.26) -- cycle ;

\draw    (318.63,50.85) -- (277.32,137.26) ;
\draw    (318.63,50.85) -- (304.86,137.26) ;
\draw    (318.63,50.85) -- (332.4,137.26) ;
\draw    (318.63,50.85) -- (359.94,137.26) ;
\draw    (318.63,50.85) -- (387.47,137.26) ;
\draw    (249.78,137.26) -- (277.32,137.26) ;
\draw    (277.32,137.26) -- (304.86,137.26) ;
\draw    (332.4,137.26) -- (359.94,137.26) ;
\draw    (359.94,137.26) -- (387.47,137.26) ;
\draw    (249.78,137.26) .. controls (257.68,91.77) and (288.44,63.21) .. (318.63,50.85) ;
\draw    (387.47,137.26) .. controls (377.79,91.77) and (352.89,65.4) .. (318.63,50.85) ;
\draw    (277.32,137.26) .. controls (284.3,98.85) and (294.55,90.53) .. (318.63,50.85) ;
\draw    (304.86,137.26) .. controls (301.14,126.42) and (300.41,94.19) .. (318.63,50.85) ;
\draw    (359.94,137.26) .. controls (354.6,93.46) and (344.35,94.19) .. (318.63,50.85) ;
\draw    (332.4,137.26) .. controls (338.49,109.57) and (334.1,97.12) .. (318.63,50.85) ;
\draw  [fill={rgb, 255:red, 0; green, 0; blue, 0 }  ,fill opacity=1 ] (457.96,140.76) .. controls (457.96,139) and (459.38,137.57) .. (461.15,137.57) .. controls (462.91,137.57) and (464.34,139) .. (464.34,140.76) .. controls (464.34,142.53) and (462.91,143.96) .. (461.15,143.96) .. controls (459.38,143.96) and (457.96,142.53) .. (457.96,140.76) -- cycle ;
\draw  [fill={rgb, 255:red, 0; green, 0; blue, 0 }  ,fill opacity=1 ] (458.5,54.76) .. controls (458.5,53) and (459.92,51.57) .. (461.69,51.57) .. controls (463.45,51.57) and (464.88,53) .. (464.88,54.76) .. controls (464.88,56.53) and (463.45,57.96) .. (461.69,57.96) .. controls (459.92,57.96) and (458.5,56.53) .. (458.5,54.76) -- cycle ;
\draw    (461.69,54.76) .. controls (411.2,73.14) and (412.2,123.14) .. (461.15,140.76) ;
\draw    (461.69,54.76) .. controls (439.2,82.14) and (441.2,118.14) .. (461.15,140.76) ;
\draw    (461.15,140.76) .. controls (511.63,122.39) and (510.63,72.39) .. (461.69,54.76) ;
\draw    (461.15,140.76) .. controls (483.63,113.39) and (481.63,77.39) .. (461.69,54.76) ;

\draw (117.57,133.59) node [anchor=north west][inner sep=0.75pt]    {$\cdots $};
\draw (128.63,34.72) node [anchor=north west][inner sep=0.75pt]    {$u$};
\draw (55.32,147.97) node [anchor=north west][inner sep=0.75pt]    {$v_{1}$};
\draw (83.59,147.97) node [anchor=north west][inner sep=0.75pt]    {$v_{2}$};
\draw (108.93,147.97) node [anchor=north west][inner sep=0.75pt]    {$v_{3}$};
\draw (132.53,147.97) node [anchor=north west][inner sep=0.75pt]    {$v_{n-2}$};
\draw (163.73,147.97) node [anchor=north west][inner sep=0.75pt]    {$v_{n-1}$};
\draw (116.63,169.94) node [anchor=north west][inner sep=0.75pt]    {$F_{1,1,\dots,1}$};
\draw (199.87,147.97) node [anchor=north west][inner sep=0.75pt]    {$v_{n}$};
\draw (453,93.94) node [anchor=north west][inner sep=0.75pt]    {$\dotsc $};
\draw (453,168.94) node [anchor=north west][inner sep=0.75pt]    {$B_{n}$};
\draw (456.38,35.73) node [anchor=north west][inner sep=0.75pt]    {$u$};
\draw (456.38,150.24) node [anchor=north west][inner sep=0.75pt]    {$v$};
\draw (302.97,133.63) node [anchor=north west][inner sep=0.75pt]    {$\cdots $};
\draw (240.73,148.74) node [anchor=north west][inner sep=0.75pt]    {$v_{1}$};
\draw (269,148.74) node [anchor=north west][inner sep=0.75pt]    {$v_{2}$};
\draw (294.34,148.74) node [anchor=north west][inner sep=0.75pt]    {$v_{3}$};
\draw (317.94,148.74) node [anchor=north west][inner sep=0.75pt]    {$v_{n-2}$};
\draw (349.14,148.74) node [anchor=north west][inner sep=0.75pt]    {$v_{n-1}$};
\draw (384.27,148.74) node [anchor=north west][inner sep=0.75pt]    {$v_{n}$};
\draw (313.3,36.23) node [anchor=north west][inner sep=0.75pt]    {$u$};
\draw (299.3,168.44) node [anchor=north west][inner sep=0.75pt]    {$F_{2,2,\dotsc ,2}$};

\end{tikzpicture}

%% file: figures/fig_extremal_diameter.tex
\begin{tikzpicture}[
    scale=0.95,
    every node/.style={circle, fill=black, inner sep=1.45pt},
    every path/.style={line width=0.5pt},
    mydash/.style={dash pattern=on 3pt off 3pt}
]

\begin{scope}[shift={(0,0)}]
    \coordinate (O) at (0,0);
    \coordinate (a) at (90:0.72);
    \coordinate (b) at (270:0.72);

    \fill[gray!25] (O) circle (1.15);

    \draw (a) -- node[left=10pt, fill=none, inner sep=1pt] {$e$} (b);
    \draw (a) .. controls (-0.45,0.35) and (-0.45,-0.35) .. (b);
    \draw[mydash] (a) .. controls (0.45,0.35) and (0.45,-0.35) .. (b);

    \node at (a) {};
    \node at (b) {};
\end{scope}

\draw (1.15,0) -- (1.65,0);

\begin{scope}[shift={(2.8,0)}]
    \coordinate (O) at (0,0);
    \coordinate (a) at (90:0.82);
    \coordinate (b) at (210:0.82);
    \coordinate (c) at (330:0.82);

    \fill[gray!25] (O) circle (1.15);

    \draw[mydash] (a) -- (b);
    \draw[mydash] (a) -- (c);
    \draw (b) -- (c);

    \node at (a) {};
    \node at (b) {};
    \node at (c) {};
\end{scope}

\draw (3.95,0) -- (4.45,0);

\begin{scope}[shift={(5.6,0)}]
    \coordinate (O) at (0,0);
    \coordinate (a) at (90:0.72);
    \coordinate (b) at (270:0.72);

    \fill[gray!25] (O) circle (1.15);

    \draw (a) -- (b);
    \draw[mydash] (a) .. controls (-0.45,0.35) and (-0.45,-0.35) .. (b);
    \draw[mydash] (a) .. controls (0.45,0.35) and (0.45,-0.35) .. (b);

    \node at (a) {};
    \node at (b) {};
\end{scope}

\draw (6.75,0) -- (7.25,0);

\begin{scope}[shift={(8.4,0)}]
    \coordinate (O) at (0,0);
    \coordinate (a) at (90:0.82);
    \coordinate (b) at (210:0.82);
    \coordinate (c) at (330:0.82);

    \fill[gray!25] (O) circle (1.15);

    \draw[mydash] (a) -- (b);
    \draw[mydash] (a) -- (c);
    \draw (b) -- (c);

    \node at (a) {};
    \node at (b) {};
    \node at (c) {};
\end{scope}

\draw (9.55,0) -- (10.05,0);

\begin{scope}[shift={(11.2,0)}]
    \coordinate (O) at (0,0);
    \coordinate (a) at (90:0.72);
    \coordinate (b) at (270:0.72);

    \fill[gray!25] (O) circle (1.15);

    \draw (a) -- node[right=8pt, fill=none, inner sep=1pt] {$f$} (b);
    \draw[mydash] (a) .. controls (-0.45,0.35) and (-0.45,-0.35) .. (b);
    \draw (a) .. controls (0.45,0.35) and (0.45,-0.35) .. (b);

    \node at (a) {};
    \node at (b) {};
\end{scope}

\end{tikzpicture}

%% file: figures/fig_broken_wheels.tex
\begin{tikzpicture}[
    scale=0.85,
    every node/.style={circle, fill=black, inner sep=2.2pt},
    every path/.style={line width=0.5pt}
]


\begin{scope}[shift={(0,0)}]
    \coordinate (O) at (0,0);
    \coordinate (t) at (90:1.35);
    \coordinate (l) at (180:1.35);
    \coordinate (r) at (0:1.35);
    \coordinate (b) at (270:1.35);
    \coordinate (c) at (0,0);

    \draw (b) arc[start angle=270,end angle=0,radius=1.35];
    \draw[dashed] (r) arc[start angle=0,end angle=-90,radius=1.35];
    \draw (b) arc[start angle=270,end angle=90,radius=1.35];
    \draw (t) -- (c);
    \draw (c) -- (r);

    \node at (t) {};
    \node at (l) {};
    \node at (r) {};
    \node at (b) {};
    \node at (c) {};
\end{scope}

\begin{scope}[shift={(3.7,0)}]
    \coordinate (O) at (0,0);
    \coordinate (t) at (90:1.35);
    \coordinate (l) at (180:1.35);
    \coordinate (r) at (0:1.35);
    \coordinate (b) at (270:1.35);
    \coordinate (c) at (0,0);
    \coordinate (ur) at (45:1.35);
    \coordinate (dl) at (225:1.35);

    \draw (b) arc[start angle=270,end angle=0,radius=1.35];
    \draw[dashed] (r) arc[start angle=0,end angle=-90,radius=1.35];
    \draw (b) arc[start angle=270,end angle=90,radius=1.35];

    \draw (t) -- (c);
    \draw (ur) -- (dl);

    \draw (t) .. controls (-0.35,0.75) and (-0.35,0.35) .. (c);
    \draw (t) .. controls (0.35,0.75) and (0.35,0.35) .. (c);

    \node at (t) {};
    \node at (l) {};
    \node at (r) {};
    \node at (b) {};
    \node at (c) {};
    \node at (ur) {};
    \node at (dl) {};
\end{scope}

\begin{scope}[shift={(7.4,0)}]
    \coordinate (O) at (0,0);
    \coordinate (t) at (90:1.35);
    \coordinate (l) at (180:1.35);
    \coordinate (r) at (0:1.35);
    \coordinate (b) at (270:1.35);
    \coordinate (c) at (0,0);
    \coordinate (ul) at (135:1.35);
    \coordinate (ur) at (45:1.35);
    \coordinate (dl) at (225:1.35);

    \draw (b) arc[start angle=270,end angle=0,radius=1.35];
    \draw[dashed] (r) arc[start angle=0,end angle=-90,radius=1.35];
    \draw (b) arc[start angle=270,end angle=90,radius=1.35];

    \draw (t) -- (b);
    \draw (l) -- (r);
    \draw (c) -- (ul);
    \draw (ur) -- (dl);

    \node at (t) {};
    \node at (l) {};
    \node at (r) {};
    \node at (b) {};
    \node at (c) {};
    \node at (ul) {};
    \node at (ur) {};
    \node at (dl) {};
\end{scope}

\end{tikzpicture}

%% file: figures/fig_relations_of_trianlges.tex
\tikzset{every picture/.style={line width=0.75pt}} 

\begin{tikzpicture}[x=0.75pt,y=0.75pt,yscale=-1,xscale=1]

\draw  [color={rgb, 255:red, 255; green, 255; blue, 255 }  ,draw opacity=1 ][fill={rgb, 255:red, 155; green, 155; blue, 155 }  ,fill opacity=0.5 ] (483.27,105.66) .. controls (483.27,68.47) and (513.42,38.32) .. (550.61,38.32) .. controls (587.8,38.32) and (617.95,68.47) .. (617.95,105.66) .. controls (617.95,142.85) and (587.8,173) .. (550.61,173) .. controls (513.42,173) and (483.27,142.85) .. (483.27,105.66) -- cycle ;
\draw  [dash pattern={on 4.5pt off 4.5pt}]  (531.5,65.36) -- (591.15,105.68) ;
\draw  [dash pattern={on 4.5pt off 4.5pt}]  (531.5,65.36) -- (543.15,145.96) ;
\draw  [dash pattern={on 4.5pt off 4.5pt}]  (510.07,124.25) -- (543.15,145.96) ;
\draw  [dash pattern={on 4.5pt off 4.5pt}]  (510.07,124.25) -- (591.15,105.68) ;
\draw    (591.15,105.68) -- (543.15,145.96) ;
\draw  [color={rgb, 255:red, 255; green, 255; blue, 255 }  ,draw opacity=1 ][fill={rgb, 255:red, 255; green, 255; blue, 255 }  ,fill opacity=1 ] (545.07,148.25) .. controls (543.8,149.32) and (541.92,149.15) .. (540.85,147.88) .. controls (539.79,146.62) and (539.96,144.73) .. (541.22,143.67) .. controls (542.49,142.61) and (544.38,142.77) .. (545.44,144.04) .. controls (546.5,145.3) and (546.34,147.19) .. (545.07,148.25) -- cycle ;
\draw  [color={rgb, 255:red, 255; green, 255; blue, 255 }  ,draw opacity=1 ][fill={rgb, 255:red, 255; green, 255; blue, 255 }  ,fill opacity=1 ] (569.07,128.11) .. controls (567.81,129.18) and (565.92,129.01) .. (564.86,127.74) .. controls (563.79,126.48) and (563.96,124.59) .. (565.22,123.53) .. controls (566.49,122.47) and (568.38,122.63) .. (569.44,123.9) .. controls (570.5,125.16) and (570.34,127.05) .. (569.07,128.11) -- cycle ;
\draw  [color={rgb, 255:red, 255; green, 255; blue, 255 }  ,draw opacity=1 ][fill={rgb, 255:red, 255; green, 255; blue, 255 }  ,fill opacity=1 ] (593.07,107.97) .. controls (591.81,109.03) and (589.92,108.87) .. (588.86,107.6) .. controls (587.8,106.34) and (587.96,104.45) .. (589.23,103.39) .. controls (590.49,102.32) and (592.38,102.49) .. (593.44,103.76) .. controls (594.51,105.02) and (594.34,106.91) .. (593.07,107.97) -- cycle ;

\draw  [fill={rgb, 255:red, 0; green, 0; blue, 0 }  ,fill opacity=1 ] (512.88,125.27) .. controls (512.32,126.83) and (510.6,127.63) .. (509.05,127.06) .. controls (507.49,126.5) and (506.69,124.78) .. (507.26,123.23) .. controls (507.82,121.67) and (509.54,120.87) .. (511.1,121.44) .. controls (512.65,122) and (513.45,123.72) .. (512.88,125.27) -- cycle ;
\draw  [fill={rgb, 255:red, 0; green, 0; blue, 0 }  ,fill opacity=1 ] (523.6,95.83) .. controls (523.04,97.38) and (521.32,98.18) .. (519.76,97.62) .. controls (518.21,97.05) and (517.41,95.34) .. (517.98,93.78) .. controls (518.54,92.23) and (520.26,91.43) .. (521.81,91.99) .. controls (523.36,92.56) and (524.17,94.28) .. (523.6,95.83) -- cycle ;
\draw  [fill={rgb, 255:red, 0; green, 0; blue, 0 }  ,fill opacity=1 ] (534.32,66.39) .. controls (533.75,67.94) and (532.03,68.74) .. (530.48,68.18) .. controls (528.93,67.61) and (528.13,65.89) .. (528.69,64.34) .. controls (529.26,62.79) and (530.97,61.99) .. (532.53,62.55) .. controls (534.08,63.12) and (534.88,64.83) .. (534.32,66.39) -- cycle ;

\draw    (510.07,124.25) -- (531.5,65.36) ;

\draw  [color={rgb, 255:red, 255; green, 255; blue, 255 }  ,draw opacity=1 ][fill={rgb, 255:red, 155; green, 155; blue, 155 }  ,fill opacity=0.5 ] (267.87,105.66) .. controls (267.87,68.47) and (298.02,38.32) .. (335.21,38.32) .. controls (372.4,38.32) and (402.55,68.47) .. (402.55,105.66) .. controls (402.55,142.85) and (372.4,173) .. (335.21,173) .. controls (298.02,173) and (267.87,142.85) .. (267.87,105.66) -- cycle ;
\draw    (338.57,78.53) -- (369.9,132.8) ;
\draw    (331.85,78.53) -- (300.52,132.8) ;

\draw  [fill={rgb, 255:red, 0; green, 0; blue, 0 }  ,fill opacity=1 ] (329.26,77.35) .. controls (330.09,75.92) and (331.92,75.43) .. (333.35,76.25) .. controls (334.78,77.08) and (335.27,78.91) .. (334.44,80.34) .. controls (333.62,81.77) and (331.79,82.26) .. (330.35,81.44) .. controls (328.92,80.61) and (328.43,78.78) .. (329.26,77.35) -- cycle ;
\draw  [fill={rgb, 255:red, 0; green, 0; blue, 0 }  ,fill opacity=1 ] (313.59,104.49) .. controls (314.42,103.05) and (316.25,102.56) .. (317.68,103.39) .. controls (319.11,104.22) and (319.6,106.05) .. (318.78,107.48) .. controls (317.95,108.91) and (316.12,109.4) .. (314.69,108.57) .. controls (313.26,107.75) and (312.77,105.92) .. (313.59,104.49) -- cycle ;
\draw  [fill={rgb, 255:red, 0; green, 0; blue, 0 }  ,fill opacity=1 ] (297.93,131.62) .. controls (298.75,130.19) and (300.58,129.7) .. (302.01,130.52) .. controls (303.45,131.35) and (303.94,133.18) .. (303.11,134.61) .. controls (302.28,136.04) and (300.45,136.53) .. (299.02,135.71) .. controls (297.59,134.88) and (297.1,133.05) .. (297.93,131.62) -- cycle ;

\draw  [color={rgb, 255:red, 255; green, 255; blue, 255 }  ,draw opacity=1 ][fill={rgb, 255:red, 255; green, 255; blue, 255 }  ,fill opacity=1 ] (372.49,130.98) .. controls (373.32,132.41) and (372.83,134.24) .. (371.39,135.07) .. controls (369.96,135.9) and (368.13,135.4) .. (367.31,133.97) .. controls (366.48,132.54) and (366.97,130.71) .. (368.4,129.89) .. controls (369.83,129.06) and (371.66,129.55) .. (372.49,130.98) -- cycle ;
\draw  [color={rgb, 255:red, 255; green, 255; blue, 255 }  ,draw opacity=1 ][fill={rgb, 255:red, 255; green, 255; blue, 255 }  ,fill opacity=1 ] (356.82,103.85) .. controls (357.65,105.28) and (357.16,107.11) .. (355.73,107.93) .. controls (354.3,108.76) and (352.47,108.27) .. (351.64,106.84) .. controls (350.81,105.41) and (351.3,103.58) .. (352.74,102.75) .. controls (354.17,101.92) and (356,102.41) .. (356.82,103.85) -- cycle ;
\draw  [color={rgb, 255:red, 255; green, 255; blue, 255 }  ,draw opacity=1 ][fill={rgb, 255:red, 255; green, 255; blue, 255 }  ,fill opacity=1 ] (341.16,76.71) .. controls (341.98,78.14) and (341.49,79.97) .. (340.06,80.8) .. controls (338.63,81.63) and (336.8,81.13) .. (335.97,79.7) .. controls (335.15,78.27) and (335.64,76.44) .. (337.07,75.62) .. controls (338.5,74.79) and (340.33,75.28) .. (341.16,76.71) -- cycle ;

\draw  [color={rgb, 255:red, 255; green, 255; blue, 255 }  ,draw opacity=1 ][fill={rgb, 255:red, 155; green, 155; blue, 155 }  ,fill opacity=0.5 ] (52.47,105.66) .. controls (52.47,68.47) and (82.62,38.32) .. (119.81,38.32) .. controls (157,38.32) and (187.15,68.47) .. (187.15,105.66) .. controls (187.15,142.85) and (157,173) .. (119.81,173) .. controls (82.62,173) and (52.47,142.85) .. (52.47,105.66) -- cycle ;
\draw  [fill={rgb, 255:red, 0; green, 0; blue, 0 }  ,fill opacity=1 ] (113.46,74.33) .. controls (113.46,72.68) and (114.8,71.34) .. (116.45,71.34) .. controls (118.11,71.34) and (119.45,72.68) .. (119.45,74.33) .. controls (119.45,75.98) and (118.11,77.32) .. (116.45,77.32) .. controls (114.8,77.32) and (113.46,75.98) .. (113.46,74.33) -- cycle ;
\draw  [fill={rgb, 255:red, 0; green, 0; blue, 0 }  ,fill opacity=1 ] (113.46,105.66) .. controls (113.46,104.01) and (114.8,102.67) .. (116.45,102.67) .. controls (118.11,102.67) and (119.45,104.01) .. (119.45,105.66) .. controls (119.45,107.31) and (118.11,108.65) .. (116.45,108.65) .. controls (114.8,108.65) and (113.46,107.31) .. (113.46,105.66) -- cycle ;
\draw  [fill={rgb, 255:red, 0; green, 0; blue, 0 }  ,fill opacity=1 ] (113.46,136.99) .. controls (113.46,135.34) and (114.8,134) .. (116.45,134) .. controls (118.11,134) and (119.45,135.34) .. (119.45,136.99) .. controls (119.45,138.65) and (118.11,139.99) .. (116.45,139.99) .. controls (114.8,139.99) and (113.46,138.65) .. (113.46,136.99) -- cycle ;

\draw    (119.45,74.33) -- (119.45,136.99) ;
\draw  [color={rgb, 255:red, 255; green, 255; blue, 255 }  ,draw opacity=1 ][fill={rgb, 255:red, 255; green, 255; blue, 255 }  ,fill opacity=1 ] (120.18,74.33) .. controls (120.18,72.68) and (121.52,71.34) .. (123.17,71.34) .. controls (124.82,71.34) and (126.16,72.68) .. (126.16,74.33) .. controls (126.16,75.98) and (124.82,77.32) .. (123.17,77.32) .. controls (121.52,77.32) and (120.18,75.98) .. (120.18,74.33) -- cycle ;
\draw  [color={rgb, 255:red, 255; green, 255; blue, 255 }  ,draw opacity=1 ][fill={rgb, 255:red, 255; green, 255; blue, 255 }  ,fill opacity=1 ] (120.18,105.66) .. controls (120.18,104.01) and (121.52,102.67) .. (123.17,102.67) .. controls (124.82,102.67) and (126.16,104.01) .. (126.16,105.66) .. controls (126.16,107.31) and (124.82,108.65) .. (123.17,108.65) .. controls (121.52,108.65) and (120.18,107.31) .. (120.18,105.66) -- cycle ;
\draw  [color={rgb, 255:red, 255; green, 255; blue, 255 }  ,draw opacity=1 ][fill={rgb, 255:red, 255; green, 255; blue, 255 }  ,fill opacity=1 ] (120.18,136.99) .. controls (120.18,135.34) and (121.52,134) .. (123.17,134) .. controls (124.82,134) and (126.16,135.34) .. (126.16,136.99) .. controls (126.16,138.65) and (124.82,139.99) .. (123.17,139.99) .. controls (121.52,139.99) and (120.18,138.65) .. (120.18,136.99) -- cycle ;

\end{tikzpicture}

%% file: figures/fig_bundle.tex
\tikzset{every picture/.style={line width=0.75pt}} 

\begin{tikzpicture}[x=0.75pt,y=0.75pt,yscale=-1,xscale=1]

\draw    (477.44,104.51) -- (451.81,130.1) ;
\draw  [color={rgb, 255:red, 255; green, 255; blue, 255 }  ,draw opacity=1 ][fill={rgb, 255:red, 155; green, 155; blue, 155 }  ,fill opacity=0.5 ] (407.94,104.49) .. controls (407.94,80.53) and (427.36,61.11) .. (451.32,61.11) .. controls (475.28,61.11) and (494.71,80.53) .. (494.71,104.49) .. controls (494.71,128.46) and (475.28,147.88) .. (451.32,147.88) .. controls (427.36,147.88) and (407.94,128.46) .. (407.94,104.49) -- cycle ;
\draw  [dash pattern={on 4.5pt off 4.5pt}]  (439.01,78.53) -- (477.44,104.51) ;
\draw  [dash pattern={on 4.5pt off 4.5pt}]  (425.2,116.47) -- (446.51,130.46) ;
\draw  [dash pattern={on 4.5pt off 4.5pt}]  (425.2,116.47) -- (477.44,104.51) ;
\draw  [dash pattern={on 4.5pt off 4.5pt}]  (439.01,78.53) -- (451.81,130.1) ;
\draw  [fill={rgb, 255:red, 0; green, 0; blue, 0 }  ,fill opacity=1 ] (427.02,117.13) .. controls (426.65,118.13) and (425.54,118.65) .. (424.54,118.28) .. controls (423.54,117.92) and (423.03,116.81) .. (423.39,115.81) .. controls (423.76,114.81) and (424.86,114.29) .. (425.86,114.66) .. controls (426.86,115.02) and (427.38,116.13) .. (427.02,117.13) -- cycle ;
\draw  [fill={rgb, 255:red, 0; green, 0; blue, 0 }  ,fill opacity=1 ] (433.92,98.16) .. controls (433.56,99.16) and (432.45,99.68) .. (431.45,99.31) .. controls (430.45,98.95) and (429.93,97.84) .. (430.3,96.84) .. controls (430.66,95.84) and (431.77,95.32) .. (432.77,95.69) .. controls (433.77,96.05) and (434.28,97.16) .. (433.92,98.16) -- cycle ;
\draw  [fill={rgb, 255:red, 0; green, 0; blue, 0 }  ,fill opacity=1 ] (440.82,79.19) .. controls (440.46,80.19) and (439.35,80.71) .. (438.35,80.34) .. controls (437.35,79.98) and (436.84,78.87) .. (437.2,77.87) .. controls (437.57,76.87) and (438.67,76.35) .. (439.67,76.72) .. controls (440.67,77.08) and (441.19,78.19) .. (440.82,79.19) -- cycle ;

\draw    (425.2,116.47) -- (439.01,78.53) ;

\draw  [color={rgb, 255:red, 0; green, 0; blue, 0 }  ,draw opacity=1 ][fill={rgb, 255:red, 255; green, 255; blue, 255 }  ,fill opacity=1 ] (453.48,131.07) .. controls (454.01,130.15) and (453.69,128.97) .. (452.77,128.43) .. controls (451.85,127.9) and (450.67,128.22) .. (450.14,129.14) .. controls (449.6,130.06) and (449.92,131.24) .. (450.84,131.77) .. controls (451.77,132.31) and (452.94,131.99) .. (453.48,131.07) -- cycle ;
\draw  [color={rgb, 255:red, 0; green, 0; blue, 0 }  ,draw opacity=1 ][fill={rgb, 255:red, 255; green, 255; blue, 255 }  ,fill opacity=1 ] (479.11,105.47) .. controls (479.64,104.55) and (479.33,103.37) .. (478.41,102.84) .. controls (477.48,102.3) and (476.31,102.62) .. (475.77,103.54) .. controls (475.24,104.46) and (475.56,105.64) .. (476.48,106.18) .. controls (477.4,106.71) and (478.58,106.39) .. (479.11,105.47) -- cycle ;
\draw  [color={rgb, 255:red, 0; green, 0; blue, 0 }  ,draw opacity=1 ][fill={rgb, 255:red, 255; green, 255; blue, 255 }  ,fill opacity=1 ] (466.29,118.27) .. controls (466.83,117.35) and (466.51,116.17) .. (465.59,115.63) .. controls (464.67,115.1) and (463.49,115.42) .. (462.95,116.34) .. controls (462.42,117.26) and (462.74,118.44) .. (463.66,118.97) .. controls (464.58,119.51) and (465.76,119.19) .. (466.29,118.27) -- cycle ;

\draw  [color={rgb, 255:red, 255; green, 255; blue, 255 }  ,draw opacity=1 ][fill={rgb, 255:red, 155; green, 155; blue, 155 }  ,fill opacity=0.5 ] (129.65,104.49) .. controls (129.65,80.53) and (149.07,61.11) .. (173.03,61.11) .. controls (197,61.11) and (216.42,80.53) .. (216.42,104.49) .. controls (216.42,128.46) and (197,147.88) .. (173.03,147.88) .. controls (149.07,147.88) and (129.65,128.46) .. (129.65,104.49) -- cycle ;
\draw    (175.2,87.01) -- (195.38,121.98) ;
\draw    (170.87,87.01) -- (150.68,121.98) ;

\draw  [fill={rgb, 255:red, 0; green, 0; blue, 0 }  ,fill opacity=1 ] (169.2,86.25) .. controls (169.73,85.33) and (170.91,85.01) .. (171.84,85.55) .. controls (172.76,86.08) and (173.07,87.26) .. (172.54,88.18) .. controls (172.01,89.1) and (170.83,89.42) .. (169.91,88.89) .. controls (168.98,88.35) and (168.67,87.18) .. (169.2,86.25) -- cycle ;
\draw  [fill={rgb, 255:red, 0; green, 0; blue, 0 }  ,fill opacity=1 ] (159.11,103.74) .. controls (159.64,102.81) and (160.82,102.5) .. (161.74,103.03) .. controls (162.66,103.56) and (162.98,104.74) .. (162.45,105.66) .. controls (161.91,106.59) and (160.74,106.9) .. (159.81,106.37) .. controls (158.89,105.84) and (158.57,104.66) .. (159.11,103.74) -- cycle ;
\draw  [fill={rgb, 255:red, 0; green, 0; blue, 0 }  ,fill opacity=1 ] (149.01,121.22) .. controls (149.55,120.3) and (150.73,119.98) .. (151.65,120.51) .. controls (152.57,121.05) and (152.89,122.22) .. (152.35,123.15) .. controls (151.82,124.07) and (150.64,124.39) .. (149.72,123.85) .. controls (148.8,123.32) and (148.48,122.14) .. (149.01,121.22) -- cycle ;

\draw  [color={rgb, 255:red, 255; green, 255; blue, 255 }  ,draw opacity=1 ][fill={rgb, 255:red, 255; green, 255; blue, 255 }  ,fill opacity=1 ] (197.05,120.81) .. controls (197.59,121.73) and (197.27,122.91) .. (196.35,123.44) .. controls (195.43,123.97) and (194.25,123.66) .. (193.71,122.74) .. controls (193.18,121.81) and (193.5,120.63) .. (194.42,120.1) .. controls (195.34,119.57) and (196.52,119.88) .. (197.05,120.81) -- cycle ;
\draw  [color={rgb, 255:red, 255; green, 255; blue, 255 }  ,draw opacity=1 ][fill={rgb, 255:red, 255; green, 255; blue, 255 }  ,fill opacity=1 ] (186.96,103.32) .. controls (187.49,104.25) and (187.18,105.43) .. (186.25,105.96) .. controls (185.33,106.49) and (184.15,106.17) .. (183.62,105.25) .. controls (183.09,104.33) and (183.4,103.15) .. (184.33,102.62) .. controls (185.25,102.09) and (186.43,102.4) .. (186.96,103.32) -- cycle ;
\draw  [color={rgb, 255:red, 255; green, 255; blue, 255 }  ,draw opacity=1 ][fill={rgb, 255:red, 255; green, 255; blue, 255 }  ,fill opacity=1 ] (176.87,85.84) .. controls (177.4,86.76) and (177.08,87.94) .. (176.16,88.47) .. controls (175.24,89.01) and (174.06,88.69) .. (173.53,87.77) .. controls (172.99,86.85) and (173.31,85.67) .. (174.23,85.14) .. controls (175.16,84.6) and (176.33,84.92) .. (176.87,85.84) -- cycle ;

\draw  [color={rgb, 255:red, 255; green, 255; blue, 255 }  ,draw opacity=1 ][fill={rgb, 255:red, 155; green, 155; blue, 155 }  ,fill opacity=0.5 ] (236.54,104.49) .. controls (236.54,128.46) and (255.96,147.88) .. (279.92,147.88) .. controls (303.88,147.88) and (323.31,128.46) .. (323.31,104.49) .. controls (323.31,80.53) and (303.88,61.11) .. (279.92,61.11) .. controls (255.96,61.11) and (236.54,80.53) .. (236.54,104.49) -- cycle ;
\draw    (282.09,121.98) -- (302.27,87.01) ;
\draw    (277.76,121.98) -- (257.57,87.01) ;

\draw  [color={rgb, 255:red, 255; green, 255; blue, 255 }  ,draw opacity=1 ][fill={rgb, 255:red, 255; green, 255; blue, 255 }  ,fill opacity=1 ] (276.09,122.74) .. controls (276.62,123.66) and (277.8,123.97) .. (278.72,123.44) .. controls (279.65,122.91) and (279.96,121.73) .. (279.43,120.81) .. controls (278.9,119.88) and (277.72,119.57) .. (276.8,120.1) .. controls (275.87,120.63) and (275.56,121.81) .. (276.09,122.74) -- cycle ;
\draw  [color={rgb, 255:red, 255; green, 255; blue, 255 }  ,draw opacity=1 ][fill={rgb, 255:red, 255; green, 255; blue, 255 }  ,fill opacity=1 ] (266,105.25) .. controls (266.53,106.17) and (267.71,106.49) .. (268.63,105.96) .. controls (269.55,105.43) and (269.87,104.25) .. (269.34,103.32) .. controls (268.8,102.4) and (267.62,102.09) .. (266.7,102.62) .. controls (265.78,103.15) and (265.46,104.33) .. (266,105.25) -- cycle ;
\draw  [color={rgb, 255:red, 255; green, 255; blue, 255 }  ,draw opacity=1 ][fill={rgb, 255:red, 255; green, 255; blue, 255 }  ,fill opacity=1 ] (255.9,87.77) .. controls (256.43,88.69) and (257.61,89.01) .. (258.54,88.47) .. controls (259.46,87.94) and (259.77,86.76) .. (259.24,85.84) .. controls (258.71,84.92) and (257.53,84.6) .. (256.61,85.14) .. controls (255.69,85.67) and (255.37,86.85) .. (255.9,87.77) -- cycle ;
\draw  [color={rgb, 255:red, 0; green, 0; blue, 0 }  ,draw opacity=1 ][fill={rgb, 255:red, 255; green, 255; blue, 255 }  ,fill opacity=1 ] (303.94,87.98) .. controls (304.48,87.05) and (304.16,85.87) .. (303.24,85.34) .. controls (302.32,84.81) and (301.14,85.12) .. (300.6,86.05) .. controls (300.07,86.97) and (300.39,88.15) .. (301.31,88.68) .. controls (302.23,89.21) and (303.41,88.9) .. (303.94,87.98) -- cycle ;
\draw  [color={rgb, 255:red, 0; green, 0; blue, 0 }  ,draw opacity=1 ][fill={rgb, 255:red, 255; green, 255; blue, 255 }  ,fill opacity=1 ] (293.85,105.66) .. controls (294.38,104.74) and (294.07,103.56) .. (293.14,103.03) .. controls (292.22,102.5) and (291.04,102.81) .. (290.51,103.74) .. controls (289.98,104.66) and (290.29,105.84) .. (291.22,106.37) .. controls (292.14,106.9) and (293.32,106.59) .. (293.85,105.66) -- cycle ;
\draw  [color={rgb, 255:red, 0; green, 0; blue, 0 }  ,draw opacity=1 ][fill={rgb, 255:red, 255; green, 255; blue, 255 }  ,fill opacity=1 ] (283.76,123.15) .. controls (284.29,122.22) and (283.97,121.05) .. (283.05,120.51) .. controls (282.13,119.98) and (280.95,120.3) .. (280.42,121.22) .. controls (279.88,122.14) and (280.2,123.32) .. (281.12,123.85) .. controls (282.04,124.39) and (283.22,124.07) .. (283.76,123.15) -- cycle ;

\draw    (216.42,104.49) -- (236.54,104.49) ;

\draw  [dash pattern={on 4.5pt off 4.5pt}] (114.22,74.55) .. controls (114.22,63.53) and (123.16,54.59) .. (134.19,54.59) -- (317.26,54.59) .. controls (328.29,54.59) and (337.22,63.53) .. (337.22,74.55) -- (337.22,134.44) .. controls (337.22,145.46) and (328.29,154.4) .. (317.26,154.4) -- (134.19,154.4) .. controls (123.16,154.4) and (114.22,145.46) .. (114.22,134.44) -- cycle ;

\draw  [color={rgb, 255:red, 255; green, 255; blue, 255 }  ,draw opacity=1 ][fill={rgb, 255:red, 155; green, 155; blue, 155 }  ,fill opacity=0.5 ] (129.65,214.49) .. controls (129.65,190.53) and (149.07,171.11) .. (173.03,171.11) .. controls (197,171.11) and (216.42,190.53) .. (216.42,214.49) .. controls (216.42,238.46) and (197,257.88) .. (173.03,257.88) .. controls (149.07,257.88) and (129.65,238.46) .. (129.65,214.49) -- cycle ;
\draw    (175.2,197.01) -- (195.38,231.98) ;
\draw    (170.87,197.01) -- (150.68,231.98) ;

\draw  [fill={rgb, 255:red, 0; green, 0; blue, 0 }  ,fill opacity=1 ] (169.2,196.25) .. controls (169.73,195.33) and (170.91,195.01) .. (171.84,195.55) .. controls (172.76,196.08) and (173.07,197.26) .. (172.54,198.18) .. controls (172.01,199.1) and (170.83,199.42) .. (169.91,198.89) .. controls (168.98,198.35) and (168.67,197.18) .. (169.2,196.25) -- cycle ;
\draw  [fill={rgb, 255:red, 0; green, 0; blue, 0 }  ,fill opacity=1 ] (159.11,213.74) .. controls (159.64,212.81) and (160.82,212.5) .. (161.74,213.03) .. controls (162.66,213.56) and (162.98,214.74) .. (162.45,215.66) .. controls (161.91,216.59) and (160.74,216.9) .. (159.81,216.37) .. controls (158.89,215.84) and (158.57,214.66) .. (159.11,213.74) -- cycle ;
\draw  [fill={rgb, 255:red, 0; green, 0; blue, 0 }  ,fill opacity=1 ] (149.01,231.22) .. controls (149.55,230.3) and (150.73,229.98) .. (151.65,230.51) .. controls (152.57,231.05) and (152.89,232.22) .. (152.35,233.15) .. controls (151.82,234.07) and (150.64,234.39) .. (149.72,233.85) .. controls (148.8,233.32) and (148.48,232.14) .. (149.01,231.22) -- cycle ;

\draw  [color={rgb, 255:red, 255; green, 255; blue, 255 }  ,draw opacity=1 ][fill={rgb, 255:red, 255; green, 255; blue, 255 }  ,fill opacity=1 ] (197.05,230.81) .. controls (197.59,231.73) and (197.27,232.91) .. (196.35,233.44) .. controls (195.43,233.97) and (194.25,233.66) .. (193.71,232.74) .. controls (193.18,231.81) and (193.5,230.63) .. (194.42,230.1) .. controls (195.34,229.57) and (196.52,229.88) .. (197.05,230.81) -- cycle ;
\draw  [color={rgb, 255:red, 255; green, 255; blue, 255 }  ,draw opacity=1 ][fill={rgb, 255:red, 255; green, 255; blue, 255 }  ,fill opacity=1 ] (186.96,213.32) .. controls (187.49,214.25) and (187.18,215.43) .. (186.25,215.96) .. controls (185.33,216.49) and (184.15,216.17) .. (183.62,215.25) .. controls (183.09,214.33) and (183.4,213.15) .. (184.33,212.62) .. controls (185.25,212.09) and (186.43,212.4) .. (186.96,213.32) -- cycle ;
\draw  [color={rgb, 255:red, 255; green, 255; blue, 255 }  ,draw opacity=1 ][fill={rgb, 255:red, 255; green, 255; blue, 255 }  ,fill opacity=1 ] (176.87,195.84) .. controls (177.4,196.76) and (177.08,197.94) .. (176.16,198.47) .. controls (175.24,199.01) and (174.06,198.69) .. (173.53,197.77) .. controls (172.99,196.85) and (173.31,195.67) .. (174.23,195.14) .. controls (175.16,194.6) and (176.33,194.92) .. (176.87,195.84) -- cycle ;

\draw  [color={rgb, 255:red, 255; green, 255; blue, 255 }  ,draw opacity=1 ][fill={rgb, 255:red, 155; green, 155; blue, 155 }  ,fill opacity=0.5 ] (236.54,214.49) .. controls (236.54,190.53) and (255.96,171.11) .. (279.92,171.11) .. controls (303.88,171.11) and (323.31,190.53) .. (323.31,214.49) .. controls (323.31,238.46) and (303.88,257.88) .. (279.92,257.88) .. controls (255.96,257.88) and (236.54,238.46) .. (236.54,214.49) -- cycle ;
\draw    (282.09,197.01) -- (302.27,231.98) ;
\draw    (277.76,197.01) -- (257.57,231.98) ;

\draw  [color={rgb, 255:red, 255; green, 255; blue, 255 }  ,draw opacity=1 ][fill={rgb, 255:red, 255; green, 255; blue, 255 }  ,fill opacity=1 ] (276.09,196.25) .. controls (276.62,195.33) and (277.8,195.01) .. (278.72,195.55) .. controls (279.65,196.08) and (279.96,197.26) .. (279.43,198.18) .. controls (278.9,199.1) and (277.72,199.42) .. (276.8,198.89) .. controls (275.87,198.35) and (275.56,197.18) .. (276.09,196.25) -- cycle ;
\draw  [color={rgb, 255:red, 255; green, 255; blue, 255 }  ,draw opacity=1 ][fill={rgb, 255:red, 255; green, 255; blue, 255 }  ,fill opacity=1 ] (266,213.74) .. controls (266.53,212.81) and (267.71,212.5) .. (268.63,213.03) .. controls (269.55,213.56) and (269.87,214.74) .. (269.34,215.66) .. controls (268.8,216.59) and (267.62,216.9) .. (266.7,216.37) .. controls (265.78,215.84) and (265.46,214.66) .. (266,213.74) -- cycle ;
\draw  [color={rgb, 255:red, 255; green, 255; blue, 255 }  ,draw opacity=1 ][fill={rgb, 255:red, 255; green, 255; blue, 255 }  ,fill opacity=1 ] (255.9,231.22) .. controls (256.43,230.3) and (257.61,229.98) .. (258.54,230.51) .. controls (259.46,231.05) and (259.77,232.22) .. (259.24,233.15) .. controls (258.71,234.07) and (257.53,234.39) .. (256.61,233.85) .. controls (255.69,233.32) and (255.37,232.14) .. (255.9,231.22) -- cycle ;
\draw  [color={rgb, 255:red, 0; green, 0; blue, 0 }  ,draw opacity=1 ][fill={rgb, 255:red, 255; green, 255; blue, 255 }  ,fill opacity=1 ] (303.94,231.01) .. controls (304.48,231.94) and (304.16,233.11) .. (303.24,233.65) .. controls (302.32,234.18) and (301.14,233.86) .. (300.6,232.94) .. controls (300.07,232.02) and (300.39,230.84) .. (301.31,230.31) .. controls (302.23,229.77) and (303.41,230.09) .. (303.94,231.01) -- cycle ;
\draw  [color={rgb, 255:red, 0; green, 0; blue, 0 }  ,draw opacity=1 ][fill={rgb, 255:red, 255; green, 255; blue, 255 }  ,fill opacity=1 ] (293.85,213.32) .. controls (294.38,214.25) and (294.07,215.43) .. (293.14,215.96) .. controls (292.22,216.49) and (291.04,216.17) .. (290.51,215.25) .. controls (289.98,214.33) and (290.29,213.15) .. (291.22,212.62) .. controls (292.14,212.09) and (293.32,212.4) .. (293.85,213.32) -- cycle ;
\draw  [color={rgb, 255:red, 0; green, 0; blue, 0 }  ,draw opacity=1 ][fill={rgb, 255:red, 255; green, 255; blue, 255 }  ,fill opacity=1 ] (283.76,195.84) .. controls (284.29,196.76) and (283.97,197.94) .. (283.05,198.47) .. controls (282.13,199.01) and (280.95,198.69) .. (280.42,197.77) .. controls (279.88,196.85) and (280.2,195.67) .. (281.12,195.14) .. controls (282.04,194.6) and (283.22,194.92) .. (283.76,195.84) -- cycle ;

\draw    (216.42,214.49) -- (236.54,214.49) ;
\draw  [dash pattern={on 4.5pt off 4.5pt}] (114.22,184.55) .. controls (114.22,173.53) and (123.16,164.59) .. (134.19,164.59) -- (317.26,164.59) .. controls (328.29,164.59) and (337.22,173.53) .. (337.22,184.55) -- (337.22,244.44) .. controls (337.22,255.46) and (328.29,264.4) .. (317.26,264.4) -- (134.19,264.4) .. controls (123.16,264.4) and (114.22,255.46) .. (114.22,244.44) -- cycle ;
\draw  [color={rgb, 255:red, 255; green, 255; blue, 255 }  ,draw opacity=1 ][fill={rgb, 255:red, 155; green, 155; blue, 155 }  ,fill opacity=0.5 ] (407.94,214.49) .. controls (407.94,190.53) and (427.36,171.11) .. (451.32,171.11) .. controls (475.28,171.11) and (494.71,190.53) .. (494.71,214.49) .. controls (494.71,238.46) and (475.28,257.88) .. (451.32,257.88) .. controls (427.36,257.88) and (407.94,238.46) .. (407.94,214.49) -- cycle ;
\draw    (453.49,197.01) -- (473.67,231.98) ;
\draw    (449.16,197.01) -- (428.97,231.98) ;

\draw  [color={rgb, 255:red, 0; green, 0; blue, 0 }  ,draw opacity=1 ][fill={rgb, 255:red, 0; green, 0; blue, 0 }  ,fill opacity=1 ] (447.49,196.25) .. controls (448.02,195.33) and (449.2,195.01) .. (450.12,195.55) .. controls (451.05,196.08) and (451.36,197.26) .. (450.83,198.18) .. controls (450.3,199.1) and (449.12,199.42) .. (448.2,198.89) .. controls (447.27,198.35) and (446.96,197.18) .. (447.49,196.25) -- cycle ;
\draw  [color={rgb, 255:red, 0; green, 0; blue, 0 }  ,draw opacity=1 ][fill={rgb, 255:red, 0; green, 0; blue, 0 }  ,fill opacity=1 ] (437.4,213.74) .. controls (437.93,212.81) and (439.11,212.5) .. (440.03,213.03) .. controls (440.95,213.56) and (441.27,214.74) .. (440.74,215.66) .. controls (440.2,216.59) and (439.02,216.9) .. (438.1,216.37) .. controls (437.18,215.84) and (436.86,214.66) .. (437.4,213.74) -- cycle ;
\draw  [color={rgb, 255:red, 0; green, 0; blue, 0 }  ,draw opacity=1 ][fill={rgb, 255:red, 0; green, 0; blue, 0 }  ,fill opacity=1 ] (427.3,231.22) .. controls (427.83,230.3) and (429.01,229.98) .. (429.94,230.51) .. controls (430.86,231.05) and (431.17,232.22) .. (430.64,233.15) .. controls (430.11,234.07) and (428.93,234.39) .. (428.01,233.85) .. controls (427.09,233.32) and (426.77,232.14) .. (427.3,231.22) -- cycle ;
\draw  [color={rgb, 255:red, 0; green, 0; blue, 0 }  ,draw opacity=1 ][fill={rgb, 255:red, 255; green, 255; blue, 255 }  ,fill opacity=1 ] (475.34,231.01) .. controls (475.88,231.94) and (475.56,233.11) .. (474.64,233.65) .. controls (473.72,234.18) and (472.54,233.86) .. (472,232.94) .. controls (471.47,232.02) and (471.79,230.84) .. (472.71,230.31) .. controls (473.63,229.77) and (474.81,230.09) .. (475.34,231.01) -- cycle ;
\draw  [color={rgb, 255:red, 0; green, 0; blue, 0 }  ,draw opacity=1 ][fill={rgb, 255:red, 255; green, 255; blue, 255 }  ,fill opacity=1 ] (465.25,213.32) .. controls (465.78,214.25) and (465.47,215.43) .. (464.54,215.96) .. controls (463.62,216.49) and (462.44,216.17) .. (461.91,215.25) .. controls (461.38,214.33) and (461.69,213.15) .. (462.62,212.62) .. controls (463.54,212.09) and (464.72,212.4) .. (465.25,213.32) -- cycle ;
\draw  [color={rgb, 255:red, 0; green, 0; blue, 0 }  ,draw opacity=1 ][fill={rgb, 255:red, 255; green, 255; blue, 255 }  ,fill opacity=1 ] (455.16,195.84) .. controls (455.69,196.76) and (455.37,197.94) .. (454.45,198.47) .. controls (453.53,199.01) and (452.35,198.69) .. (451.82,197.77) .. controls (451.28,196.85) and (451.6,195.67) .. (452.52,195.14) .. controls (453.44,194.6) and (454.62,194.92) .. (455.16,195.84) -- cycle ;

\draw (71.33,97.39) node [anchor=north west][inner sep=0.75pt]    {$( a)$};
\draw (71.33,207.39) node [anchor=north west][inner sep=0.75pt]    {$( b)$};

\end{tikzpicture}

%% file: figures/fig_many_parallel_traingles.tex
\tikzset{every picture/.style={line width=0.75pt}} 

\begin{tikzpicture}[x=0.75pt,y=0.75pt,yscale=-1,xscale=1]

\draw  [color={rgb, 255:red, 255; green, 255; blue, 255 }  ,draw opacity=1 ][fill={rgb, 255:red, 155; green, 155; blue, 155 }  ,fill opacity=0.5 ] (403.55,135.34) .. controls (403.55,100.16) and (432.08,71.63) .. (467.27,71.63) .. controls (502.46,71.63) and (530.98,100.16) .. (530.98,135.34) .. controls (530.98,170.53) and (502.46,199.06) .. (467.27,199.06) .. controls (432.08,199.06) and (403.55,170.53) .. (403.55,135.34) -- cycle ;
\draw  [color={rgb, 255:red, 0; green, 0; blue, 0 }  ,draw opacity=1 ][fill={rgb, 255:red, 0; green, 0; blue, 0 }  ,fill opacity=1 ] (461.26,102.87) .. controls (461.26,101.3) and (462.53,100.03) .. (464.09,100.03) .. controls (465.66,100.03) and (466.92,101.3) .. (466.92,102.87) .. controls (466.92,104.43) and (465.66,105.7) .. (464.09,105.7) .. controls (462.53,105.7) and (461.26,104.43) .. (461.26,102.87) -- cycle ;
\draw  [color={rgb, 255:red, 0; green, 0; blue, 0 }  ,draw opacity=1 ][fill={rgb, 255:red, 0; green, 0; blue, 0 }  ,fill opacity=1 ] (461.26,132.51) .. controls (461.26,130.95) and (462.53,129.68) .. (464.09,129.68) .. controls (465.66,129.68) and (466.92,130.95) .. (466.92,132.51) .. controls (466.92,134.08) and (465.66,135.34) .. (464.09,135.34) .. controls (462.53,135.34) and (461.26,134.08) .. (461.26,132.51) -- cycle ;
\draw  [color={rgb, 255:red, 0; green, 0; blue, 0 }  ,draw opacity=1 ][fill={rgb, 255:red, 0; green, 0; blue, 0 }  ,fill opacity=1 ] (461.26,162.16) .. controls (461.26,160.6) and (462.53,159.33) .. (464.09,159.33) .. controls (465.66,159.33) and (466.92,160.6) .. (466.92,162.16) .. controls (466.92,163.72) and (465.66,164.99) .. (464.09,164.99) .. controls (462.53,164.99) and (461.26,163.72) .. (461.26,162.16) -- cycle ;

\draw [color={rgb, 255:red, 0; green, 0; blue, 0 }  ,draw opacity=1 ][fill={rgb, 255:red, 0; green, 0; blue, 0 }  ,fill opacity=1 ]   (466.92,102.87) -- (466.92,162.16) ;
\draw  [color={rgb, 255:red, 0; green, 0; blue, 0 }  ,draw opacity=1 ][fill={rgb, 255:red, 0; green, 0; blue, 0 }  ,fill opacity=1 ] (467.61,102.87) .. controls (467.61,101.3) and (468.88,100.03) .. (470.44,100.03) .. controls (472.01,100.03) and (473.28,101.3) .. (473.28,102.87) .. controls (473.28,104.43) and (472.01,105.7) .. (470.44,105.7) .. controls (468.88,105.7) and (467.61,104.43) .. (467.61,102.87) -- cycle ;
\draw  [color={rgb, 255:red, 0; green, 0; blue, 0 }  ,draw opacity=1 ][fill={rgb, 255:red, 0; green, 0; blue, 0 }  ,fill opacity=1 ] (467.61,132.51) .. controls (467.61,130.95) and (468.88,129.68) .. (470.44,129.68) .. controls (472.01,129.68) and (473.28,130.95) .. (473.28,132.51) .. controls (473.28,134.08) and (472.01,135.34) .. (470.44,135.34) .. controls (468.88,135.34) and (467.61,134.08) .. (467.61,132.51) -- cycle ;
\draw  [color={rgb, 255:red, 0; green, 0; blue, 0 }  ,draw opacity=1 ][fill={rgb, 255:red, 0; green, 0; blue, 0 }  ,fill opacity=1 ] (467.61,162.16) .. controls (467.61,160.6) and (468.88,159.33) .. (470.44,159.33) .. controls (472.01,159.33) and (473.28,160.6) .. (473.28,162.16) .. controls (473.28,163.72) and (472.01,164.99) .. (470.44,164.99) .. controls (468.88,164.99) and (467.61,163.72) .. (467.61,162.16) -- cycle ;

\draw  [color={rgb, 255:red, 255; green, 255; blue, 255 }  ,draw opacity=1 ][fill={rgb, 255:red, 155; green, 155; blue, 155 }  ,fill opacity=0.5 ] (12.16,135.35) .. controls (12.16,100.16) and (40.68,71.63) .. (75.87,71.63) .. controls (111.06,71.63) and (139.59,100.16) .. (139.59,135.35) .. controls (139.59,170.54) and (111.06,199.06) .. (75.87,199.06) .. controls (40.68,199.06) and (12.16,170.54) .. (12.16,135.35) -- cycle ;
\draw  [color={rgb, 255:red, 0; green, 0; blue, 0 }  ,draw opacity=1 ][fill={rgb, 255:red, 0; green, 0; blue, 0 }  ,fill opacity=1 ] (73.04,98.13) .. controls (73.04,96.57) and (74.31,95.3) .. (75.87,95.3) .. controls (77.44,95.3) and (78.71,96.57) .. (78.71,98.13) .. controls (78.71,99.69) and (77.44,100.96) .. (75.87,100.96) .. controls (74.31,100.96) and (73.04,99.69) .. (73.04,98.13) -- cycle ;
\draw  [color={rgb, 255:red, 0; green, 0; blue, 0 }  ,draw opacity=1 ][fill={rgb, 255:red, 0; green, 0; blue, 0 }  ,fill opacity=1 ] (73.04,127.78) .. controls (73.04,126.21) and (74.31,124.95) .. (75.87,124.95) .. controls (77.44,124.95) and (78.71,126.21) .. (78.71,127.78) .. controls (78.71,129.34) and (77.44,130.61) .. (75.87,130.61) .. controls (74.31,130.61) and (73.04,129.34) .. (73.04,127.78) -- cycle ;
\draw  [color={rgb, 255:red, 0; green, 0; blue, 0 }  ,draw opacity=1 ][fill={rgb, 255:red, 0; green, 0; blue, 0 }  ,fill opacity=1 ] (73.04,157.43) .. controls (73.04,155.86) and (74.31,154.59) .. (75.87,154.59) .. controls (77.44,154.59) and (78.71,155.86) .. (78.71,157.43) .. controls (78.71,158.99) and (77.44,160.26) .. (75.87,160.26) .. controls (74.31,160.26) and (73.04,158.99) .. (73.04,157.43) -- cycle ;

\draw [color={rgb, 255:red, 0; green, 0; blue, 0 }  ,draw opacity=1 ][fill={rgb, 255:red, 0; green, 0; blue, 0 }  ,fill opacity=1 ]   (75.87,98.13) -- (75.87,157.43) ;

\draw  [color={rgb, 255:red, 255; green, 255; blue, 255 }  ,draw opacity=1 ][fill={rgb, 255:red, 155; green, 155; blue, 155 }  ,fill opacity=0.5 ] (176.1,135.34) .. controls (176.1,100.16) and (204.63,71.63) .. (239.82,71.63) .. controls (275.01,71.63) and (303.53,100.16) .. (303.53,135.34) .. controls (303.53,170.53) and (275.01,199.06) .. (239.82,199.06) .. controls (204.63,199.06) and (176.1,170.53) .. (176.1,135.34) -- cycle ;
\draw  [color={rgb, 255:red, 0; green, 0; blue, 0 }  ,draw opacity=1 ][fill={rgb, 255:red, 0; green, 0; blue, 0 }  ,fill opacity=1 ] (233.81,102.87) .. controls (233.81,101.3) and (235.08,100.03) .. (236.64,100.03) .. controls (238.2,100.03) and (239.47,101.3) .. (239.47,102.87) .. controls (239.47,104.43) and (238.2,105.7) .. (236.64,105.7) .. controls (235.08,105.7) and (233.81,104.43) .. (233.81,102.87) -- cycle ;
\draw  [color={rgb, 255:red, 0; green, 0; blue, 0 }  ,draw opacity=1 ][fill={rgb, 255:red, 0; green, 0; blue, 0 }  ,fill opacity=1 ] (233.81,132.51) .. controls (233.81,130.95) and (235.08,129.68) .. (236.64,129.68) .. controls (238.2,129.68) and (239.47,130.95) .. (239.47,132.51) .. controls (239.47,134.08) and (238.2,135.34) .. (236.64,135.34) .. controls (235.08,135.34) and (233.81,134.08) .. (233.81,132.51) -- cycle ;
\draw  [color={rgb, 255:red, 0; green, 0; blue, 0 }  ,draw opacity=1 ][fill={rgb, 255:red, 0; green, 0; blue, 0 }  ,fill opacity=1 ] (233.81,162.16) .. controls (233.81,160.6) and (235.08,159.33) .. (236.64,159.33) .. controls (238.2,159.33) and (239.47,160.6) .. (239.47,162.16) .. controls (239.47,163.72) and (238.2,164.99) .. (236.64,164.99) .. controls (235.08,164.99) and (233.81,163.72) .. (233.81,162.16) -- cycle ;

\draw [color={rgb, 255:red, 0; green, 0; blue, 0 }  ,draw opacity=1 ][fill={rgb, 255:red, 0; green, 0; blue, 0 }  ,fill opacity=1 ]   (239.47,102.87) -- (239.47,162.16) ;
\draw  [color={rgb, 255:red, 0; green, 0; blue, 0 }  ,draw opacity=1 ][fill={rgb, 255:red, 0; green, 0; blue, 0 }  ,fill opacity=1 ] (240.16,102.87) .. controls (240.16,101.3) and (241.43,100.03) .. (242.99,100.03) .. controls (244.56,100.03) and (245.82,101.3) .. (245.82,102.87) .. controls (245.82,104.43) and (244.56,105.7) .. (242.99,105.7) .. controls (241.43,105.7) and (240.16,104.43) .. (240.16,102.87) -- cycle ;
\draw  [color={rgb, 255:red, 0; green, 0; blue, 0 }  ,draw opacity=1 ][fill={rgb, 255:red, 0; green, 0; blue, 0 }  ,fill opacity=1 ] (240.16,132.51) .. controls (240.16,130.95) and (241.43,129.68) .. (242.99,129.68) .. controls (244.56,129.68) and (245.82,130.95) .. (245.82,132.51) .. controls (245.82,134.08) and (244.56,135.34) .. (242.99,135.34) .. controls (241.43,135.34) and (240.16,134.08) .. (240.16,132.51) -- cycle ;
\draw  [color={rgb, 255:red, 0; green, 0; blue, 0 }  ,draw opacity=1 ][fill={rgb, 255:red, 0; green, 0; blue, 0 }  ,fill opacity=1 ] (240.16,162.16) .. controls (240.16,160.6) and (241.43,159.33) .. (242.99,159.33) .. controls (244.56,159.33) and (245.82,160.6) .. (245.82,162.16) .. controls (245.82,163.72) and (244.56,164.99) .. (242.99,164.99) .. controls (241.43,164.99) and (240.16,163.72) .. (240.16,162.16) -- cycle ;

\draw  [color={rgb, 255:red, 255; green, 255; blue, 255 }  ,draw opacity=1 ][fill={rgb, 255:red, 155; green, 155; blue, 155 }  ,fill opacity=0.5 ] (567.51,135.35) .. controls (567.51,100.16) and (596.04,71.63) .. (631.23,71.63) .. controls (666.42,71.63) and (694.95,100.16) .. (694.95,135.35) .. controls (694.95,170.54) and (666.42,199.06) .. (631.23,199.06) .. controls (596.04,199.06) and (567.51,170.54) .. (567.51,135.35) -- cycle ;
\draw  [color={rgb, 255:red, 0; green, 0; blue, 0 }  ,draw opacity=1 ][fill={rgb, 255:red, 0; green, 0; blue, 0 }  ,fill opacity=1 ] (628.4,98.13) .. controls (628.4,96.57) and (629.67,95.3) .. (631.23,95.3) .. controls (632.79,95.3) and (634.06,96.57) .. (634.06,98.13) .. controls (634.06,99.69) and (632.79,100.96) .. (631.23,100.96) .. controls (629.67,100.96) and (628.4,99.69) .. (628.4,98.13) -- cycle ;
\draw  [color={rgb, 255:red, 0; green, 0; blue, 0 }  ,draw opacity=1 ][fill={rgb, 255:red, 0; green, 0; blue, 0 }  ,fill opacity=1 ] (628.4,127.78) .. controls (628.4,126.21) and (629.67,124.95) .. (631.23,124.95) .. controls (632.79,124.95) and (634.06,126.21) .. (634.06,127.78) .. controls (634.06,129.34) and (632.79,130.61) .. (631.23,130.61) .. controls (629.67,130.61) and (628.4,129.34) .. (628.4,127.78) -- cycle ;
\draw  [color={rgb, 255:red, 0; green, 0; blue, 0 }  ,draw opacity=1 ][fill={rgb, 255:red, 0; green, 0; blue, 0 }  ,fill opacity=1 ] (628.4,157.43) .. controls (628.4,155.86) and (629.67,154.59) .. (631.23,154.59) .. controls (632.79,154.59) and (634.06,155.86) .. (634.06,157.43) .. controls (634.06,158.99) and (632.79,160.26) .. (631.23,160.26) .. controls (629.67,160.26) and (628.4,158.99) .. (628.4,157.43) -- cycle ;

\draw [color={rgb, 255:red, 0; green, 0; blue, 0 }  ,draw opacity=1 ][fill={rgb, 255:red, 0; green, 0; blue, 0 }  ,fill opacity=1 ]   (631.23,98.13) -- (631.23,157.43) ;

\draw    (139.59,135.35) -- (176.1,135.34) ;
\draw    (530.98,135.34) -- (567.49,135.34) ;
\draw    (378.33,135.34) -- (403.55,135.34) ;
\draw    (303.53,135.34) -- (328.75,135.34) ;

\draw (57.37,91.38) node [anchor=north west][inner sep=0.75pt]    {$t_{1}$};
\draw (66.37,161.16) node [anchor=north west][inner sep=0.75pt]    {$T_{1}$};
\draw (66.37,210.36) node [anchor=north west][inner sep=0.75pt]    {$P_{0}$};
\draw (230.32,210.37) node [anchor=north west][inner sep=0.75pt]    {$P_{1}$};
\draw (243.88,165.18) node [anchor=north west][inner sep=0.75pt]    {$T_{2}$};
\draw (216.75,165.18) node [anchor=north west][inner sep=0.75pt]    {$T_{1}$};
\draw (249.14,94.22) node [anchor=north west][inner sep=0.75pt]    {$t_{2}$};
\draw (216.49,94.22) node [anchor=north west][inner sep=0.75pt]    {$t_{1}$};
\draw (471.33,165.18) node [anchor=north west][inner sep=0.75pt]    {$T_{l}$};
\draw (444.2,165.18) node [anchor=north west][inner sep=0.75pt]    {$T_{l-1}$};
\draw (621.73,210.36) node [anchor=north west][inner sep=0.75pt]    {$P_{l}$};
\draw (621.73,161.16) node [anchor=north west][inner sep=0.75pt]    {$T_{l}$};
\draw (612.72,91.38) node [anchor=north west][inner sep=0.75pt]    {$t_{l}$};
\draw (437.28,95.55) node [anchor=north west][inner sep=0.75pt]    {$t_{l-1}$};
\draw (477.26,95.55) node [anchor=north west][inner sep=0.75pt]    {$t_{l}$};
\draw (457.77,210.37) node [anchor=north west][inner sep=0.75pt]    {$P_{l-1}$};
\draw (341.04,143.2) node [anchor=north west][inner sep=0.75pt]    {$\dotsc $};

\end{tikzpicture}

%% file: figures/fig_many_intersecting_triangles.tex
\tikzset{every picture/.style={line width=0.75pt}} 

\begin{tikzpicture}[x=0.75pt,y=0.75pt,yscale=-1,xscale=1]

\draw  [color={rgb, 255:red, 255; green, 255; blue, 255 }  ,draw opacity=1 ][fill={rgb, 255:red, 155; green, 155; blue, 155 }  ,fill opacity=0.5 ] (133.65,122.01) .. controls (133.65,98.05) and (153.08,78.63) .. (177.04,78.63) .. controls (201,78.63) and (220.43,98.05) .. (220.43,122.01) .. controls (220.43,145.97) and (201,165.4) .. (177.04,165.4) .. controls (153.08,165.4) and (133.65,145.97) .. (133.65,122.01) -- cycle ;
\draw [color={rgb, 255:red, 0; green, 0; blue, 0 }  ,draw opacity=1 ][fill={rgb, 255:red, 0; green, 0; blue, 0 }  ,fill opacity=1 ]   (179.2,104.53) -- (199.39,139.49) ;
\draw [color={rgb, 255:red, 0; green, 0; blue, 0 }  ,draw opacity=1 ][fill={rgb, 255:red, 0; green, 0; blue, 0 }  ,fill opacity=1 ]   (174.88,104.53) -- (154.69,139.49) ;

\draw  [color={rgb, 255:red, 0; green, 0; blue, 0 }  ,draw opacity=1 ][fill={rgb, 255:red, 0; green, 0; blue, 0 }  ,fill opacity=1 ] (173.21,103.77) .. controls (173.74,102.85) and (174.92,102.53) .. (175.84,103.06) .. controls (176.76,103.6) and (177.08,104.78) .. (176.55,105.7) .. controls (176.01,106.62) and (174.83,106.94) .. (173.91,106.4) .. controls (172.99,105.87) and (172.67,104.69) .. (173.21,103.77) -- cycle ;
\draw  [color={rgb, 255:red, 0; green, 0; blue, 0 }  ,draw opacity=1 ][fill={rgb, 255:red, 0; green, 0; blue, 0 }  ,fill opacity=1 ] (163.11,121.25) .. controls (163.65,120.33) and (164.82,120.01) .. (165.75,120.55) .. controls (166.67,121.08) and (166.99,122.26) .. (166.45,123.18) .. controls (165.92,124.1) and (164.74,124.42) .. (163.82,123.89) .. controls (162.9,123.35) and (162.58,122.18) .. (163.11,121.25) -- cycle ;
\draw  [color={rgb, 255:red, 0; green, 0; blue, 0 }  ,draw opacity=1 ][fill={rgb, 255:red, 0; green, 0; blue, 0 }  ,fill opacity=1 ] (153.02,138.74) .. controls (153.55,137.81) and (154.73,137.5) .. (155.65,138.03) .. controls (156.58,138.56) and (156.89,139.74) .. (156.36,140.66) .. controls (155.83,141.59) and (154.65,141.9) .. (153.72,141.37) .. controls (152.8,140.84) and (152.49,139.66) .. (153.02,138.74) -- cycle ;

\draw  [color={rgb, 255:red, 0; green, 0; blue, 0 }  ,draw opacity=1 ][fill={rgb, 255:red, 0; green, 0; blue, 0 }  ,fill opacity=1 ] (201.06,138.32) .. controls (201.59,139.25) and (201.28,140.43) .. (200.35,140.96) .. controls (199.43,141.49) and (198.25,141.17) .. (197.72,140.25) .. controls (197.19,139.33) and (197.5,138.15) .. (198.43,137.62) .. controls (199.35,137.09) and (200.53,137.4) .. (201.06,138.32) -- cycle ;
\draw  [color={rgb, 255:red, 0; green, 0; blue, 0 }  ,draw opacity=1 ][fill={rgb, 255:red, 0; green, 0; blue, 0 }  ,fill opacity=1 ] (190.97,120.84) .. controls (191.5,121.76) and (191.18,122.94) .. (190.26,123.47) .. controls (189.34,124.01) and (188.16,123.69) .. (187.63,122.77) .. controls (187.09,121.85) and (187.41,120.67) .. (188.33,120.14) .. controls (189.25,119.6) and (190.43,119.92) .. (190.97,120.84) -- cycle ;
\draw  [color={rgb, 255:red, 0; green, 0; blue, 0 }  ,draw opacity=1 ][fill={rgb, 255:red, 0; green, 0; blue, 0 }  ,fill opacity=1 ] (180.87,103.36) .. controls (181.4,104.28) and (181.09,105.46) .. (180.17,105.99) .. controls (179.24,106.52) and (178.07,106.21) .. (177.53,105.29) .. controls (177,104.36) and (177.32,103.18) .. (178.24,102.65) .. controls (179.16,102.12) and (180.34,102.44) .. (180.87,103.36) -- cycle ;

\draw  [color={rgb, 255:red, 255; green, 255; blue, 255 }  ,draw opacity=1 ][fill={rgb, 255:red, 155; green, 155; blue, 155 }  ,fill opacity=0.5 ] (116.72,122.01) .. controls (116.72,98.53) and (97.69,79.49) .. (74.21,79.49) .. controls (50.72,79.49) and (31.69,98.53) .. (31.69,122.01) .. controls (31.69,145.49) and (50.72,164.53) .. (74.21,164.53) .. controls (97.69,164.53) and (116.72,145.49) .. (116.72,122.01) -- cycle ;
\draw [fill={rgb, 255:red, 0; green, 0; blue, 0 }  ,fill opacity=1 ]   (74.21,141.79) -- (74.21,102.23) ;
\draw  [fill={rgb, 255:red, 0; green, 0; blue, 0 }  ,fill opacity=1 ] (76.1,141.8) .. controls (76.1,142.84) and (75.25,143.68) .. (74.21,143.68) .. controls (73.16,143.68) and (72.32,142.84) .. (72.32,141.8) .. controls (72.32,140.75) and (73.16,139.91) .. (74.21,139.91) .. controls (75.25,139.91) and (76.1,140.75) .. (76.1,141.8) -- cycle ;
\draw  [fill={rgb, 255:red, 0; green, 0; blue, 0 }  ,fill opacity=1 ] (76.1,122.01) .. controls (76.1,123.05) and (75.25,123.9) .. (74.21,123.9) .. controls (73.16,123.9) and (72.32,123.05) .. (72.32,122.01) .. controls (72.32,120.97) and (73.16,120.12) .. (74.21,120.12) .. controls (75.25,120.12) and (76.1,120.97) .. (76.1,122.01) -- cycle ;
\draw  [fill={rgb, 255:red, 0; green, 0; blue, 0 }  ,fill opacity=1 ] (76.1,102.23) .. controls (76.1,103.27) and (75.25,104.12) .. (74.21,104.12) .. controls (73.16,104.12) and (72.32,103.27) .. (72.32,102.23) .. controls (72.32,101.18) and (73.16,100.34) .. (74.21,100.34) .. controls (75.25,100.34) and (76.1,101.18) .. (76.1,102.23) -- cycle ;

\draw  [color={rgb, 255:red, 255; green, 255; blue, 255 }  ,draw opacity=1 ][fill={rgb, 255:red, 155; green, 155; blue, 155 }  ,fill opacity=0.5 ] (237.36,122.01) .. controls (237.36,98.05) and (256.78,78.63) .. (280.74,78.63) .. controls (304.7,78.63) and (324.13,98.05) .. (324.13,122.01) .. controls (324.13,145.97) and (304.7,165.4) .. (280.74,165.4) .. controls (256.78,165.4) and (237.36,145.97) .. (237.36,122.01) -- cycle ;
\draw [color={rgb, 255:red, 0; green, 0; blue, 0 }  ,draw opacity=1 ][fill={rgb, 255:red, 0; green, 0; blue, 0 }  ,fill opacity=1 ]   (282.9,104.53) -- (303.09,139.49) ;
\draw [color={rgb, 255:red, 0; green, 0; blue, 0 }  ,draw opacity=1 ][fill={rgb, 255:red, 0; green, 0; blue, 0 }  ,fill opacity=1 ]   (278.58,104.53) -- (258.39,139.49) ;

\draw  [color={rgb, 255:red, 0; green, 0; blue, 0 }  ,draw opacity=1 ][fill={rgb, 255:red, 0; green, 0; blue, 0 }  ,fill opacity=1 ] (276.91,103.77) .. controls (277.44,102.85) and (278.62,102.53) .. (279.54,103.06) .. controls (280.46,103.6) and (280.78,104.78) .. (280.25,105.7) .. controls (279.72,106.62) and (278.54,106.94) .. (277.61,106.4) .. controls (276.69,105.87) and (276.38,104.69) .. (276.91,103.77) -- cycle ;
\draw  [color={rgb, 255:red, 0; green, 0; blue, 0 }  ,draw opacity=1 ][fill={rgb, 255:red, 0; green, 0; blue, 0 }  ,fill opacity=1 ] (266.81,121.25) .. controls (267.35,120.33) and (268.53,120.01) .. (269.45,120.55) .. controls (270.37,121.08) and (270.69,122.26) .. (270.15,123.18) .. controls (269.62,124.1) and (268.44,124.42) .. (267.52,123.89) .. controls (266.6,123.35) and (266.28,122.18) .. (266.81,121.25) -- cycle ;
\draw  [color={rgb, 255:red, 0; green, 0; blue, 0 }  ,draw opacity=1 ][fill={rgb, 255:red, 0; green, 0; blue, 0 }  ,fill opacity=1 ] (256.72,138.74) .. controls (257.25,137.81) and (258.43,137.5) .. (259.35,138.03) .. controls (260.28,138.56) and (260.59,139.74) .. (260.06,140.66) .. controls (259.53,141.59) and (258.35,141.9) .. (257.43,141.37) .. controls (256.5,140.84) and (256.19,139.66) .. (256.72,138.74) -- cycle ;

\draw  [color={rgb, 255:red, 0; green, 0; blue, 0 }  ,draw opacity=1 ][fill={rgb, 255:red, 0; green, 0; blue, 0 }  ,fill opacity=1 ] (304.76,138.32) .. controls (305.29,139.25) and (304.98,140.43) .. (304.06,140.96) .. controls (303.13,141.49) and (301.95,141.17) .. (301.42,140.25) .. controls (300.89,139.33) and (301.21,138.15) .. (302.13,137.62) .. controls (303.05,137.09) and (304.23,137.4) .. (304.76,138.32) -- cycle ;
\draw  [color={rgb, 255:red, 0; green, 0; blue, 0 }  ,draw opacity=1 ][fill={rgb, 255:red, 0; green, 0; blue, 0 }  ,fill opacity=1 ] (294.67,120.84) .. controls (295.2,121.76) and (294.88,122.94) .. (293.96,123.47) .. controls (293.04,124.01) and (291.86,123.69) .. (291.33,122.77) .. controls (290.8,121.85) and (291.11,120.67) .. (292.03,120.14) .. controls (292.96,119.6) and (294.14,119.92) .. (294.67,120.84) -- cycle ;
\draw  [color={rgb, 255:red, 0; green, 0; blue, 0 }  ,draw opacity=1 ][fill={rgb, 255:red, 0; green, 0; blue, 0 }  ,fill opacity=1 ] (284.57,103.36) .. controls (285.11,104.28) and (284.79,105.46) .. (283.87,105.99) .. controls (282.95,106.52) and (281.77,106.21) .. (281.23,105.29) .. controls (280.7,104.36) and (281.02,103.18) .. (281.94,102.65) .. controls (282.86,102.12) and (284.04,102.44) .. (284.57,103.36) -- cycle ;

\draw  [color={rgb, 255:red, 255; green, 255; blue, 255 }  ,draw opacity=1 ][fill={rgb, 255:red, 155; green, 155; blue, 155 }  ,fill opacity=0.5 ] (384.99,122.01) .. controls (384.99,98.05) and (404.41,78.63) .. (428.37,78.63) .. controls (452.33,78.63) and (471.76,98.05) .. (471.76,122.01) .. controls (471.76,145.97) and (452.33,165.4) .. (428.37,165.4) .. controls (404.41,165.4) and (384.99,145.97) .. (384.99,122.01) -- cycle ;
\draw [color={rgb, 255:red, 0; green, 0; blue, 0 }  ,draw opacity=1 ][fill={rgb, 255:red, 0; green, 0; blue, 0 }  ,fill opacity=1 ]   (430.54,104.53) -- (450.72,139.49) ;
\draw [color={rgb, 255:red, 0; green, 0; blue, 0 }  ,draw opacity=1 ][fill={rgb, 255:red, 0; green, 0; blue, 0 }  ,fill opacity=1 ]   (426.21,104.53) -- (406.02,139.49) ;

\draw  [color={rgb, 255:red, 0; green, 0; blue, 0 }  ,draw opacity=1 ][fill={rgb, 255:red, 0; green, 0; blue, 0 }  ,fill opacity=1 ] (424.54,103.77) .. controls (425.07,102.85) and (426.25,102.53) .. (427.17,103.06) .. controls (428.1,103.6) and (428.41,104.78) .. (427.88,105.7) .. controls (427.35,106.62) and (426.17,106.94) .. (425.25,106.4) .. controls (424.32,105.87) and (424.01,104.69) .. (424.54,103.77) -- cycle ;
\draw  [color={rgb, 255:red, 0; green, 0; blue, 0 }  ,draw opacity=1 ][fill={rgb, 255:red, 0; green, 0; blue, 0 }  ,fill opacity=1 ] (414.45,121.25) .. controls (414.98,120.33) and (416.16,120.01) .. (417.08,120.55) .. controls (418,121.08) and (418.32,122.26) .. (417.79,123.18) .. controls (417.25,124.1) and (416.07,124.42) .. (415.15,123.89) .. controls (414.23,123.35) and (413.91,122.18) .. (414.45,121.25) -- cycle ;
\draw  [color={rgb, 255:red, 0; green, 0; blue, 0 }  ,draw opacity=1 ][fill={rgb, 255:red, 0; green, 0; blue, 0 }  ,fill opacity=1 ] (404.35,138.74) .. controls (404.88,137.81) and (406.06,137.5) .. (406.99,138.03) .. controls (407.91,138.56) and (408.22,139.74) .. (407.69,140.66) .. controls (407.16,141.59) and (405.98,141.9) .. (405.06,141.37) .. controls (404.14,140.84) and (403.82,139.66) .. (404.35,138.74) -- cycle ;

\draw  [color={rgb, 255:red, 0; green, 0; blue, 0 }  ,draw opacity=1 ][fill={rgb, 255:red, 0; green, 0; blue, 0 }  ,fill opacity=1 ] (452.39,138.32) .. controls (452.93,139.25) and (452.61,140.43) .. (451.69,140.96) .. controls (450.76,141.49) and (449.59,141.17) .. (449.05,140.25) .. controls (448.52,139.33) and (448.84,138.15) .. (449.76,137.62) .. controls (450.68,137.09) and (451.86,137.4) .. (452.39,138.32) -- cycle ;
\draw  [color={rgb, 255:red, 0; green, 0; blue, 0 }  ,draw opacity=1 ][fill={rgb, 255:red, 0; green, 0; blue, 0 }  ,fill opacity=1 ] (442.3,120.84) .. controls (442.83,121.76) and (442.52,122.94) .. (441.59,123.47) .. controls (440.67,124.01) and (439.49,123.69) .. (438.96,122.77) .. controls (438.43,121.85) and (438.74,120.67) .. (439.67,120.14) .. controls (440.59,119.6) and (441.77,119.92) .. (442.3,120.84) -- cycle ;
\draw  [color={rgb, 255:red, 0; green, 0; blue, 0 }  ,draw opacity=1 ][fill={rgb, 255:red, 0; green, 0; blue, 0 }  ,fill opacity=1 ] (432.21,103.36) .. controls (432.74,104.28) and (432.42,105.46) .. (431.5,105.99) .. controls (430.58,106.52) and (429.4,106.21) .. (428.87,105.29) .. controls (428.33,104.36) and (428.65,103.18) .. (429.57,102.65) .. controls (430.49,102.12) and (431.67,102.44) .. (432.21,103.36) -- cycle ;

\draw  [color={rgb, 255:red, 255; green, 255; blue, 255 }  ,draw opacity=1 ][fill={rgb, 255:red, 155; green, 155; blue, 155 }  ,fill opacity=0.5 ] (573.72,122.01) .. controls (573.72,98.53) and (554.69,79.49) .. (531.21,79.49) .. controls (507.72,79.49) and (488.69,98.53) .. (488.69,122.01) .. controls (488.69,145.49) and (507.72,164.53) .. (531.21,164.53) .. controls (554.69,164.53) and (573.72,145.49) .. (573.72,122.01) -- cycle ;
\draw [fill={rgb, 255:red, 0; green, 0; blue, 0 }  ,fill opacity=1 ]   (531.21,141.79) -- (531.21,102.23) ;
\draw  [fill={rgb, 255:red, 0; green, 0; blue, 0 }  ,fill opacity=1 ] (533.1,141.8) .. controls (533.1,142.84) and (532.25,143.68) .. (531.21,143.68) .. controls (530.16,143.68) and (529.32,142.84) .. (529.32,141.8) .. controls (529.32,140.75) and (530.16,139.91) .. (531.21,139.91) .. controls (532.25,139.91) and (533.1,140.75) .. (533.1,141.8) -- cycle ;
\draw  [fill={rgb, 255:red, 0; green, 0; blue, 0 }  ,fill opacity=1 ] (533.1,122.01) .. controls (533.1,123.05) and (532.25,123.9) .. (531.21,123.9) .. controls (530.16,123.9) and (529.32,123.05) .. (529.32,122.01) .. controls (529.32,120.97) and (530.16,120.12) .. (531.21,120.12) .. controls (532.25,120.12) and (533.1,120.97) .. (533.1,122.01) -- cycle ;
\draw  [fill={rgb, 255:red, 0; green, 0; blue, 0 }  ,fill opacity=1 ] (533.1,102.23) .. controls (533.1,103.27) and (532.25,104.12) .. (531.21,104.12) .. controls (530.16,104.12) and (529.32,103.27) .. (529.32,102.23) .. controls (529.32,101.18) and (530.16,100.34) .. (531.21,100.34) .. controls (532.25,100.34) and (533.1,101.18) .. (533.1,102.23) -- cycle ;

\draw    (116.72,122.01) -- (133.65,122.01) ;
\draw    (220.43,122.01) -- (237.36,122.01) ;
\draw    (324.13,122.01) -- (341.06,122.01) ;
\draw    (368.06,122.01) -- (384.99,122.01) ;
\draw    (471.76,122.01) -- (488.69,122.01) ;

\draw (342.06,114.91) node [anchor=north west][inner sep=0.75pt]    {$\dotsc $};

\end{tikzpicture}

%% file: figures/fig_long_path.tex
\tikzset{every picture/.style={line width=0.75pt}} 

\begin{tikzpicture}[x=0.75pt,y=0.75pt,yscale=-1,xscale=1]

\draw  [color={rgb, 255:red, 255; green, 255; blue, 255 }  ,draw opacity=1 ][fill={rgb, 255:red, 155; green, 155; blue, 155 }  ,fill opacity=0.5 ] (22.47,159.01) .. controls (22.47,142.34) and (35.99,128.82) .. (52.66,128.82) .. controls (69.32,128.82) and (82.84,142.34) .. (82.84,159.01) .. controls (82.84,175.67) and (69.32,189.19) .. (52.66,189.19) .. controls (35.99,189.19) and (22.47,175.67) .. (22.47,159.01) -- cycle ;

\draw  [color={rgb, 255:red, 255; green, 255; blue, 255 }  ,draw opacity=1 ][fill={rgb, 255:red, 155; green, 155; blue, 155 }  ,fill opacity=0.5 ] (112.39,159.01) .. controls (112.39,142.34) and (125.9,128.82) .. (142.57,128.82) .. controls (159.24,128.82) and (172.75,142.34) .. (172.75,159.01) .. controls (172.75,175.67) and (159.24,189.19) .. (142.57,189.19) .. controls (125.9,189.19) and (112.39,175.67) .. (112.39,159.01) -- cycle ;

\draw  [color={rgb, 255:red, 255; green, 255; blue, 255 }  ,draw opacity=1 ][fill={rgb, 255:red, 155; green, 155; blue, 155 }  ,fill opacity=0.5 ] (292.22,159.01) .. controls (292.22,142.34) and (305.73,128.82) .. (322.4,128.82) .. controls (339.07,128.82) and (352.58,142.34) .. (352.58,159.01) .. controls (352.58,175.67) and (339.07,189.19) .. (322.4,189.19) .. controls (305.73,189.19) and (292.22,175.67) .. (292.22,159.01) -- cycle ;

\draw  [color={rgb, 255:red, 255; green, 255; blue, 255 }  ,draw opacity=1 ][fill={rgb, 255:red, 155; green, 155; blue, 155 }  ,fill opacity=0.5 ] (202.3,159.01) .. controls (202.3,142.34) and (215.81,128.82) .. (232.48,128.82) .. controls (249.15,128.82) and (262.67,142.34) .. (262.67,159.01) .. controls (262.67,175.67) and (249.15,189.19) .. (232.48,189.19) .. controls (215.81,189.19) and (202.3,175.67) .. (202.3,159.01) -- cycle ;

\draw  [color={rgb, 255:red, 255; green, 255; blue, 255 }  ,draw opacity=1 ][fill={rgb, 255:red, 155; green, 155; blue, 155 }  ,fill opacity=0.5 ] (439.68,159.01) .. controls (439.68,142.34) and (453.19,128.82) .. (469.86,128.82) .. controls (486.53,128.82) and (500.04,142.34) .. (500.04,159.01) .. controls (500.04,175.67) and (486.53,189.19) .. (469.86,189.19) .. controls (453.19,189.19) and (439.68,175.67) .. (439.68,159.01) -- cycle ;

\draw  [color={rgb, 255:red, 255; green, 255; blue, 255 }  ,draw opacity=1 ][fill={rgb, 255:red, 155; green, 155; blue, 155 }  ,fill opacity=0.5 ] (529.59,159.01) .. controls (529.59,142.34) and (543.11,128.82) .. (559.77,128.82) .. controls (576.44,128.82) and (589.96,142.34) .. (589.96,159.01) .. controls (589.96,175.67) and (576.44,189.19) .. (559.77,189.19) .. controls (543.11,189.19) and (529.59,175.67) .. (529.59,159.01) -- cycle ;

\draw  [color={rgb, 255:red, 255; green, 255; blue, 255 }  ,draw opacity=1 ][fill={rgb, 255:red, 155; green, 155; blue, 155 }  ,fill opacity=0.5 ] (188.47,223.67) .. controls (188.47,212.48) and (197.55,203.41) .. (208.74,203.41) .. controls (219.93,203.41) and (229,212.48) .. (229,223.67) .. controls (229,234.87) and (219.93,243.94) .. (208.74,243.94) .. controls (197.55,243.94) and (188.47,234.87) .. (188.47,223.67) -- cycle ;
\draw  [color={rgb, 255:red, 255; green, 255; blue, 255 }  ,draw opacity=1 ][fill={rgb, 255:red, 155; green, 155; blue, 155 }  ,fill opacity=0.5 ] (240.74,223.67) .. controls (240.74,212.48) and (249.81,203.41) .. (261,203.41) .. controls (272.19,203.41) and (281.26,212.48) .. (281.26,223.67) .. controls (281.26,234.87) and (272.19,243.94) .. (261,243.94) .. controls (249.81,243.94) and (240.74,234.87) .. (240.74,223.67) -- cycle ;
\draw  [color={rgb, 255:red, 255; green, 255; blue, 255 }  ,draw opacity=1 ][fill={rgb, 255:red, 155; green, 155; blue, 155 }  ,fill opacity=0.5 ] (188.47,278.59) .. controls (188.47,267.4) and (197.55,258.32) .. (208.74,258.32) .. controls (219.93,258.32) and (229,267.4) .. (229,278.59) .. controls (229,289.78) and (219.93,298.85) .. (208.74,298.85) .. controls (197.55,298.85) and (188.47,289.78) .. (188.47,278.59) -- cycle ;

\draw  [color={rgb, 255:red, 255; green, 255; blue, 255 }  ,draw opacity=1 ][fill={rgb, 255:red, 155; green, 155; blue, 155 }  ,fill opacity=0.5 ] (619.47,159.01) .. controls (619.47,142.34) and (632.99,128.82) .. (649.66,128.82) .. controls (666.32,128.82) and (679.84,142.34) .. (679.84,159.01) .. controls (679.84,175.67) and (666.32,189.19) .. (649.66,189.19) .. controls (632.99,189.19) and (619.47,175.67) .. (619.47,159.01) -- cycle ;

\draw    (112.39,159.01) -- (82.84,159.01) ;
\draw    (202.3,159.01) -- (172.75,159.01) ;
\draw    (292.22,159.01) -- (262.67,159.01) ;
\draw    (382.13,159.01) -- (352.58,159.01) ;
\draw    (439.68,159.01) -- (410.13,159.01) ;
\draw    (529.59,159.01) -- (500.04,159.01) ;
\draw    (619.51,159.01) -- (589.96,159.01) ;
\draw    (559.77,128.82) -- (559.77,119.05) ;
\draw  [color={rgb, 255:red, 255; green, 255; blue, 255 }  ,draw opacity=1 ][fill={rgb, 255:red, 155; green, 155; blue, 155 }  ,fill opacity=0.5 ] (302.13,97.83) .. controls (302.13,86.64) and (311.21,77.57) .. (322.4,77.57) .. controls (333.59,77.57) and (342.66,86.64) .. (342.66,97.83) .. controls (342.66,109.02) and (333.59,118.1) .. (322.4,118.1) .. controls (311.21,118.1) and (302.13,109.02) .. (302.13,97.83) -- cycle ;
\draw  [color={rgb, 255:red, 255; green, 255; blue, 255 }  ,draw opacity=1 ][fill={rgb, 255:red, 155; green, 155; blue, 155 }  ,fill opacity=0.5 ] (302.13,46.59) .. controls (302.13,35.4) and (311.21,26.32) .. (322.4,26.32) .. controls (333.59,26.32) and (342.66,35.4) .. (342.66,46.59) .. controls (342.66,57.78) and (333.59,66.85) .. (322.4,66.85) .. controls (311.21,66.85) and (302.13,57.78) .. (302.13,46.59) -- cycle ;
\draw  [color={rgb, 255:red, 255; green, 255; blue, 255 }  ,draw opacity=1 ][fill={rgb, 255:red, 155; green, 155; blue, 155 }  ,fill opacity=0.5 ] (449.6,98.42) .. controls (449.6,87.23) and (458.67,78.16) .. (469.86,78.16) .. controls (481.05,78.16) and (490.12,87.23) .. (490.12,98.42) .. controls (490.12,109.62) and (481.05,118.69) .. (469.86,118.69) .. controls (458.67,118.69) and (449.6,109.62) .. (449.6,98.42) -- cycle ;
\draw  [color={rgb, 255:red, 255; green, 255; blue, 255 }  ,draw opacity=1 ][fill={rgb, 255:red, 155; green, 155; blue, 155 }  ,fill opacity=0.5 ] (539.51,98.79) .. controls (539.51,87.6) and (548.58,78.53) .. (559.77,78.53) .. controls (570.97,78.53) and (580.04,87.6) .. (580.04,98.79) .. controls (580.04,109.98) and (570.97,119.05) .. (559.77,119.05) .. controls (548.58,119.05) and (539.51,109.98) .. (539.51,98.79) -- cycle ;
\draw  [dash pattern={on 4.5pt off 4.5pt}] (18.27,119.39) .. controls (18.27,111.79) and (24.43,105.64) .. (32.03,105.64) -- (73.29,105.64) .. controls (80.88,105.64) and (87.04,111.79) .. (87.04,119.39) -- (87.04,198) .. controls (87.04,205.6) and (80.88,211.75) .. (73.29,211.75) -- (32.03,211.75) .. controls (24.43,211.75) and (18.27,205.6) .. (18.27,198) -- cycle ;
\draw  [dash pattern={on 4.5pt off 4.5pt}] (108.19,119.7) .. controls (108.19,112.1) and (114.34,105.95) .. (121.94,105.95) -- (163.2,105.95) .. controls (170.8,105.95) and (176.95,112.1) .. (176.95,119.7) -- (176.95,198.31) .. controls (176.95,205.91) and (170.8,212.06) .. (163.2,212.06) -- (121.94,212.06) .. controls (114.34,212.06) and (108.19,205.91) .. (108.19,198.31) -- cycle ;
\draw  [dash pattern={on 4.5pt off 4.5pt}] (185.23,124.89) .. controls (185.23,113.98) and (194.08,105.14) .. (204.99,105.14) -- (264.25,105.14) .. controls (275.16,105.14) and (284,113.98) .. (284,124.89) -- (284,298.25) .. controls (284,309.16) and (275.16,318) .. (264.25,318) -- (204.99,318) .. controls (194.08,318) and (185.23,309.16) .. (185.23,298.25) -- cycle ;
\draw  [dash pattern={on 4.5pt off 4.5pt}] (289.27,36.86) .. controls (289.27,29.26) and (295.43,23.1) .. (303.03,23.1) -- (344.29,23.1) .. controls (351.88,23.1) and (358.04,29.26) .. (358.04,36.86) -- (358.04,183.75) .. controls (358.04,191.34) and (351.88,197.5) .. (344.29,197.5) -- (303.03,197.5) .. controls (295.43,197.5) and (289.27,191.34) .. (289.27,183.75) -- cycle ;
\draw  [dash pattern={on 4.5pt off 4.5pt}] (435.48,89.25) .. controls (435.48,81.66) and (441.63,75.5) .. (449.23,75.5) -- (490.49,75.5) .. controls (498.09,75.5) and (504.24,81.66) .. (504.24,89.25) -- (504.24,183.75) .. controls (504.24,191.34) and (498.09,197.5) .. (490.49,197.5) -- (449.23,197.5) .. controls (441.63,197.5) and (435.48,191.34) .. (435.48,183.75) -- cycle ;
\draw  [dash pattern={on 4.5pt off 4.5pt}] (524.48,89.25) .. controls (524.48,81.66) and (530.63,75.5) .. (538.23,75.5) -- (579.49,75.5) .. controls (587.09,75.5) and (593.24,81.66) .. (593.24,89.25) -- (593.24,183.75) .. controls (593.24,191.34) and (587.09,197.5) .. (579.49,197.5) -- (538.23,197.5) .. controls (530.63,197.5) and (524.48,191.34) .. (524.48,183.75) -- cycle ;
\draw  [dash pattern={on 4.5pt off 4.5pt}] (614.19,105.14) .. controls (614.19,97.54) and (620.34,91.38) .. (627.94,91.38) -- (669.2,91.38) .. controls (676.8,91.38) and (682.95,97.54) .. (682.95,105.14) -- (682.95,183.75) .. controls (682.95,191.34) and (676.8,197.5) .. (669.2,197.5) -- (627.94,197.5) .. controls (620.34,197.5) and (614.19,191.34) .. (614.19,183.75) -- cycle ;
\draw    (322.4,66.85) -- (322.4,77.57) ;
\draw    (322.4,118.1) -- (322.4,128.82) ;
\draw    (208.74,243.94) -- (208.74,258.32) ;
\draw    (232.48,189.19) -- (208.74,203.41) ;
\draw    (232.48,189.19) -- (261,203.41) ;
\draw    (469.86,118.69) -- (469.86,128.82) ;

\draw (383.13,158.41) node [anchor=north west][inner sep=0.75pt]    {$\dotsc $};
\draw (44,219.4) node [anchor=north west][inner sep=0.75pt]    {$P_{0}$};
\draw (130,219.4) node [anchor=north west][inner sep=0.75pt]    {$P_{1}$};
\draw (222,323.9) node [anchor=north west][inner sep=0.75pt]    {$P_{2}$};
\draw (311,205.4) node [anchor=north west][inner sep=0.75pt]    {$P_{3}$};
\draw (456,205.4) node [anchor=north west][inner sep=0.75pt]    {$P_{l-2}$};
\draw (545,205.4) node [anchor=north west][inner sep=0.75pt]    {$P_{l-1}$};
\draw (639,205.4) node [anchor=north west][inner sep=0.75pt]    {$P_{l}$};
\draw (39.66,150.41) node [anchor=north west][inner sep=0.75pt]    {$M_{0}$};
\draw (130.07,150.91) node [anchor=north west][inner sep=0.75pt]    {$M_{1}$};
\draw (219.98,150.91) node [anchor=north west][inner sep=0.75pt]    {$M_{2}$};
\draw (309.9,150.91) node [anchor=north west][inner sep=0.75pt]    {$M_{3}$};
\draw (452.36,150.91) node [anchor=north west][inner sep=0.75pt]    {$M_{l-2}$};
\draw (542.27,150.91) node [anchor=north west][inner sep=0.75pt]    {$M_{l-1}$};
\draw (638.66,150.91) node [anchor=north west][inner sep=0.75pt]    {$M_{l}$};

\end{tikzpicture}

%% file: figures/fig_pair-up.tex
\tikzset{every picture/.style={line width=0.75pt}} 

\begin{tikzpicture}[x=0.75pt,y=0.75pt,yscale=-1,xscale=1]

\draw  [color={rgb, 255:red, 255; green, 255; blue, 255 }  ,draw opacity=1 ][fill={rgb, 255:red, 155; green, 155; blue, 155 }  ,fill opacity=0.5 ] (222.29,289.13) .. controls (222.29,271.92) and (236.24,257.97) .. (253.45,257.97) .. controls (270.66,257.97) and (284.61,271.92) .. (284.61,289.13) .. controls (284.61,306.34) and (270.66,320.29) .. (253.45,320.29) .. controls (236.24,320.29) and (222.29,306.34) .. (222.29,289.13) -- cycle ;
\draw    (272.21,289.14) -- (253.8,307.52) ;
\draw  [dash pattern={on 4.5pt off 4.5pt}]  (244.61,270.48) -- (272.21,289.14) ;
\draw  [dash pattern={on 4.5pt off 4.5pt}]  (234.69,297.73) -- (250,307.78) ;
\draw  [dash pattern={on 4.5pt off 4.5pt}]  (234.69,297.73) -- (272.21,289.14) ;
\draw  [dash pattern={on 4.5pt off 4.5pt}]  (244.61,270.48) -- (253.8,307.52) ;
\draw    (234.69,297.73) -- (244.61,270.48) ;
\draw  [color={rgb, 255:red, 0; green, 0; blue, 0 }  ,draw opacity=1 ][fill={rgb, 255:red, 255; green, 255; blue, 255 }  ,fill opacity=1 ] (255,308.21) .. controls (255.38,307.55) and (255.15,306.71) .. (254.49,306.32) .. controls (253.83,305.94) and (252.98,306.17) .. (252.6,306.83) .. controls (252.22,307.49) and (252.45,308.34) .. (253.11,308.72) .. controls (253.77,309.1) and (254.62,308.88) .. (255,308.21) -- cycle ;
\draw  [color={rgb, 255:red, 0; green, 0; blue, 0 }  ,draw opacity=1 ][fill={rgb, 255:red, 255; green, 255; blue, 255 }  ,fill opacity=1 ] (273.41,289.83) .. controls (273.79,289.17) and (273.57,288.32) .. (272.9,287.94) .. controls (272.24,287.56) and (271.4,287.78) .. (271.01,288.44) .. controls (270.63,289.11) and (270.86,289.95) .. (271.52,290.34) .. controls (272.18,290.72) and (273.03,290.49) .. (273.41,289.83) -- cycle ;
\draw  [color={rgb, 255:red, 0; green, 0; blue, 0 }  ,draw opacity=1 ][fill={rgb, 255:red, 255; green, 255; blue, 255 }  ,fill opacity=1 ] (264.21,299.02) .. controls (264.59,298.36) and (264.36,297.51) .. (263.7,297.13) .. controls (263.04,296.75) and (262.19,296.97) .. (261.81,297.64) .. controls (261.42,298.3) and (261.65,299.15) .. (262.31,299.53) .. controls (262.98,299.91) and (263.82,299.68) .. (264.21,299.02) -- cycle ;
\draw  [color={rgb, 255:red, 255; green, 255; blue, 255 }  ,draw opacity=1 ][fill={rgb, 255:red, 255; green, 255; blue, 255 }  ,fill opacity=1 ] (235.99,298.2) .. controls (235.73,298.92) and (234.94,299.29) .. (234.22,299.03) .. controls (233.5,298.77) and (233.13,297.98) .. (233.39,297.26) .. controls (233.65,296.54) and (234.45,296.17) .. (235.17,296.43) .. controls (235.89,296.69) and (236.26,297.49) .. (235.99,298.2) -- cycle ;
\draw  [color={rgb, 255:red, 255; green, 255; blue, 255 }  ,draw opacity=1 ][fill={rgb, 255:red, 255; green, 255; blue, 255 }  ,fill opacity=1 ] (240.95,284.58) .. controls (240.69,285.3) and (239.9,285.67) .. (239.18,285.41) .. controls (238.46,285.15) and (238.09,284.35) .. (238.35,283.63) .. controls (238.61,282.91) and (239.41,282.54) .. (240.13,282.8) .. controls (240.84,283.07) and (241.21,283.86) .. (240.95,284.58) -- cycle ;
\draw  [color={rgb, 255:red, 255; green, 255; blue, 255 }  ,draw opacity=1 ][fill={rgb, 255:red, 255; green, 255; blue, 255 }  ,fill opacity=1 ] (245.91,270.95) .. controls (245.65,271.67) and (244.86,272.04) .. (244.14,271.78) .. controls (243.42,271.52) and (243.05,270.73) .. (243.31,270.01) .. controls (243.57,269.29) and (244.37,268.92) .. (245.08,269.18) .. controls (245.8,269.44) and (246.17,270.24) .. (245.91,270.95) -- cycle ;

\draw  [color={rgb, 255:red, 255; green, 255; blue, 255 }  ,draw opacity=1 ][fill={rgb, 255:red, 155; green, 155; blue, 155 }  ,fill opacity=0.5 ] (144.87,289.13) .. controls (144.87,271.92) and (158.82,257.97) .. (176.03,257.97) .. controls (193.24,257.97) and (207.19,271.92) .. (207.19,289.13) .. controls (207.19,306.34) and (193.24,320.29) .. (176.03,320.29) .. controls (158.82,320.29) and (144.87,306.34) .. (144.87,289.13) -- cycle ;
\draw    (194.79,289.14) -- (176.37,307.52) ;
\draw  [dash pattern={on 4.5pt off 4.5pt}]  (167.19,270.48) -- (194.79,289.14) ;
\draw  [dash pattern={on 4.5pt off 4.5pt}]  (157.27,297.73) -- (172.57,307.78) ;
\draw  [dash pattern={on 4.5pt off 4.5pt}]  (157.27,297.73) -- (194.79,289.14) ;
\draw  [dash pattern={on 4.5pt off 4.5pt}]  (167.19,270.48) -- (176.37,307.52) ;
\draw  [fill={rgb, 255:red, 0; green, 0; blue, 0 }  ,fill opacity=1 ] (158.57,298.2) .. controls (158.31,298.92) and (157.51,299.29) .. (156.79,299.03) .. controls (156.08,298.77) and (155.7,297.98) .. (155.97,297.26) .. controls (156.23,296.54) and (157.02,296.17) .. (157.74,296.43) .. controls (158.46,296.69) and (158.83,297.49) .. (158.57,298.2) -- cycle ;
\draw  [fill={rgb, 255:red, 0; green, 0; blue, 0 }  ,fill opacity=1 ] (163.53,284.58) .. controls (163.27,285.3) and (162.47,285.67) .. (161.75,285.41) .. controls (161.03,285.15) and (160.66,284.35) .. (160.93,283.63) .. controls (161.19,282.91) and (161.98,282.54) .. (162.7,282.8) .. controls (163.42,283.07) and (163.79,283.86) .. (163.53,284.58) -- cycle ;
\draw  [fill={rgb, 255:red, 0; green, 0; blue, 0 }  ,fill opacity=1 ] (168.49,270.95) .. controls (168.23,271.67) and (167.43,272.04) .. (166.71,271.78) .. controls (165.99,271.52) and (165.62,270.73) .. (165.88,270.01) .. controls (166.15,269.29) and (166.94,268.92) .. (167.66,269.18) .. controls (168.38,269.44) and (168.75,270.24) .. (168.49,270.95) -- cycle ;

\draw    (157.27,297.73) -- (167.19,270.48) ;

\draw  [color={rgb, 255:red, 255; green, 255; blue, 255 }  ,draw opacity=1 ][fill={rgb, 255:red, 255; green, 255; blue, 255 }  ,fill opacity=1 ] (177.57,308.21) .. controls (177.96,307.55) and (177.73,306.71) .. (177.07,306.32) .. controls (176.4,305.94) and (175.56,306.17) .. (175.18,306.83) .. controls (174.79,307.49) and (175.02,308.34) .. (175.68,308.72) .. controls (176.34,309.1) and (177.19,308.88) .. (177.57,308.21) -- cycle ;
\draw  [color={rgb, 255:red, 255; green, 255; blue, 255 }  ,draw opacity=1 ][fill={rgb, 255:red, 255; green, 255; blue, 255 }  ,fill opacity=1 ] (195.99,289.83) .. controls (196.37,289.17) and (196.14,288.32) .. (195.48,287.94) .. controls (194.82,287.56) and (193.97,287.78) .. (193.59,288.44) .. controls (193.21,289.11) and (193.43,289.95) .. (194.09,290.34) .. controls (194.76,290.72) and (195.6,290.49) .. (195.99,289.83) -- cycle ;
\draw  [color={rgb, 255:red, 255; green, 255; blue, 255 }  ,draw opacity=1 ][fill={rgb, 255:red, 255; green, 255; blue, 255 }  ,fill opacity=1 ] (186.78,299.02) .. controls (187.16,298.36) and (186.94,297.51) .. (186.27,297.13) .. controls (185.61,296.75) and (184.76,296.97) .. (184.38,297.64) .. controls (184,298.3) and (184.23,299.15) .. (184.89,299.53) .. controls (185.55,299.91) and (186.4,299.68) .. (186.78,299.02) -- cycle ;

\draw    (394.6,289.14) -- (376.19,307.52) ;
\draw  [color={rgb, 255:red, 255; green, 255; blue, 255 }  ,draw opacity=1 ][fill={rgb, 255:red, 155; green, 155; blue, 155 }  ,fill opacity=0.5 ] (344.68,289.13) .. controls (344.68,271.92) and (358.63,257.97) .. (375.84,257.97) .. controls (393.05,257.97) and (407,271.92) .. (407,289.13) .. controls (407,306.34) and (393.05,320.29) .. (375.84,320.29) .. controls (358.63,320.29) and (344.68,306.34) .. (344.68,289.13) -- cycle ;
\draw  [dash pattern={on 4.5pt off 4.5pt}]  (367,270.48) -- (394.6,289.14) ;
\draw  [dash pattern={on 4.5pt off 4.5pt}]  (357.08,297.73) -- (372.38,307.78) ;
\draw  [dash pattern={on 4.5pt off 4.5pt}]  (357.08,297.73) -- (394.6,289.14) ;
\draw  [dash pattern={on 4.5pt off 4.5pt}]  (367,270.48) -- (376.19,307.52) ;
\draw  [fill={rgb, 255:red, 0; green, 0; blue, 0 }  ,fill opacity=1 ] (358.38,298.2) .. controls (358.12,298.92) and (357.32,299.29) .. (356.61,299.03) .. controls (355.89,298.77) and (355.52,297.98) .. (355.78,297.26) .. controls (356.04,296.54) and (356.83,296.17) .. (357.55,296.43) .. controls (358.27,296.69) and (358.64,297.49) .. (358.38,298.2) -- cycle ;
\draw  [fill={rgb, 255:red, 0; green, 0; blue, 0 }  ,fill opacity=1 ] (363.34,284.58) .. controls (363.08,285.3) and (362.28,285.67) .. (361.56,285.41) .. controls (360.85,285.15) and (360.48,284.35) .. (360.74,283.63) .. controls (361,282.91) and (361.79,282.54) .. (362.51,282.8) .. controls (363.23,283.07) and (363.6,283.86) .. (363.34,284.58) -- cycle ;
\draw  [fill={rgb, 255:red, 0; green, 0; blue, 0 }  ,fill opacity=1 ] (368.3,270.95) .. controls (368.04,271.67) and (367.24,272.04) .. (366.52,271.78) .. controls (365.81,271.52) and (365.43,270.73) .. (365.7,270.01) .. controls (365.96,269.29) and (366.75,268.92) .. (367.47,269.18) .. controls (368.19,269.44) and (368.56,270.24) .. (368.3,270.95) -- cycle ;

\draw    (357.08,297.73) -- (367,270.48) ;

\draw  [color={rgb, 255:red, 0; green, 0; blue, 0 }  ,draw opacity=1 ][fill={rgb, 255:red, 255; green, 255; blue, 255 }  ,fill opacity=1 ] (377.39,308.21) .. controls (377.77,307.55) and (377.54,306.71) .. (376.88,306.32) .. controls (376.22,305.94) and (375.37,306.17) .. (374.99,306.83) .. controls (374.6,307.49) and (374.83,308.34) .. (375.49,308.72) .. controls (376.16,309.1) and (377,308.88) .. (377.39,308.21) -- cycle ;
\draw  [color={rgb, 255:red, 0; green, 0; blue, 0 }  ,draw opacity=1 ][fill={rgb, 255:red, 255; green, 255; blue, 255 }  ,fill opacity=1 ] (395.8,289.83) .. controls (396.18,289.17) and (395.95,288.32) .. (395.29,287.94) .. controls (394.63,287.56) and (393.78,287.78) .. (393.4,288.44) .. controls (393.02,289.11) and (393.24,289.95) .. (393.91,290.34) .. controls (394.57,290.72) and (395.42,290.49) .. (395.8,289.83) -- cycle ;
\draw  [color={rgb, 255:red, 0; green, 0; blue, 0 }  ,draw opacity=1 ][fill={rgb, 255:red, 255; green, 255; blue, 255 }  ,fill opacity=1 ] (386.59,299.02) .. controls (386.97,298.36) and (386.75,297.51) .. (386.09,297.13) .. controls (385.42,296.75) and (384.58,296.97) .. (384.19,297.64) .. controls (383.81,298.3) and (384.04,299.15) .. (384.7,299.53) .. controls (385.36,299.91) and (386.21,299.68) .. (386.59,299.02) -- cycle ;

\draw    (207.12,289.13) -- (221.57,289.13) ;
\draw  [dash pattern={on 4.5pt off 4.5pt}] (133.72,267.62) .. controls (133.72,259.7) and (140.14,253.29) .. (148.06,253.29) -- (279.55,253.29) .. controls (287.47,253.29) and (293.89,259.7) .. (293.89,267.62) -- (293.89,310.63) .. controls (293.89,318.55) and (287.47,324.97) .. (279.55,324.97) -- (148.06,324.97) .. controls (140.14,324.97) and (133.72,318.55) .. (133.72,310.63) -- cycle ;

\draw    (394.6,109.44) -- (376.19,127.83) ;
\draw  [color={rgb, 255:red, 255; green, 255; blue, 255 }  ,draw opacity=1 ][fill={rgb, 255:red, 155; green, 155; blue, 155 }  ,fill opacity=0.5 ] (344.68,109.43) .. controls (344.68,92.22) and (358.63,78.27) .. (375.84,78.27) .. controls (393.05,78.27) and (407,92.22) .. (407,109.43) .. controls (407,126.64) and (393.05,140.59) .. (375.84,140.59) .. controls (358.63,140.59) and (344.68,126.64) .. (344.68,109.43) -- cycle ;
\draw  [dash pattern={on 4.5pt off 4.5pt}]  (367,90.78) -- (394.6,109.44) ;
\draw  [dash pattern={on 4.5pt off 4.5pt}]  (357.08,118.03) -- (372.38,128.08) ;
\draw  [dash pattern={on 4.5pt off 4.5pt}]  (357.08,118.03) -- (394.6,109.44) ;
\draw  [dash pattern={on 4.5pt off 4.5pt}]  (367,90.78) -- (376.19,127.83) ;
\draw  [fill={rgb, 255:red, 0; green, 0; blue, 0 }  ,fill opacity=1 ] (358.38,118.51) .. controls (358.12,119.23) and (357.32,119.6) .. (356.61,119.34) .. controls (355.89,119.07) and (355.52,118.28) .. (355.78,117.56) .. controls (356.04,116.84) and (356.83,116.47) .. (357.55,116.73) .. controls (358.27,116.99) and (358.64,117.79) .. (358.38,118.51) -- cycle ;
\draw  [fill={rgb, 255:red, 0; green, 0; blue, 0 }  ,fill opacity=1 ] (363.34,104.88) .. controls (363.08,105.6) and (362.28,105.97) .. (361.56,105.71) .. controls (360.85,105.45) and (360.48,104.65) .. (360.74,103.94) .. controls (361,103.22) and (361.79,102.85) .. (362.51,103.11) .. controls (363.23,103.37) and (363.6,104.16) .. (363.34,104.88) -- cycle ;
\draw  [fill={rgb, 255:red, 0; green, 0; blue, 0 }  ,fill opacity=1 ] (368.3,91.26) .. controls (368.04,91.98) and (367.24,92.35) .. (366.52,92.09) .. controls (365.81,91.82) and (365.43,91.03) .. (365.7,90.31) .. controls (365.96,89.59) and (366.75,89.22) .. (367.47,89.48) .. controls (368.19,89.74) and (368.56,90.54) .. (368.3,91.26) -- cycle ;

\draw    (357.08,118.03) -- (367,90.78) ;

\draw  [color={rgb, 255:red, 0; green, 0; blue, 0 }  ,draw opacity=1 ][fill={rgb, 255:red, 255; green, 255; blue, 255 }  ,fill opacity=1 ] (377.39,128.52) .. controls (377.77,127.86) and (377.54,127.01) .. (376.88,126.63) .. controls (376.22,126.24) and (375.37,126.47) .. (374.99,127.13) .. controls (374.6,127.8) and (374.83,128.64) .. (375.49,129.03) .. controls (376.16,129.41) and (377,129.18) .. (377.39,128.52) -- cycle ;
\draw  [color={rgb, 255:red, 0; green, 0; blue, 0 }  ,draw opacity=1 ][fill={rgb, 255:red, 255; green, 255; blue, 255 }  ,fill opacity=1 ] (395.8,110.13) .. controls (396.18,109.47) and (395.95,108.62) .. (395.29,108.24) .. controls (394.63,107.86) and (393.78,108.09) .. (393.4,108.75) .. controls (393.02,109.41) and (393.24,110.26) .. (393.91,110.64) .. controls (394.57,111.02) and (395.42,110.8) .. (395.8,110.13) -- cycle ;
\draw  [color={rgb, 255:red, 0; green, 0; blue, 0 }  ,draw opacity=1 ][fill={rgb, 255:red, 255; green, 255; blue, 255 }  ,fill opacity=1 ] (386.59,119.33) .. controls (386.97,118.66) and (386.75,117.82) .. (386.09,117.43) .. controls (385.42,117.05) and (384.58,117.28) .. (384.19,117.94) .. controls (383.81,118.6) and (384.04,119.45) .. (384.7,119.83) .. controls (385.36,120.22) and (386.21,119.99) .. (386.59,119.33) -- cycle ;

\draw  [color={rgb, 255:red, 255; green, 255; blue, 255 }  ,draw opacity=1 ][fill={rgb, 255:red, 155; green, 155; blue, 155 }  ,fill opacity=0.5 ] (144.8,109.43) .. controls (144.8,92.22) and (158.75,78.27) .. (175.96,78.27) .. controls (193.17,78.27) and (207.12,92.22) .. (207.12,109.43) .. controls (207.12,126.64) and (193.17,140.59) .. (175.96,140.59) .. controls (158.75,140.59) and (144.8,126.64) .. (144.8,109.43) -- cycle ;
\draw    (177.52,96.88) -- (192.02,121.99) ;
\draw    (174.41,96.88) -- (159.91,121.99) ;

\draw  [fill={rgb, 255:red, 0; green, 0; blue, 0 }  ,fill opacity=1 ] (173.21,96.33) .. controls (173.59,95.67) and (174.44,95.44) .. (175.1,95.82) .. controls (175.76,96.21) and (175.99,97.05) .. (175.61,97.72) .. controls (175.23,98.38) and (174.38,98.61) .. (173.72,98.22) .. controls (173.05,97.84) and (172.83,96.99) .. (173.21,96.33) -- cycle ;
\draw  [fill={rgb, 255:red, 0; green, 0; blue, 0 }  ,fill opacity=1 ] (165.96,108.89) .. controls (166.34,108.23) and (167.19,108) .. (167.85,108.38) .. controls (168.51,108.76) and (168.74,109.61) .. (168.36,110.27) .. controls (167.98,110.94) and (167.13,111.16) .. (166.47,110.78) .. controls (165.81,110.4) and (165.58,109.55) .. (165.96,108.89) -- cycle ;
\draw  [fill={rgb, 255:red, 0; green, 0; blue, 0 }  ,fill opacity=1 ] (158.71,121.44) .. controls (159.09,120.78) and (159.94,120.56) .. (160.6,120.94) .. controls (161.27,121.32) and (161.49,122.17) .. (161.11,122.83) .. controls (160.73,123.49) and (159.88,123.72) .. (159.22,123.34) .. controls (158.56,122.95) and (158.33,122.11) .. (158.71,121.44) -- cycle ;

\draw  [color={rgb, 255:red, 255; green, 255; blue, 255 }  ,draw opacity=1 ][fill={rgb, 255:red, 255; green, 255; blue, 255 }  ,fill opacity=1 ] (193.22,121.15) .. controls (193.6,121.81) and (193.37,122.66) .. (192.71,123.04) .. controls (192.05,123.42) and (191.2,123.2) .. (190.82,122.53) .. controls (190.43,121.87) and (190.66,121.02) .. (191.32,120.64) .. controls (191.99,120.26) and (192.83,120.49) .. (193.22,121.15) -- cycle ;
\draw  [color={rgb, 255:red, 255; green, 255; blue, 255 }  ,draw opacity=1 ][fill={rgb, 255:red, 255; green, 255; blue, 255 }  ,fill opacity=1 ] (185.97,108.59) .. controls (186.35,109.25) and (186.12,110.1) .. (185.46,110.48) .. controls (184.8,110.87) and (183.95,110.64) .. (183.57,109.98) .. controls (183.18,109.31) and (183.41,108.47) .. (184.07,108.09) .. controls (184.74,107.7) and (185.58,107.93) .. (185.97,108.59) -- cycle ;
\draw  [color={rgb, 255:red, 255; green, 255; blue, 255 }  ,draw opacity=1 ][fill={rgb, 255:red, 255; green, 255; blue, 255 }  ,fill opacity=1 ] (178.72,96.04) .. controls (179.1,96.7) and (178.87,97.54) .. (178.21,97.93) .. controls (177.55,98.31) and (176.7,98.08) .. (176.32,97.42) .. controls (175.94,96.76) and (176.16,95.91) .. (176.82,95.53) .. controls (177.49,95.15) and (178.33,95.37) .. (178.72,96.04) -- cycle ;

\draw  [color={rgb, 255:red, 255; green, 255; blue, 255 }  ,draw opacity=1 ][fill={rgb, 255:red, 155; green, 155; blue, 155 }  ,fill opacity=0.5 ] (221.57,109.43) .. controls (221.57,126.64) and (235.52,140.59) .. (252.73,140.59) .. controls (269.94,140.59) and (283.9,126.64) .. (283.9,109.43) .. controls (283.9,92.22) and (269.94,78.27) .. (252.73,78.27) .. controls (235.52,78.27) and (221.57,92.22) .. (221.57,109.43) -- cycle ;
\draw    (254.29,121.99) -- (268.79,96.88) ;
\draw    (251.18,121.99) -- (236.68,96.88) ;

\draw  [color={rgb, 255:red, 255; green, 255; blue, 255 }  ,draw opacity=1 ][fill={rgb, 255:red, 255; green, 255; blue, 255 }  ,fill opacity=1 ] (249.98,122.53) .. controls (250.36,123.2) and (251.21,123.42) .. (251.87,123.04) .. controls (252.54,122.66) and (252.76,121.81) .. (252.38,121.15) .. controls (252,120.49) and (251.15,120.26) .. (250.49,120.64) .. controls (249.83,121.02) and (249.6,121.87) .. (249.98,122.53) -- cycle ;
\draw  [color={rgb, 255:red, 255; green, 255; blue, 255 }  ,draw opacity=1 ][fill={rgb, 255:red, 255; green, 255; blue, 255 }  ,fill opacity=1 ] (242.73,109.98) .. controls (243.11,110.64) and (243.96,110.87) .. (244.62,110.48) .. controls (245.29,110.1) and (245.51,109.25) .. (245.13,108.59) .. controls (244.75,107.93) and (243.9,107.7) .. (243.24,108.09) .. controls (242.58,108.47) and (242.35,109.31) .. (242.73,109.98) -- cycle ;
\draw  [color={rgb, 255:red, 255; green, 255; blue, 255 }  ,draw opacity=1 ][fill={rgb, 255:red, 255; green, 255; blue, 255 }  ,fill opacity=1 ] (235.48,97.42) .. controls (235.86,98.08) and (236.71,98.31) .. (237.37,97.93) .. controls (238.04,97.54) and (238.26,96.7) .. (237.88,96.04) .. controls (237.5,95.37) and (236.65,95.15) .. (235.99,95.53) .. controls (235.33,95.91) and (235.1,96.76) .. (235.48,97.42) -- cycle ;
\draw  [color={rgb, 255:red, 0; green, 0; blue, 0 }  ,draw opacity=1 ][fill={rgb, 255:red, 255; green, 255; blue, 255 }  ,fill opacity=1 ] (269.99,97.57) .. controls (270.37,96.91) and (270.14,96.06) .. (269.48,95.68) .. controls (268.82,95.29) and (267.97,95.52) .. (267.59,96.18) .. controls (267.21,96.85) and (267.43,97.69) .. (268.09,98.08) .. controls (268.76,98.46) and (269.6,98.23) .. (269.99,97.57) -- cycle ;
\draw  [color={rgb, 255:red, 0; green, 0; blue, 0 }  ,draw opacity=1 ][fill={rgb, 255:red, 255; green, 255; blue, 255 }  ,fill opacity=1 ] (262.74,110.27) .. controls (263.12,109.61) and (262.89,108.76) .. (262.23,108.38) .. controls (261.57,108) and (260.72,108.23) .. (260.34,108.89) .. controls (259.96,109.55) and (260.18,110.4) .. (260.84,110.78) .. controls (261.51,111.16) and (262.35,110.94) .. (262.74,110.27) -- cycle ;
\draw  [color={rgb, 255:red, 0; green, 0; blue, 0 }  ,draw opacity=1 ][fill={rgb, 255:red, 255; green, 255; blue, 255 }  ,fill opacity=1 ] (255.49,122.83) .. controls (255.87,122.17) and (255.64,121.32) .. (254.98,120.94) .. controls (254.32,120.56) and (253.47,120.78) .. (253.09,121.44) .. controls (252.71,122.11) and (252.93,122.95) .. (253.6,123.34) .. controls (254.26,123.72) and (255.1,123.49) .. (255.49,122.83) -- cycle ;

\draw    (207.12,109.43) -- (221.57,109.43) ;

\draw  [dash pattern={on 4.5pt off 4.5pt}] (133.72,87.93) .. controls (133.72,80.01) and (140.14,73.59) .. (148.06,73.59) -- (279.55,73.59) .. controls (287.47,73.59) and (293.89,80.01) .. (293.89,87.93) -- (293.89,130.94) .. controls (293.89,138.86) and (287.47,145.28) .. (279.55,145.28) -- (148.06,145.28) .. controls (140.14,145.28) and (133.72,138.86) .. (133.72,130.94) -- cycle ;

\draw    (394.17,199.29) -- (375.76,217.67) ;
\draw  [color={rgb, 255:red, 255; green, 255; blue, 255 }  ,draw opacity=1 ][fill={rgb, 255:red, 155; green, 155; blue, 155 }  ,fill opacity=0.5 ] (344.25,199.28) .. controls (344.25,182.07) and (358.2,168.12) .. (375.41,168.12) .. controls (392.62,168.12) and (406.57,182.07) .. (406.57,199.28) .. controls (406.57,216.49) and (392.62,230.44) .. (375.41,230.44) .. controls (358.2,230.44) and (344.25,216.49) .. (344.25,199.28) -- cycle ;
\draw  [dash pattern={on 4.5pt off 4.5pt}]  (366.57,180.63) -- (394.17,199.29) ;
\draw  [dash pattern={on 4.5pt off 4.5pt}]  (356.65,207.88) -- (371.95,217.93) ;
\draw  [dash pattern={on 4.5pt off 4.5pt}]  (356.65,207.88) -- (394.17,199.29) ;
\draw  [dash pattern={on 4.5pt off 4.5pt}]  (366.57,180.63) -- (375.76,217.67) ;
\draw  [fill={rgb, 255:red, 0; green, 0; blue, 0 }  ,fill opacity=1 ] (357.95,208.35) .. controls (357.69,209.07) and (356.89,209.44) .. (356.17,209.18) .. controls (355.46,208.92) and (355.09,208.12) .. (355.35,207.41) .. controls (355.61,206.69) and (356.4,206.32) .. (357.12,206.58) .. controls (357.84,206.84) and (358.21,207.63) .. (357.95,208.35) -- cycle ;
\draw  [fill={rgb, 255:red, 0; green, 0; blue, 0 }  ,fill opacity=1 ] (362.91,194.73) .. controls (362.65,195.45) and (361.85,195.82) .. (361.13,195.56) .. controls (360.42,195.29) and (360.04,194.5) .. (360.31,193.78) .. controls (360.57,193.06) and (361.36,192.69) .. (362.08,192.95) .. controls (362.8,193.21) and (363.17,194.01) .. (362.91,194.73) -- cycle ;
\draw  [fill={rgb, 255:red, 0; green, 0; blue, 0 }  ,fill opacity=1 ] (367.87,181.1) .. controls (367.61,181.82) and (366.81,182.19) .. (366.09,181.93) .. controls (365.37,181.67) and (365,180.87) .. (365.27,180.16) .. controls (365.53,179.44) and (366.32,179.07) .. (367.04,179.33) .. controls (367.76,179.59) and (368.13,180.38) .. (367.87,181.1) -- cycle ;

\draw    (356.65,207.88) -- (366.57,180.63) ;

\draw  [color={rgb, 255:red, 0; green, 0; blue, 0 }  ,draw opacity=1 ][fill={rgb, 255:red, 255; green, 255; blue, 255 }  ,fill opacity=1 ] (376.96,218.36) .. controls (377.34,217.7) and (377.11,216.85) .. (376.45,216.47) .. controls (375.79,216.09) and (374.94,216.32) .. (374.56,216.98) .. controls (374.17,217.64) and (374.4,218.49) .. (375.06,218.87) .. controls (375.73,219.25) and (376.57,219.03) .. (376.96,218.36) -- cycle ;
\draw  [color={rgb, 255:red, 0; green, 0; blue, 0 }  ,draw opacity=1 ][fill={rgb, 255:red, 255; green, 255; blue, 255 }  ,fill opacity=1 ] (395.37,199.98) .. controls (395.75,199.32) and (395.52,198.47) .. (394.86,198.09) .. controls (394.2,197.7) and (393.35,197.93) .. (392.97,198.59) .. controls (392.59,199.26) and (392.81,200.1) .. (393.48,200.49) .. controls (394.14,200.87) and (394.98,200.64) .. (395.37,199.98) -- cycle ;
\draw  [color={rgb, 255:red, 0; green, 0; blue, 0 }  ,draw opacity=1 ][fill={rgb, 255:red, 255; green, 255; blue, 255 }  ,fill opacity=1 ] (386.16,209.17) .. controls (386.54,208.51) and (386.32,207.66) .. (385.65,207.28) .. controls (384.99,206.9) and (384.14,207.12) .. (383.76,207.79) .. controls (383.38,208.45) and (383.61,209.3) .. (384.27,209.68) .. controls (384.93,210.06) and (385.78,209.83) .. (386.16,209.17) -- cycle ;

\draw  [color={rgb, 255:red, 255; green, 255; blue, 255 }  ,draw opacity=1 ][fill={rgb, 255:red, 155; green, 155; blue, 155 }  ,fill opacity=0.5 ] (145.23,199.28) .. controls (145.23,182.07) and (159.18,168.12) .. (176.39,168.12) .. controls (193.6,168.12) and (207.56,182.07) .. (207.56,199.28) .. controls (207.56,216.49) and (193.6,230.44) .. (176.39,230.44) .. controls (159.18,230.44) and (145.23,216.49) .. (145.23,199.28) -- cycle ;
\draw    (177.95,186.72) -- (192.45,211.83) ;
\draw    (174.84,186.72) -- (160.34,211.83) ;

\draw  [fill={rgb, 255:red, 0; green, 0; blue, 0 }  ,fill opacity=1 ] (173.64,186.18) .. controls (174.02,185.51) and (174.87,185.29) .. (175.53,185.67) .. controls (176.2,186.05) and (176.42,186.9) .. (176.04,187.56) .. controls (175.66,188.22) and (174.81,188.45) .. (174.15,188.07) .. controls (173.49,187.69) and (173.26,186.84) .. (173.64,186.18) -- cycle ;
\draw  [fill={rgb, 255:red, 0; green, 0; blue, 0 }  ,fill opacity=1 ] (166.39,198.73) .. controls (166.77,198.07) and (167.62,197.84) .. (168.28,198.23) .. controls (168.95,198.61) and (169.17,199.46) .. (168.79,200.12) .. controls (168.41,200.78) and (167.56,201.01) .. (166.9,200.63) .. controls (166.24,200.24) and (166.01,199.4) .. (166.39,198.73) -- cycle ;
\draw  [fill={rgb, 255:red, 0; green, 0; blue, 0 }  ,fill opacity=1 ] (159.14,211.29) .. controls (159.52,210.63) and (160.37,210.4) .. (161.03,210.78) .. controls (161.7,211.17) and (161.92,212.01) .. (161.54,212.68) .. controls (161.16,213.34) and (160.31,213.56) .. (159.65,213.18) .. controls (158.99,212.8) and (158.76,211.95) .. (159.14,211.29) -- cycle ;

\draw  [color={rgb, 255:red, 255; green, 255; blue, 255 }  ,draw opacity=1 ][fill={rgb, 255:red, 255; green, 255; blue, 255 }  ,fill opacity=1 ] (193.65,210.99) .. controls (194.03,211.66) and (193.8,212.5) .. (193.14,212.89) .. controls (192.48,213.27) and (191.63,213.04) .. (191.25,212.38) .. controls (190.87,211.72) and (191.09,210.87) .. (191.75,210.49) .. controls (192.42,210.1) and (193.26,210.33) .. (193.65,210.99) -- cycle ;
\draw  [color={rgb, 255:red, 255; green, 255; blue, 255 }  ,draw opacity=1 ][fill={rgb, 255:red, 255; green, 255; blue, 255 }  ,fill opacity=1 ] (186.4,198.44) .. controls (186.78,199.1) and (186.55,199.95) .. (185.89,200.33) .. controls (185.23,200.71) and (184.38,200.48) .. (184,199.82) .. controls (183.62,199.16) and (183.84,198.31) .. (184.5,197.93) .. controls (185.17,197.55) and (186.01,197.77) .. (186.4,198.44) -- cycle ;
\draw  [color={rgb, 255:red, 255; green, 255; blue, 255 }  ,draw opacity=1 ][fill={rgb, 255:red, 255; green, 255; blue, 255 }  ,fill opacity=1 ] (179.15,185.88) .. controls (179.53,186.54) and (179.3,187.39) .. (178.64,187.77) .. controls (177.98,188.15) and (177.13,187.93) .. (176.75,187.27) .. controls (176.37,186.6) and (176.59,185.76) .. (177.26,185.37) .. controls (177.92,184.99) and (178.76,185.22) .. (179.15,185.88) -- cycle ;

\draw    (207.56,199.28) -- (222,199.28) ;
\draw  [dash pattern={on 4.5pt off 4.5pt}] (134.16,177.77) .. controls (134.16,169.85) and (140.57,163.44) .. (148.49,163.44) -- (279.98,163.44) .. controls (287.9,163.44) and (294.32,169.85) .. (294.32,177.77) -- (294.32,220.78) .. controls (294.32,228.7) and (287.9,235.12) .. (279.98,235.12) -- (148.49,235.12) .. controls (140.57,235.12) and (134.16,228.7) .. (134.16,220.78) -- cycle ;
\draw    (272.21,199.29) -- (253.8,217.67) ;
\draw  [color={rgb, 255:red, 255; green, 255; blue, 255 }  ,draw opacity=1 ][fill={rgb, 255:red, 155; green, 155; blue, 155 }  ,fill opacity=0.5 ] (222.29,199.28) .. controls (222.29,182.07) and (236.24,168.12) .. (253.45,168.12) .. controls (270.66,168.12) and (284.61,182.07) .. (284.61,199.28) .. controls (284.61,216.49) and (270.66,230.44) .. (253.45,230.44) .. controls (236.24,230.44) and (222.29,216.49) .. (222.29,199.28) -- cycle ;
\draw  [dash pattern={on 4.5pt off 4.5pt}]  (244.61,180.63) -- (272.21,199.29) ;
\draw  [dash pattern={on 4.5pt off 4.5pt}]  (234.69,207.88) -- (250,217.93) ;
\draw  [dash pattern={on 4.5pt off 4.5pt}]  (234.69,207.88) -- (272.21,199.29) ;
\draw  [dash pattern={on 4.5pt off 4.5pt}]  (244.61,180.63) -- (253.8,217.67) ;
\draw    (234.69,207.88) -- (244.61,180.63) ;
\draw  [color={rgb, 255:red, 0; green, 0; blue, 0 }  ,draw opacity=1 ][fill={rgb, 255:red, 255; green, 255; blue, 255 }  ,fill opacity=1 ] (255,218.36) .. controls (255.38,217.7) and (255.15,216.85) .. (254.49,216.47) .. controls (253.83,216.09) and (252.98,216.32) .. (252.6,216.98) .. controls (252.22,217.64) and (252.45,218.49) .. (253.11,218.87) .. controls (253.77,219.25) and (254.62,219.03) .. (255,218.36) -- cycle ;
\draw  [color={rgb, 255:red, 0; green, 0; blue, 0 }  ,draw opacity=1 ][fill={rgb, 255:red, 255; green, 255; blue, 255 }  ,fill opacity=1 ] (273.41,199.98) .. controls (273.79,199.32) and (273.57,198.47) .. (272.9,198.09) .. controls (272.24,197.7) and (271.4,197.93) .. (271.01,198.59) .. controls (270.63,199.26) and (270.86,200.1) .. (271.52,200.49) .. controls (272.18,200.87) and (273.03,200.64) .. (273.41,199.98) -- cycle ;
\draw  [color={rgb, 255:red, 0; green, 0; blue, 0 }  ,draw opacity=1 ][fill={rgb, 255:red, 255; green, 255; blue, 255 }  ,fill opacity=1 ] (264.21,209.17) .. controls (264.59,208.51) and (264.36,207.66) .. (263.7,207.28) .. controls (263.04,206.9) and (262.19,207.12) .. (261.81,207.79) .. controls (261.42,208.45) and (261.65,209.3) .. (262.31,209.68) .. controls (262.98,210.06) and (263.82,209.83) .. (264.21,209.17) -- cycle ;
\draw  [color={rgb, 255:red, 255; green, 255; blue, 255 }  ,draw opacity=1 ][fill={rgb, 255:red, 255; green, 255; blue, 255 }  ,fill opacity=1 ] (235.99,208.35) .. controls (235.73,209.07) and (234.94,209.44) .. (234.22,209.18) .. controls (233.5,208.92) and (233.13,208.12) .. (233.39,207.41) .. controls (233.65,206.69) and (234.45,206.32) .. (235.17,206.58) .. controls (235.89,206.84) and (236.26,207.63) .. (235.99,208.35) -- cycle ;
\draw  [color={rgb, 255:red, 255; green, 255; blue, 255 }  ,draw opacity=1 ][fill={rgb, 255:red, 255; green, 255; blue, 255 }  ,fill opacity=1 ] (240.95,194.73) .. controls (240.69,195.45) and (239.9,195.82) .. (239.18,195.56) .. controls (238.46,195.29) and (238.09,194.5) .. (238.35,193.78) .. controls (238.61,193.06) and (239.41,192.69) .. (240.13,192.95) .. controls (240.84,193.21) and (241.21,194.01) .. (240.95,194.73) -- cycle ;
\draw  [color={rgb, 255:red, 255; green, 255; blue, 255 }  ,draw opacity=1 ][fill={rgb, 255:red, 255; green, 255; blue, 255 }  ,fill opacity=1 ] (245.91,181.1) .. controls (245.65,181.82) and (244.86,182.19) .. (244.14,181.93) .. controls (243.42,181.67) and (243.05,180.87) .. (243.31,180.16) .. controls (243.57,179.44) and (244.37,179.07) .. (245.08,179.33) .. controls (245.8,179.59) and (246.17,180.38) .. (245.91,181.1) -- cycle ;

\end{tikzpicture}